\documentclass[a4paper,english,hyperref]{article}

\usepackage{amsmath}
\usepackage{amsfonts}
\usepackage{amssymb}
\usepackage[amsmath,amsthm,thmmarks,hyperref]{ntheorem}
\usepackage{graphicx}
\usepackage{color}%
\usepackage{enumitem}
\usepackage{longtable}
\usepackage[hidelinks]{hyperref}
\usepackage[capitalize]{cleveref}
\newtheorem{theorem}{Theorem}

\newtheorem{proposition}[theorem]{Proposition}
\newtheorem{lemma}[theorem]{Lemma}
\newtheorem{corollary}[theorem]{Corollary}

\newtheorem{remark}[theorem]{Remark}
\newtheorem{definition}[theorem]{Definitioon}

\newtheorem{notation}[theorem]{Notation}

\newtheorem{assumption}[theorem]{Assumption}

\newcommand{\BIGOP}[1]{\mathop{\mathchoice{\raise-0.22em\hbox{\huge
$#1$}} {\raise-0.05em\hbox{\Large $#1$}}{\hbox{\large $#1$}}{#1}}}
\usepackage{etoolbox}
\usepackage{mathtools}
\newcommand\swapifbranches[3]{#1{#3}{#2}}
\MHInternalSyntaxOn{}
\patchcmd{\DeclarePairedDelimiter}{\@ifstar}{\swapifbranches\@ifstar}{}{}
\MHInternalSyntaxOff{}

\DeclarePairedDelimiter\abs{\lvert}{\rvert}
\DeclarePairedDelimiter\normDelimiters{\lVert}{\rVert}

\DeclarePairedDelimiter\bilinearDelimiters{\langle}{\rangle}
\DeclarePairedDelimiter\jumpPair{[}{]} 
\newcommand{\delimitMean}[1]{\ensuremath{\left\{\mkern-4.5mu\left\{ #1 \right\}\mkern-4.5mu\right\}}} 
\newcommand{\norm}[2]{\ensuremath{\normDelimiters{#1}_{#2}}}

\newcommand{\sesquilinear}[3]{\ensuremath{\bilinearDelimiters{#1,\overline{#2}}_{#3}}}

\newcommand{\jump}[3]{\ensuremath{\jumpPair{#1}_{\operatorname{#2}#3}}}
\newcommand{\mean}[3]{\ensuremath{\delimitMean{#1}_{\operatorname{#2}#3}}} 
\newcommand{\modified}[1]{#1}

\title{Skeleton Integral Equations for Acoustic Transmission Problems with Varying Coefficients}

\author{F. Florian\thanks{Institut f\"{u}r Mathematik,
Universit\"{a}t Z\"{u}rich, Winterthurerstr 190, CH-8057 Z\"{u}rich,
Switzerland (\url{francesco.florian@uzh.ch})}
\and R. Hiptmair\thanks{Seminar for Applied
Mathematics, ETH Z\"{u}rich, Z\"{u}rich, Switzerland (\url{ralf.hiptmair@sam.math.ethz.ch})}
\and S.A. Sauter\thanks{Institut f\"{u}r Mathematik,
Universit\"{a}t Z\"{u}rich, Winterthurerstr 190, CH-8057 Z\"{u}rich,
Switzerland (\url{stas@math.uzh.ch})}}

\usepackage{amsopn}

\hypersetup{
  pdftitle={Skeleton Integral Equations for Acoustic Transmission Problems with Varying Coefficients},
  pdfauthor={F. Florian, R. Hiptmair and S.A. Sauter}
}

\newcommand{\rhc}[1]{#1}

\begin{document}

\maketitle

\begin{abstract}
  In this paper we will derive an \rhc{non-local (``integral'')} equation which transforms
  a three-dimensional acoustic transmission problem with \emph{variable} coefficients,
  non-zero absorption, and mixed boundary conditions to a non-local equation on a
  ``skeleton'' of the domain $\Omega\subset\mathbb{R}^{3}$, where ``skeleton'' stands for
  the union of the interfaces and boundaries of a Lipschitz partition of $\Omega$. To that
  end, we introduce and analyze abstract layer potentials as solutions of auxiliary
  coercive full space variational problems and derive jump conditions across domain
  interfaces. This allows us to formulate the non-local skeleton equation as a
  \emph{direct method} for the unknown Cauchy data \rhc{of the solution} of the original
  partial differential equation. We establish coercivity and continuity of the variational
  form of the skeleton { equation based} on auxiliary full space variational
  problems. Explicit \rhc{expressions for} Green's functions is not required and \rhc{all}
  our estimates are \emph{explicit} in the complex wave number.
\end{abstract}


\section{Introduction}

\textbf{Setting.} In this paper we consider acoustic transmission problems in Laplace domain.
\begin{equation}
  \label{eq:1}
  -\operatorname{div}\left(\mathbb{A}\nabla w\right)  +s^{2}p\,w = 0
  \quad\text{in }\Omega \subset \mathbb{R}^{3}.
\end{equation}
We admit general essentially bounded and uniformly positive (definite) coefficient
functions $\mathbb{A}$ and $p$ and mixed boundary conditions. More precisely, the boundary
conditions on $\partial\Omega$ are of Dirichlet and/or Neumann type and decay conditions
are imposed at infinity if the domain is unbounded. We assume that the (complex) wave
number $s$ has positive real part, \rhc{$\operatorname{Re}s>0$}, so that the arising
sesquilinear form in the variational formulation is coercive and well-posedness in
$H^{1}(\mathbb{R}^{3})$ follows by the Lax-Milgram lemma. More details are given in
Section~\ref{SubSecDiffOp} and the following.

{\textbf{Goal.} There exist many approaches to transform continuous and
  coercive acoustic transmission problems to a non-local equation on a \emph{domain
    skeleton} (interfaces \rhc{of a Lipschitz partition} and the domain boundary); among
  them are the \emph{direct }and \emph{indirect formulation}, equations of \emph{first
  }and \emph{second kind, and symmetric} and \emph{non-symmetric couplings} for interface
  problems. However, not all of them lead to well-posed skeleton equations. The goal of
  this paper is to develop a transformation strategy for acoustic transmission problems
  with mixed boundary condition such that well-posedness will always be inherited from the
  well-posedness of the \rhc{boundary value problem}. This transformation will be based on
  a direct formulation by Green's representation formula to express the homogeneous
  solutions in the subdomains via their Cauchy data, i.e., by their traces and co-normal
  traces on the subdomain boundaries.  Green's formula is typically based on explicit
  expressions for the fundamental solution for the differential operator and boils down to
  a linear combination of the single layer and double layer potential. However, the
  explicit expressions for fundamental solutions are known only for very special
  configurations which include full and half space problems and constant coefficients. A
  semi-explicit representation by Sommerfeld-type integrals exist for half space problems
  with impedance boundary conditions or full space problems for layered media and
  piecewise constant coefficients \rhc{\cite{hoernig2010green,Hein_Nedelec_green}} while
  for more general domains and varying coefficients the explicit form of the fundamental
  solution is unknown.

  The main goals of this paper are
  \begin{enumerate}[label={\alph{*})}]
  \item to represent the solution of a homogeneous acoustic PDE with very general
    coefficients as a linear combination of generalized layer potentials acting on the
    Cauchy traces of the solution and we will present an appropriate definition of these
    layer potentials, 
  \item to derive a \emph{non-local single-trace skeleton equation} for the unknown Cauchy
    data such that well-posedness is inherited from the PDE, and
  \item \rhc{to provide fully $s$-explicit stability estimates for that new equation; we will elaborate
    the dependence on $s$ in all estimates of the arising operators and solutions.}
  \end{enumerate}

  In this way, the question of finding a representation for a fundamental solution is
  decoupled from the transformation method: once a fundamental solution or an
  approximation to it is available it can be used for an integral representation of the
  relevant layer potentials resulting in a stable non-local single-trace skeleton
  \emph{integral equation}.}
\medskip

\textbf{Main contributions.} Usually, Green's representation formula contains the
fundamental solution of the underlying PDE explicitly and, hence, in literature the
arising boundary integral equations are usually considered for cases where the fundamental
solution is known explicitly. Our approach to transforming the PDE to a non-local equation
on the skeleton does not rely on fundamental solutions; neither their existence nor an
explicit form is required. Instead, we define the layer potentials directly via the
variational form of the PDE as solutions of appropriate variational problems. We derive
jump relations for these abstract potentials, Green's representation formula, and
non-local skeleton operators which allow us to define the Calder\'{o}n operator. We show
how coercivity of the sesquilinear form on the skeleton can be derived directly from the
coercivity of the PDE.

While the definition of the single layer potential as the solution of the variational form
of the full space PDE for certain types of right-hand side is standard and applies also to
elliptic PDEs with variable coefficients, the definition of the double layer potential is
more delicate. Various (equivalent) definitions exist in literature for certain types
of elliptic PDEs and we briefly review some of them:

\begin{enumerate}[label={\arabic{*})}]
\item If the fundamental solution, say $G\left(  \mathbf{x},\mathbf{y}\right)  $,
of the differential operator is known the double layer potential can be
defined as an integral over the skeleton of the co-normal derivative of $G$
convoluted with a boundary density -- first for sufficiently regular boundary
functions and then by continuous extension as a mapping between appropriate
Sobolev spaces. The analysis of the double layer potential (mapping
properties/jump relations, etc.) is then derived from properties of the
fundamental solution. However, if the fundamental solution is not known
explicitly as, e.g., for variable $L^{\infty}$ coefficients the analysis is
far from trivial.

\item For problems with constant coefficients the double layer potential can be
defined as the composition of the full space solution operator (acoustic
Newton potential) with the dual of the co-normal derivative. However, this
dual co-normal derivative maps into a space which is larger than the natural
domain of the Newton potential. For PDEs with constant coefficients this
problem can be solved since it is known that the Newton potential satisfies
some regularity shift properties. For variable $L^{\infty}$ coefficients this
is a subtle issue.

\item
 In \cite{CostabelElemRes}, the case of $C^{\infty}$- coefficients is
  considered. First the double layer potential is introduced as explained in 1); then a
  regularity shift theorem from \cite{necas} is employed to directly derive a Green's
  representation formula. This Green's formula can then be used as an alternative
  definition of the double layer potential.

\item The definition in \cite[(4.5)]{Barton_potential} expresses the double layer
  potential as a composition of a trace lifting of the boundary density with the
  differential operator and the Newton potential and thus avoids both, the explicit
  knowledge of the fundamental solution and the range space of the dual co-normal
  derivative. Although the analysis of the double layer potential can be based on the
  mature theory of elliptic PDEs, it seems that our new definition allows for a much more
  straightforward analysis.
\end{enumerate}

Our \textbf{new approach} defines the double layer potential as the solution of an
\emph{ultra-weak variational formulation} of the full space PDE with a certain type of right-hand
sides. This definition allows us to derive directly the mapping properties, jump
relations, and representation formula from the underlying PDE. \medskip

{We derive the skeleton and Calder\'{o}n operators from this idea.
The formulation of acoustic transmission problems with \textit{constant}
coefficients in each subdomain and mixed boundary condition as skeleton
equations is a topic of active research in numerical analysis and we mention
the approaches via the \textit{multi-trace formulation }(see,
\cite{Hiptmair_multiple_trace}, \cite{ClHiptJH}), the \textit{single-trace
formulation} \cite{Cost_Steph_trans}, \cite{vonPetersdorff89},
\cite{Hipt_Spindler}, \cite{Natroshvili_BEM_variable},
\cite{LabarcaHiptmair2023}, and the related PMCHWT method in electromagnetics
(see \cite{harrington1989boundary}, \cite{PoggioMiller}). Our approach with
emphasis on transmission and mixed boundary conditions can be regarded as a
generalization of the recent paper \cite{EbFlHiSa} from the piecewise
constant coefficient case to general $L^{\infty}$ coefficients. Another
approach for problems with variable coefficients is based on the use of a
parametrix instead of the unknown Green's function and presented in
\cite{Natroshvili_BEM_variable}.
We also generalize \cite{ClHiptJH} by allowing for unbounded domains (full space/half space), variable coefficients in the
  subdomains, and do not require the explicit knowledge of a Green's function. We also
  generalize the stability theory for the Calder\'{o}n operator developed in
  \cite{banjai_coupling} (see also the monograph \cite{sayas_book}) to variable
  coefficients in the principal and zeroth order part of the differential equation. The
  estimates for the layer potentials, Calder\'{o}n operators, and skeleton operators are
  explicit with respect to the wave number ${s}$ and generalize the known estimates for
  problems with piecewise constant coefficients (see, e.g., \cite{bambduong},
  \cite{LaSa2009}, \cite{EbFlHiSa}).}

\rhc{%
\begin{remark}
  We emphasize that this work is meant to be a contribution to the theory of partial
  differential equations. Nevertheless, our new non-local single-trace skeleton equations
  may be the foundation of numerical methods, but this will require the representations of
  the layer potentials as integral operators acting on trace on the skeleton. This
  representation might be explicit, semi-explicit via Fourier- or Hankel transforms, might
  be given by an asymptotic series or, alternatively, a parametrix can be employed.
\end{remark}}

\textbf{Outline.} The paper is structured as follows. In Section \ref{SecSetting} we
formulate the acoustic transmission problem with mixed boundary conditions. This requires
the introduction of a domain partitioning and its skeleton, the definition of one-sided
trace operators as well as the jumps and means of piecewise regular functions. The
transmission problem is formulated in (\ref{generalTP}) and defines the starting point for
the various steps in the derivation of the non-local skeleton equations.

In Section \ref{SecGRF}, we derive Green's representation formula in an abstract way. We
consider the homogeneous PDE on a subdomain as well as on its complement domain in
$\mathbb{R}^{3}$ (with extended coefficients) and formulate auxiliary variational full
space problems which are coercive and continuous. The single layer potential is defined as
the solution operator for a distribution (density) located on the interface (see
(\ref{elljsS})); the explicit knowledge of a fundamental solution is not required. We
present a new and simple definition of the double layer potential as the solution of an
ultra-weak variational full space problem for a certain type of right-hand sides. With
these layer potentials at hand we prove a Green's representation formula on both
subdomains (Lemma \ref{LemGreen}) as well as jump relations for both layer potentials.

Section \ref{SecCaldOp} is devoted to the definition of the non-local skeleton operators
$\mathsf{V}$, $\mathsf{K}$, $\mathsf{K}^{\prime}$, $\mathsf{W}$ which are used to build
the Calder\'{o}n operator. The important projection property for the Calder\'{o}n operator
is derived in Lemma \ref{LemCaldProj}.

In Section \ref{SecMultiSingleTraceForm} we define the free single trace space
$\mathbb{X}^{\operatorname{single}}$ on the skeleton and the one with incorporated
boundary conditions $\mathbb{X}_{0}^{\operatorname{single}}$.  Then, the non-local
skeleton equation is formulated in (\ref{singskeleq}) as a variational problem with energy
space $\mathbb{X}_{0}^{\operatorname{single}}%
$. The remaining part of this section is devoted to the analysis of the skeleton equation
and leads to its well-posedness, formulated in Theorem \ref{ThmWellPosedSkeletonEq}.

We summarize our main achievement in the concluding Section \ref{SecConclusion} and give
comments on some straightforward extensions of this integral equation method.

In Appendix \ref{AppProof} we give the proof of $s$-explicit coercivity and
continuity estimates for the boundary integral operators and layer potentials. Since the
arguments are very similar to those in \cite[Prop.  16, 19]{LaSa2009} and
\cite[Lem. 5.2]{Barton_potential} we have shifted this proof to the appendix.

{A list of notation is assembled at the end of the paper.}

\section{Setting\label{SecSetting}}

In this section we give details about the acoustic transmission problem. First, we
introduce the appropriate Sobolev spaces, standard trace operators, and co-normal
derivatives. Then we specify assumptions on the coefficients of the problem and formulate
boundary and decay conditions. We write
$\mathbb{R}_{>0}:=\left\{ x\in\mathbb{R}\mid x>0\right\} $, and
$\mathbb{C}_{>0}:=\left\{ z\in\mathbb{C}\mid\operatorname{Re} z>0\right\} $, respectively.

\subsection{Function spaces\label{SubSecFuSp}}

Let $\omega\subset\mathbb{R}^{3}$ be a bounded or unbounded Lipschitz domain
with (possibly empty) boundary $\partial\omega$.
{Let $L^{p}\left(  \omega\right)  $, $1\leq p\leq\infty$, be the usual Lebesgue
spaces with norm $\left\Vert \cdot\right\Vert _{L^{p}\left(  \omega\right)  }%
$. For $k\geq0$, the classical Sobolev space $H^{k}\left(  \omega\right)  $
consists of all functions whose $k$-th weak derivatives are square-integrable;
its norm is denoted by $\left\Vert \cdot\right\Vert _{H^{k}\left(
\omega\right)  }$.}
For $k\geq0$, we denote by $H_{0}^{k}\left(
\omega\right)  $ the closure of the space of infinitely smooth functions with
compact support in $\omega$ with respect to the $H^{k}\left(  \omega\right)  $
norm. Its dual space is denoted by $H^{-k}\left(  \omega\right)  :=\left(
H_{0}^{k}\left(  \omega\right)  \right)  ^{\prime}$. Vector- and tensor valued
versions of the Lebesgue spaces are denoted by $\mathbf{L}^{p}\left(
\omega\right)  :=L^{p}\left(  \omega\right)  ^{3}$ and $\mathbb{L}^{p}\left(
\omega\right)  :=L^{p}\left(  \omega\right)  ^{3\times3}$ with
norm $\left\Vert \cdot\right\Vert _{\mathbf{L}^{p}\left(  \omega\right)  }$
and $\left\Vert \cdot\right\Vert _{\mathbb{L}^{p}\left(  \omega\right)  }$,
respectively and we use an analogous notation for vector and tensor valued
Sobolev spaces.
For $p=2$, these spaces are Hilbert spaces with scalar product $\left(
\cdot,\cdot\right)  _{L^{2}\left(  \omega\right)  }$, $\left(  \cdot
,\cdot\right)  _{\mathbf{L}^{2}\left(  \omega\right)  }$, $\left(  \cdot
,\cdot\right)  _{\mathbb{L}^{2}\left(  \omega\right)  }$.
We also employ a \textquotedblleft frequency-dependent\textquotedblright%
\ $H^{1}\left(  \omega\right)  $ norm and define for $s\in\mathbb{C}%
\backslash\left\{  0\right\}  $%
\begin{equation}
\left\Vert v\right\Vert _{H^{1}\left(  \omega\right)  ;s}:=\left(  \left\Vert
\nabla v\right\Vert _{\mathbf{L}^{2}\left(  \omega\right)  }^{2}+\left\vert
s\right\vert ^{2}\left\Vert v\right\Vert _{L^{2}\left(  \omega\right)  }%
^{2}\right)  ^{1/2}. \label{fs_norm}%
\end{equation}
{Note that the usual norm $\left\Vert
\cdot\right\Vert _{H^{1}\left(  \omega\right)  }$ coincides with $\left\Vert
\cdot\right\Vert _{H^{1}\left(  \omega\right)  ;s}$ for $|s|=1$.}
The space $\mathbf{H}\left(  \omega,\operatorname*{div}\right)  $ is given by%
\begin{equation}
\mathbf{H}\left(  \omega,\operatorname*{div}\right)  :=\left\{  \mathbf{w}%
\in\mathbf{L}^{2}\left(  \Omega\right)  \mid\operatorname*{div}\mathbf{w}\in
L^{2}\left(  \omega\right)  \right\}  . \label{defHomegadiv}%
\end{equation}

On the boundary of $\omega$, we define the Sobolev space $H^{\alpha}\left(
\partial\omega\right)  $, $\alpha\geq0$, in the usual way (see, e.g.,
\cite[pp. 98]{Mclean00}). Note that the range of $\alpha$ for which
$H^{\alpha}\left(  \partial\omega\right)  $ is defined may be limited,
depending on the global smoothness of the surface $\partial\omega$; for
Lipschitz surfaces, $\alpha$ can be chosen in the range $\left[  0,1\right]
$; for $\alpha<0$, the space $H^{\alpha}\left(  \partial\omega\right)  $ is
the dual of $H^{-\alpha}\left(  \partial\omega\right)  $.
{Revall that the Sobolev space $H^{1/2}\left(
\partial\omega\right)  $ is equipped with the Sobolev-Slobodeckij norm%
\[
\left\Vert v\right\Vert _{H^{1/2}\left(  \partial\omega\right)  }:=\left(
\left\Vert v\right\Vert _{L^{2}\left(  \partial\omega\right)  }^{2}%
+\int_{\partial\omega\times\partial\omega}\frac{\left\vert v\left(
\mathbf{x}\right)  -v\left(  \mathbf{y}\right)  \right\vert ^{2}}{\left\Vert
\mathbf{x}-\mathbf{y}\right\Vert ^{3}}d\Gamma_{\mathbf{x}}d\Gamma_{\mathbf{y}%
}\right)  ^{1/2}.
\]}

We write $\left\langle \cdot,\cdot\right\rangle _{\omega}$ for the
\textit{bilinear form}%
\begin{equation}
  \label{eq:2}
  \left\langle u,v\right\rangle _{\omega}:=\int_{\omega}uv\quad\text{so
    that\quad}\left(  u,v\right)  _{L^{2}\left(  \omega\right)  }=\left\langle
    u,\overline{v}\right\rangle _{\omega},
\end{equation}
and identify $\left\langle \cdot,\cdot\right\rangle _{\omega}$ with its
continuous extension to the duality pairing $H^{-k}\left(  \omega\right)
\times H_{0}^{k}\left(  \omega\right)  $. For $k\geq0$, the spaces
$H_{\operatorname*{loc}}^{k}\left(  \omega\right)  $ are defined based on smooth
and compactly-supported cutoff functions via%
\begin{equation}
  \label{eq:3}
  H_{\operatorname*{loc}}^{k}\left(  \omega\right)  :=\left\{  v:\chi v\in
    H^{k}\left(  \omega\right)  \text{ for all }\chi\in C_{0}^{\infty}\left(
      \mathbb{R}^{3}\right)  \right\}
\end{equation}
and the subscript \textquotedblleft$\operatorname*{loc}$\textquotedblright\ is
used in an analogue way also for other spaces.

Let $\mathbb{R}_{\operatorname*{sym}}^{3\times3}$ denote the set of real
symmetric $3\times3$ matrices. We denote by $\left\langle \cdot,\cdot
\right\rangle :\mathbb{C}^{3}\times\mathbb{C}^{3}\rightarrow\mathbb{C}$ the
bilinear form $\left\langle \mathbf{a},\mathbf{b}\right\rangle :=\sum_{\ell
=1}^{3}a_{\ell}b_{\ell}$ for $\mathbf{a}=\left(  a_{\ell}\right)  _{\ell
=1}^{3}\in\mathbb{C}^{3}$ and $\mathbf{b}=\left(  b_{\ell}\right)  _{\ell
=1}^{3}\in\mathbb{C}^{3}$. Clearly, this bilinear form is the standard
Euclidean scalar product if restricted to $\mathbb{R}^{3}\times\mathbb{R}^{3}%
$. Let $\mathbb{L}^{\infty}\left(  \omega,\mathbb{R}_{\operatorname*{sym}%
}^{3\times3}\right)  $ denote the space of all functions $\mathbb{B}%
:\omega\rightarrow\mathbb{R}_{\operatorname*{sym}}^{3\times3}$ whose
components belong to the Lebesgue space $L^{\infty}\left(  \omega\right)  $.
We define the \textit{spectral bounds} for $\mathbb{B}\in\mathbb{L}^{\infty
}\left(  \omega,\mathbb{R}_{\operatorname*{sym}}^{3\times3}\right)  $ and
$q\in L^{\infty}\left(  \omega,\mathbb{R}\right)  $ by%
\begin{subequations}
\label{spectralbounds}
\begin{gather}
  \lambda\left(  \mathbb{B}\right)   :=\underset{\mathbf{y}\in\omega
  }{\operatorname*{ess}\inf}%
  \inf_{\mathbf{v}\in\mathbb{R}^{3}\backslash\left\{  0\right\}  }%
  \frac{\left\langle \mathbb{B}\left(  \mathbf{y}\right)  \mathbf{v}%
      ,\mathbf{v}\right\rangle }{\left\langle \mathbf{v},\mathbf{v}\right\rangle
  }\leq\underset{\mathbf{y}\in\omega}{\operatorname*{ess}\sup}\sup
  _{\mathbf{v}\in\mathbb{R}^{3}\backslash\left\{  0\right\}  }\frac{\left\langle
      \mathbb{B}\left(  \mathbf{y}\right)  \mathbf{v},\mathbf{v}\right\rangle
  }{\left\langle \mathbf{v},\mathbf{v}\right\rangle }=:\Lambda\left(
    \mathbb{B}\right)  <\infty, \\
  \lambda\left(  q\right)   :=\underset{\mathbf{y}\in\omega}%
  {\operatorname*{ess}\inf}\;q\left(  \mathbf{y}\right)  \leq\underset
  {\mathbf{y}\in\omega}{\operatorname*{ess}\sup}\;q\left(  \mathbf{y}\right)
  =:\Lambda\left(  q\right)  <\infty.
\end{gather}
\end{subequations}%

\begin{definition}
\label{DefH1omegaB}Let%
\begin{align*}
L_{>0}^{\infty}\left(  \omega,\mathbb{R}\right)   &  :=\left\{  q\in
L^{\infty}\left(  \omega,\mathbb{R}\right)  \mid\lambda\left(  q\right)
>0\right\}  ,\\
\mathbb{L}_{>0}^{\infty}\left(  \omega,\mathbb{R}_{\operatorname*{sym}%
}^{3\times3}\right)   &  :=\left\{  \mathbb{B}\in\mathbb{L}^{\infty}\left(
\omega,\mathbb{R}_{\operatorname*{sym}}^{3\times3}\right)  \mid\lambda\left(
\mathbb{B}\right)  >0\right\}  .
\end{align*}
For $\mathbb{B}\in\mathbb{L}_{>0}^{\infty}\left(  \omega,\mathbb{R}%
_{\operatorname*{sym}}^{3\times3}\right)  $, the space $H^{1}\left(
\omega,\mathbb{B}\right)  $ is given by%
\[
H^{1}\left(  \omega,\mathbb{B}\right)  :=\left\{  u\in H^{1}\left(
\omega\right)  \mid\operatorname{div}\left(  \mathbb{B}\nabla u\right)  \in
L^{2}\left(  \omega\right)  \right\}
\]
and equipped with the graph norm%
\[
\left\Vert u\right\Vert _{H^{1}\left(  \omega,\mathbb{B}\right)  }:=\left(
\left\Vert u\right\Vert _{H^{1}\left(  \omega\right)  }^{2}+\left\Vert
\operatorname{div}\left(  \mathbb{B}\nabla u\right)  \right\Vert
_{L^{2}\left(  \omega\right)  }^{2}\right)  ^{1/2}.zzzz
\]

\end{definition}

\subsection{Differential operators\label{SubSecDiffOp}}

Next we describe our assumptions on the computational domain and its
partition. Let $\Omega\subset\mathbb{R}^{3}$ be a bounded or unbounded
Lipschitz domain with (possibly empty) boundary $\Gamma:=\partial\Omega$. We
assume that there is a finite partition of $\Omega$ consisting of disjoint
Lipschitz domains $\Omega_{j}$, $1\leq j\leq n_{\Omega}$, {with boundaries}
$\Gamma_{j}:=\partial\Omega_{j}$, which satisfy $\overline{\Omega}=%
{\textstyle\bigcup\nolimits_{j=1}^{n_{\Omega}}}
\overline{\Omega_{j}}$. The subdomains are collected in the partition
$\mathcal{P}_{\Omega}=\left\{  \Omega_{j}:1\leq j\leq n_{\Omega}\right\}  $.
The intersection of the boundaries $\partial\Omega_{j}$ and $\partial
\Omega_{k}$ is denoted by $\Gamma_{j,k}:=\partial\Omega_{j}\cap\partial
\Omega_{k}$. The \textit{skeleton} of this partition is given by $\Sigma:=%
{\displaystyle\bigcup\limits_{j=1}^{n_{\Omega}}}
\partial\Omega_{j}$. To unify notation, we write $\Omega_{j}^{-}:=\Omega_{j}$
and set $\Omega_{j}^{+}:=\mathbb{R}^{3}\backslash\overline{\Omega_{j}^{-}}$.

We consider mixed Dirichlet and Neumann boundary conditions on $\partial
\Omega$. In this way, we split
\begin{equation}
\partial\Omega=\Gamma_{\operatorname*{D}}\cup\Gamma_{\operatorname*{N}}
\label{GammaDGammaN}%
\end{equation}
and assume the relative interiors of these subsets are disjoint.

In the subdomains $\Omega_{j}\in\mathcal{P}_{\Omega}$, we consider partial
differential equations and formulate appropriate assumptions on the
coefficients next.

\begin{assumption}
\label{Acoeff}For any $1\leq j\leq n_{\Omega}$ we are given coefficients that satisfy
\begin{enumerate}
\item $\mathbb{A}_{j}^{-}\in\mathbb{L}_{>0}^{\infty}\left(  \Omega
_{j},\mathbb{R}_{\operatorname*{sym}}^{3\times3}\right)  $ and $\mathbb{A}%
_{j}^{-}$ can be extended to some $\mathbb{A}_{j}^{\operatorname*{ext}}%
\in\mathbb{L}_{>0}^{\infty}\left(  \mathbb{R}^{3},\mathbb{R}%
_{\operatorname*{sym}}^{3\times3}\right)  $,

\item $p_{j}^{-}\in L_{>0}^{\infty}\left(  \Omega_{j},\mathbb{R}\right)  $ and
$p_{j}^{-}$ can be extended to some $p_{j}^{\operatorname*{ext}}\in
L_{>0}^{\infty}\left(  \mathbb{R}^{3},\mathbb{R}\right)  $,

\item $s\in\mathbb{C}_{>0}$ and $\left\vert s\right\vert \geq s_{0}$ for some
$s_{0}>0$.
\end{enumerate}
\end{assumption}

We exclude a neighborhood of $0$ for the frequencies $s\in\mathbb{C}$ since our focus is
on the high-frequency behavior. Note that the constants in our estimates depend
continuously on $s_{0}$ and, possibly, deteriorate as $s_{0}\rightarrow0$.

For $\sigma\in\left\{ +,-\right\} $, we formally define the differential operators:%
\begin{equation}
\mathsf{L}_{j}^{\sigma}\left(  s\right)  w:=-\operatorname{div}\left(
\mathbb{A}_{j}^{\sigma}\nabla w\right)  +s^{2}p_{j}^{\sigma}w\quad\text{in
}\Omega_{j}^{\sigma}, \label{defDiffOpsigma}%
\end{equation}
where%
\begin{equation}
\mathbb{A}_{j}^{\sigma}:=\left.  \mathbb{A}_{j}^{\operatorname*{ext}%
}\right\vert _{\Omega_{j}^{\sigma}}\quad\text{and\quad}p_{j}^{\sigma}:=\left.
p_{j}^{\operatorname*{ext}}\right\vert _{\Omega_{j}^{\sigma}}\text{\quad
}\sigma\in\left\{  +,-\right\}  . \label{defajpjpm}%
\end{equation}
The differential equation on the subdomain $\Omega_{j}$ is given by%
\begin{equation}
\mathsf{L}_{j}^{-}\left(  s\right)  u_{j}=0\qquad\text{in }\Omega_{j}.
\label{homPDEGreen1}%
\end{equation}

\begin{remark}
  Time harmonic wave propagation \emph{with absorption} can be described in the simplest
  case by a Helmholtz equation with wave number (frequency parameter) $s$ of positive real
  part. Such problems arise in many applications such as, e.g., in viscoelastodynamics for
  materials with damping (see, e.g., \cite{Achenbach2005}), in electromagnetism for wave
  propagation in lossy media (see, e.g., \cite{Jackson02engl}) and in nonlinear optics
  (see, e.g., \cite{sauter1996nonlinear}). The Helmholtz equation for complex wave numbers
  also arises within the popular convolution quadrature method for solving time depending
  wave propagation problems and within some iterative algorithms for solving the linear
  system for the Helmholtz equation (see, e.g., \cite[\S 2]{brm2019variable} for a more
  detailed description of applications).
\end{remark}

\begin{remark}
  \label{RemAgiven}
  Typically, the coefficients $\mathbb{A}_{j}^{-}$, $p_{j}^{-}$ are the restrictions of
  some given global coefficients $\mathbb{A}%
  \in\mathbb{L}_{>0}^{\infty}\left( \mathbb{R}^{3},\mathbb{R}%
    _{\operatorname*{sym}}^{3\times3}\right) $,
  $p\in L_{>0}^{\infty}\left( \mathbb{R}^{3},\mathbb{R}\right) $. Then, the choice
  $\mathbb{A}%
  _{j}^{\operatorname*{ext}}:=\mathbb{A}$ is admissible and seems to be natural.  In some
  practical applications, a different choice might be \textquotedblleft
  simpler\textquotedblright\ and preferable. For instance, if the global coefficient
  $\mathbb{A}$ is constant on the subdomains $\Omega_{j}$ and given by a positive definite
  matrix $\mathbb{A}_{j}^{-}\in\mathbb{R}%
  _{\operatorname*{sym}}^{3\times3}$ and $p_{j}^{-}$ is also constant, then, the choice of
  $\mathbb{A}_{j}^{\operatorname*{ext}}$ and $p_{j}%
  ^{\operatorname*{ext}}$ as the constant extensions of $\mathbb{A}_{j}^{-}$, $p_{j}^{-}$
  are preferable since the Green's function is explicitly known in these cases (see, e.g.,
  \cite[(3.1.3)]{SauterSchwab2010}). However, in our abstract setting the existence or
  explicit knowledge of the Green's function is not needed and, hence, the concrete choice
  of the extension is irrelevant.
  \rhc{Of course,} the single layer and double layer operators will depend on the chosen
  extension; however the key point is that their combination in a Green's representation
  formula always represents a homogeneous solution in the corresponding subdomain as will
  be shown in Lemma \ref{LemGreen}.
\end{remark}

\subsection{Traces and jumps}

Next, we introduce \textit{jumps} and \textit{means }of functions across the
boundaries $\Gamma_{j}$; the index $j$ indicates that the two-dimensional
manifold $\Gamma_{j}$ is regarded from the domain $\Omega_{j}$.

The following trace operators along their properties are well known for
domains with compact boundary (see, e.g., \cite[Thm. 3.37, 3.38, Lem. 4.3,
Thm. 4.4]{Mclean00}, \cite[Thm. 2.5]{Girault86}). For domains with non-compact
boundary we refer to \cite[Thm. 2.3, Cor. 3.14, Lem. 2.6]{Mikhailov_traces_1}.
We define the one-sided co-normal derivatives for an abstract diffusion
coefficient $\mathbb{B}\in\mathbb{L}_{>0}^{\infty}\left(  \mathbb{R}%
^{3},\mathbb{R}_{\operatorname*{sym}}^{3\times3}\right)  $; in our
applications, this will be either $\mathbb{A}$ or $\mathbb{A}_{j}%
^{\operatorname*{ext}}$.

\begin{proposition}
\label{Proptrace}Let $\Omega$, $\Omega_{j}$, $\Omega_{j}^{\sigma}$, $1\leq
j\leq n_{\Omega}$, $\sigma\in\left\{  +,-\right\}  $, be as explained above.

\begin{enumerate}
\item For $\sigma\in\left\{  +,-\right\}  $, there exist linear one-sided
trace operators (\emph{Dirichlet trace})
\[
\gamma_{\operatorname*{D};j}^{\sigma}:H^{1}\left(  \Omega_{j}^{\sigma}\right)
\rightarrow H^{1/2}\left(  \Gamma_{j}\right)  ,
\]
which are the continuous extensions of the classical trace operators: for
$u\in C^{0}\left(  \overline{\Omega_{j}^{\sigma}}\right)  $, it holds
\[
\gamma_{\operatorname*{D};j}^{\sigma}u=\left.  u\right\vert _{\Gamma_{j}}.
\]
These operators are surjective and bounded%
\begin{equation}
\left\Vert \gamma_{\operatorname*{D};j}^{\sigma}\right\Vert _{H^{1/2}\left(
\Gamma_{j}\right)  \leftarrow H^{1}\left(  \Omega_{j}^{\sigma}\right)  }\leq
C_{\operatorname*{D}}. \label{traceest}%
\end{equation}

For $u\in H^{1}\left(  \mathbb{R}^{3}\right)  $, the one-sided traces
coincide, i.e.,%
\begin{equation}
\gamma_{\operatorname*{D};j}^{-}\left(  \left.  u\right\vert _{\Omega_{j}%
}\right)  =\gamma_{\operatorname*{D};j}^{+}\left(  \left.  u\right\vert
_{\Omega_{j}^{+}}\right)  \label{nojumptrace}%
\end{equation}
and we write short $\gamma_{\operatorname*{D};j}u$ for $\gamma
_{\operatorname*{D};j}^{\sigma}\left(  \left.  u\right\vert _{\Omega
_{j}^{\sigma}}\right)  $, $\sigma\in\left\{  -,+\right\}  $, in such cases.

\item For $\sigma\in\left\{  +,-\right\}  $, there exist linear one-sided
normal trace operators (\emph{normal trace})%
\[
\gamma_{\mathbf{n};j}^{\sigma}:\mathbf{H}\left(  \Omega_{j}^{\sigma
},\operatorname{div}\right)  \rightarrow H^{-1/2}\left(  \Gamma_{j}\right)
\]
which are continuous extensions of the classical normal trace: for $%
\mbox{\boldmath$ \psi$}%
^{\sigma} \in \mathbf{C}^{0} \left(  \overline{\Omega_{j}^{\sigma}}\right)  $, it
holds%
\[
\gamma_{\mathbf{n};j}^{-}\left(
\mbox{\boldmath$ \psi$}%
^{-}\right)  =\left\langle \left.
\mbox{\boldmath$ \psi$}%
^{-}\right\vert _{\Gamma_{j}},\mathbf{n}_{j}\right\rangle \quad\text{and\quad
}\gamma_{\mathbf{n};j}^{+}\left(
\mbox{\boldmath$ \psi$}%
^{+}\right)  =\left\langle \left.
\mbox{\boldmath$ \psi$}%
^{+}\right\vert _{\Gamma_{j}},-\mathbf{n}_{j}\right\rangle ,
\]
where $\mathbf{n}_{j}$ is the unit normal vector on $\Gamma_{j}$ pointing from
$\Omega_{j}^{-}$ into $\Omega_{j}^{+}$.
These operators are bounded
\begin{equation}
\left\Vert \gamma_{\mathbf{n};j}^{\sigma}\right\Vert _{H^{-1/2}\left(
\Gamma_{j}\right)  \leftarrow\mathbf{H}\left(  \Omega_{j}^{\sigma
},\operatorname*{div}\right)  }\leq C_{\mathbf{n}}. \label{normaltracebounded}%
\end{equation}

For $\mbox{\boldmath$ \psi$}\in\mathbf{H}\left(  \mathbb{R}^{3},\operatorname{div}\right)  $ the one-sided
normal traces in the fixed direction $\mathbf{n}_{j}$ coincide, more
precisely,%
\begin{equation}
  \gamma_{\mathbf{n};j}^{-}\left(  \left.
      \mbox{\boldmath$ \psi$}
    \right\vert _{\Omega_{j}^{-}}\right)  =-\gamma_{\mathbf{n};j}^{+}\left(
    \left.
      \mbox{\boldmath$ \psi$}
    \right\vert _{\Omega_{j}^{+}}\right)
\end{equation}
and we write short $\gamma_{\mathbf{n};j}%
\mbox{\boldmath$ \psi$} $ for $\gamma_{\mathbf{n};j}^{-}\left( \left.
    \mbox{\boldmath$ \psi$} \right\vert _{\Omega_{j}^{-}}\right) $.

\item Let $\mathbb{B}\in\mathbb{L}_{>0}^{\infty}\left(  \mathbb{R}%
^{3},\mathbb{R}_{\operatorname*{sym}}^{3\times3}\right)  $. For $\sigma
\in\left\{  +,-\right\}  $, $1\leq j\leq n_{\Omega}$, set $\mathbb{B}%
_{j}^{\sigma}:=\left.  \mathbb{B}\right\vert _{\Omega_{j}^{\sigma}}$. There
exist linear one-sided co-normal derivative operators (\emph{Neumann trace})
\[
\gamma_{\operatorname*{N};j}^{\mathbb{B},\sigma}:H^{1}\left(  \Omega
_{j}^{\sigma},\mathbb{B}_{j}^{\sigma}\right)  \rightarrow H^{-1/2}\left(
\Gamma_{j}\right)
\]
which are the continuous extensions of the classical co-normal derivatives:
for $u^{-}\in C^{1}\left(  \overline{\Omega_{j}^{-}}\right)  $ and $u^{+}\in
C^{1}\left(  \overline{\Omega_{j}^{+}}\right)  $ it holds%
\[
\gamma_{\operatorname*{N};j}^{\mathbb{B},-}u^{-}=\left\langle \mathbb{B}%
_{j}^{-}\nabla u^{-},\mathbf{n}_{j}\right\rangle \quad\text{and\quad}%
\gamma_{\operatorname*{N};j}^{\mathbb{B},+}u^{+}=\left\langle \mathbb{B}%
_{j}^{+}\nabla u^{+},-\mathbf{n}_{j}\right\rangle.
\]
These operators are bounded%
\[
\left\Vert \gamma_{\operatorname*{N};j}^{\mathbb{B},\sigma}\right\Vert
_{H^{-1/2}\left(  \Gamma_{j}\right)  \leftarrow H^{1}\left(  \Omega
_{j}^{\sigma},\mathbb{B}_{j}^{\sigma}\right)  }\leq C_{\operatorname*{N}}.
\]

For $u\in H^{1}\left(  \mathbb{R}^{3},\mathbb{B}\right)  $ the one-sided
co-normal derivatives in the fixed direction $\mathbf{n}_{j}$ coincide, more
precisely,%
\begin{equation}
\gamma_{\operatorname*{N};j}^{\mathbb{B},-}\left(  \left.  u\right\vert
_{\Omega_{j}^{-}}\right)  =-\gamma_{\operatorname*{N},j}^{\mathbb{B},+}\left(
\left.  u\right\vert _{\Omega_{j}^{+}}\right)  \label{continuitynormaltrace}%
\end{equation}
and we write short $\gamma_{\operatorname*{N};j}^{\mathbb{B}}u$ for
$\gamma_{\operatorname*{N};j}^{\mathbb{B},-}\left(  \left.  u\right\vert
_{\Omega_{j}^{-}}\right)  $.
\end{enumerate}
\end{proposition}

The one-sided Dirichlet and Neumann traces are collected in the Cauchy trace
operators $%
\mbox{\boldmath$ \gamma$}%
_{\operatorname*{C};j}^{\mathbb{B},\sigma}:H^{1}\left(  \Omega_{j}^{\sigma
},\mathbb{B}_{j}^{\sigma}\right)  \rightarrow H^{1/2}\left(  \Gamma
_{j}\right)  \times H^{-1/2}\left(  \Gamma_{j}\right)  $ given by%
\begin{equation}%
\mbox{\boldmath$ \gamma$}%
_{\operatorname*{C};j}^{\mathbb{B},\sigma}:=\left(  \gamma_{\operatorname*{D}%
;j}^{\sigma},\gamma_{\operatorname*{N};j}^{\mathbb{B},\sigma}\right)  .
\label{defboldgammaform}%
\end{equation}
For $u\in H^{1}\left(  \mathbb{R}^{3},\mathbb{B}\right)  $ and $u^{\sigma
}:=\left.  u\right\vert _{\Omega_{j}^{\sigma}}$, $\sigma\in\left\{
+,-\right\}  $, the one-sided Cauchy traces satisfy $\left(  \gamma
_{\operatorname*{D};j}^{-}u^{-},\gamma_{\operatorname*{N};j}^{\mathbb{B}%
,-}u^{-}\right)  =\left(  \gamma_{\operatorname*{D};j}^{+}u^{+},-\gamma
_{\operatorname*{N};j}^{\mathbb{B},+}u^{+}\right)  $ and we write%
\begin{equation}%
\mbox{\boldmath$ \gamma$}%
_{\operatorname*{C};j}^{\mathbb{B}}:H^{1}\left(  \mathbb{R}^{3},\mathbb{B}%
\right)  \rightarrow H^{1/2}\left(  \Gamma_{j}\right)  \times H^{-1/2}\left(
\Gamma_{j}\right)  ,\quad%
\mbox{\boldmath$ \gamma$}%
_{\operatorname*{C};j}^{\mathbb{B}}u:=\left(  \gamma_{\operatorname*{D}%
;j}u,\gamma_{\operatorname*{N};j}^{\mathbb{B}}u\right)  . \label{DefCT}%
\end{equation}
We will also use versions of these operators which are scaled by a frequency
parameter $s\in\mathbb{C}_{>0}$ and set for $\sigma\in\left\{  +,-\right\}  $%
\begin{gather}%
\begin{array}[c]{lll}%
\gamma_{\operatorname*{D};j}^{\sigma}\left(  s\right)  :=s^{1/2}%
\gamma_{\operatorname*{D};j}^{\sigma}, & \gamma_{\mathbf{n};j}^{\sigma}\left(
s\right)  :=s^{-1/2}\gamma_{\mathbf{n};j}^{\sigma}, & \gamma
_{\operatorname*{N};j}^{\mathbb{B},\sigma}\left(  s\right)  :=s^{-1/2}%
\gamma_{\operatorname*{N};j}^{\mathbb{B},\sigma},
\\
\gamma_{\operatorname*{D};j}\left(  s\right)  :=s^{1/2}\gamma
_{\operatorname*{D};j}, & \gamma_{\mathbf{n};j}\left(  s\right)
:=s^{-1/2}\gamma_{\mathbf{n};j}, & \gamma_{\operatorname*{N};j}^{\mathbb{B}%
}\left(  s\right)  :=s^{-1/2}\gamma_{\operatorname*{N};j}^{\mathbb{B}},
\end{array}\\
\mbox{\boldmath$ \gamma$}_{\operatorname*{C};j}^{\mathbb{B},\sigma}\left(  s\right)  :=\left(
s^{1/2}\gamma_{\operatorname*{D};j}^{\sigma},s^{-1/2}\gamma_{\operatorname*{N}%
;j}^{\mathbb{B},\sigma}\right)  .
\label{gammadjs}%
\end{gather}

\begin{remark}
  It will turn out that the \textit{Calder\'{o}n operator} (see Def.  \ref{DefCaldOp}) for
  these scaled trace operators has a coercivity estimate which is better balanced with
  respect to the \textit{frequency parameter }$s$ compared to the Calder\'{o}n operator
  for the standard trace operators (see, e.g., \cite{banjai_coupling}).
\end{remark}

\begin{definition}
\label{DefJumpmean}Let $\mathbb{B}\in\mathbb{L}_{>0}^{\infty}\left(
\mathbb{R}^{3},\mathbb{R}_{\operatorname*{sym}}^{3\times3}\right)  $. For
$\sigma\in\left\{  +,-\right\}  $, $1\leq j\leq n_{\Omega}$, set
$\mathbb{B}_{j}^{\sigma}:=\left.  \mathbb{B}\right\vert _{\Omega_{j}^{\sigma}%
}$. For a function $u\in L^{2}\left(  \Omega\right)  $ with $\left.
u\right\vert _{\Omega_{j}^{\sigma}}\in H^{1}\left(  \Omega_{j}^{\sigma
},\mathbb{B}_{j}^{\sigma}\right)  $, the \emph{(Dirichlet) jump} and the
\emph{jump of the co-normal derivative (Neumann jump)} of $u$ across
$\Gamma_{j}$ are given by%
\begin{subequations}
\label{jumpdef}
\end{subequations}%
\begin{align}
\left[  u\right]  _{\operatorname*{D};j}  &  :=\gamma_{\operatorname*{D}%
;j}^{+}\left(  \left.  u\right\vert _{\Omega_{j}^{+}}\right)  -\gamma
_{\operatorname*{D};j}^{-}\left(  \left.  u\right\vert _{\Omega_{j}^{-}%
}\right)  ,\tag{%
\ref{jumpdef}%
a}\label{jumpdefa}\\
\left[  u\right]  _{\operatorname*{N};j}^{\mathbb{B}}  &  :=-\gamma
_{\operatorname*{N};j}^{\mathbb{B},+}\left(  \left.  u\right\vert _{\Omega
_{j}^{+}}\right)  -\gamma_{\operatorname*{N};j}^{\mathbb{B},-}\left(  \left.
u\right\vert _{\Omega_{j}^{-}}\right)  . \tag{%
\ref{jumpdef}%
b}\label{jumpdefb}%
\end{align}
For $s\in\mathbb{C}_{>0}$, the frequency-scaled versions are given by $\left[
u\right]  _{\operatorname*{D};j}\left(  s\right)  :=s^{1/2}\left[  u\right]
_{\operatorname*{D};j}$ and $\left[  u\right]  _{\operatorname*{N}%
;j}^{\mathbb{B}}\left(  s\right)  :=s^{-1/2}\left[  u\right]
_{\operatorname*{N};j}^{\mathbb{B}}.$

The $\emph{(Dirichlet)}$ \emph{mean }and the \emph{mean of the co-normal
derivative (Neumann mean)} across $\Gamma_{j}$ are given by%
\begin{subequations}
\label{meandef}
\end{subequations}%
\begin{align}
\{\!\!\{u\}\!\!\}_{\operatorname*{D};j}  &  :=\frac{1}{2}\left(
\gamma_{\operatorname*{D};j}^{+}\left(  \left.  u\right\vert _{\Omega_{j}^{+}%
}\right)  +\gamma_{\operatorname*{D};j}^{-}\left(  \left.  u\right\vert
_{\Omega_{j}^{-}}\right)  \right)  ,\tag{%
\ref{meandef}%
a}\label{meandefa}\\
\{\!\!\{u\}\!\!\}_{\operatorname*{N};j}^{\mathbb{B}}  &  :=\frac{1}{2}\left(
-\gamma_{\operatorname*{N};j}^{\mathbb{B},+}\left(  \left.  u\right\vert
_{\Omega_{j}^{+}}\right)  +\gamma_{\operatorname*{N};j}^{\mathbb{B},-}\left(
\left.  u\right\vert _{\Omega_{j}^{-}}\right)  \right)  . \tag{%
\ref{meandef}%
b}\label{meandefb}%
\end{align}
For $s\in\mathbb{C}_{>0}$, the frequency-scaled versions are given by
$\{\!\!\{u\}\!\!\}_{\operatorname*{D};j}\left(  s\right)  :=s^{1/2}%
\{\!\!\{u\}\!\!\}_{\operatorname*{D};j}$ and
$\{\!\!\{u\}\!\!\}_{\operatorname*{N};j}^{\mathbb{B}}\left(  s\right)
:=s^{-1/2}\{\!\!\{u\}\!\!\}_{\operatorname*{N};j}^{\mathbb{B}}$.\medskip
\end{definition}

We also need to formulate jump conditions on partial boundaries $\Gamma_{j,k}$
of the subdomains. For a measurable subset $M\subseteq\partial\Omega_{j}$ we
denote by $\left\vert M\right\vert $ its two-dimensional surface measure. Let
$\Omega_{j}$ and $\Omega_{k}$ be such that $\Gamma_{j,k}:=\Gamma_{j}\cap
\Gamma_{k}$ has positive surface measure. We define the Sobolev spaces%
\begin{gather}%
\begin{aligned}
H^{1/2}\left(  \Gamma_{j,k}\right)          & :=\left\{  \left.  \varphi\right\vert
_{\Gamma_{j,k}}:\varphi\in H^{1/2}\left(  \Gamma_{j}\right)  \right\}  , \\
\tilde{H}^{-1/2}\left(  \Gamma_{j,k}\right) & :=\left(  H^{1/2}\left(
\Gamma_{j,k}\right)  \right)  ^{\prime},                                 \\
\tilde{H}^{1/2}\left(  \Gamma_{j,k}\right)  & :=\left\{  \left.  \varphi
\right\vert _{\Gamma_{j,k}}:\varphi\in H^{1/2}\left(  \Gamma_{j}\right)
\wedge\varphi=0\text{ in }\Gamma_{j}\backslash\Gamma_{j,k}\right\},      \\
H^{-1/2}\left(  \Gamma_{j,k}\right)         & :=\left(  \tilde{H}^{1/2}\left(
\Gamma_{j,k}\right)  \right)  ^{\prime}.
\end{aligned}
\label{deftildespaces}%
\end{gather}

\begin{definition}
\label{DefMT}The \emph{multi trace space }$\mathbb{X}\left(  \mathcal{P}%
_{\Omega}\right)  $ for the partition $\mathcal{P}_{\Omega}$ is given by%
\[
\mathbb{X}\left(  \mathcal{P}_{\Omega}\right)  :=%
\BIGOP{\times}%
_{j=1}^{n_{\Omega}}\mathbf{X}_{j}\quad\text{with\quad}\mathbf{X}_{j}%
:=H^{1/2}\left(  \Gamma_{j}\right)  \times H^{-1/2}\left(  \Gamma_{j}\right)
,
\]
and equipped with the norm%
\begin{align*}
  \left\Vert\mbox{\boldmath$ \psi$}_{j}\right\Vert _{\mathbf{X}_{j}}&
  :=\left(  \left\Vert \psi_{\operatorname*{D};j}\right\Vert _{H^{1/2}\left(  \Gamma_{j}\right)  }^{2}+\left\Vert
\psi_{\operatorname*{N};j}\right\Vert _{H^{-1/2}\left(  \Gamma_{j}\right)
}^{2}\right)  ^{1/2} & \forall\mbox{\boldmath$ \psi$}_{j}=\left(  \psi_{\operatorname*{D};j},\psi_{\operatorname*{N};j}\right)
\in\mathbf{X}_{j},\\
\left\Vert\mbox{\boldmath$ \psi$}\right\Vert _{\mathbb{X}} & :=\left(
{\displaystyle\sum_{j=1}^{n_{\Omega}}}
\left\Vert\mbox{\boldmath$ \psi$}_{j}\right\Vert _{\mathbf{X}_{j}}^{2}\right)  ^{1/2} & \forall
\mbox{\boldmath$ \psi$}=\left(\mbox{\boldmath$ \psi$}_{j}\right)  _{j=1}^{n_{\Omega}}\in\mathbb{X}\left(  \mathcal{P}_{\Omega
}\right)  .
\end{align*}
\end{definition}

We seek the solution of our transmission problem in the space%
\[
\mathbb{H}^{1}\left(  \mathcal{P}_{\Omega},\mathbb{A}\right)  :=%
\BIGOP{\times}%
_{j=1}^{n_{\Omega}}H^{1}\left(  \Omega_{j},\mathbb{A}_{j}^{-}\right)
\]
(cf. Assumption \ref{Acoeff}, Remark \ref{RemAgiven}).

Then, for $\mathbf{u}\in%
\BIGOP{\times}%
_{j=1}^{n_{\Omega}}H^{1}\left(  \Omega_{j}\right)  $ and $\mathbf{w}%
\in\mathbb{H}^{1}\left(  \mathcal{P}_{\Omega},\mathbb{B}\right)  $ the jump
$\left[  \mathbf{u}\right]  _{\operatorname*{D};j,k}\in H^{1/2}\left(
\Gamma_{j,k}\right)  $ and the Neumann jump $\left[  \mathbf{w}\right]
_{\operatorname*{N};j,k}^{\mathbb{B}}\in H^{-1/2}\left(  \Gamma_{j,k}\right)
$ across $\Gamma_{j,k}:=\Gamma_{j}\cap\Gamma_{k}$ (and frequency-scaled
versions thereof) are defined by%
\begin{subequations}
\label{partjump}
\begin{align}
  \left[  \mathbf{u}\right]  _{\operatorname*{D};j,k}  &  :=\left.  \left(
      \gamma_{\operatorname*{D},j}^{-}u_{j}\right)  \right\vert _{\Gamma_{j,k}%
  }-\left.  \left(  \gamma_{\operatorname*{D},k}^{-}u_{k}\right)  \right\vert
  _{\Gamma_{j,k}}, & \left[  \mathbf{u}\right]  _{\operatorname*{D}%
    ;j,k}\left(  s\right)  & :=s^{1/2}\left[  \mathbf{u}\right]  _{\operatorname*{D}%
    ;j,k},\label{partjumpa}\\
  \left[  \mathbf{w}\right]  _{\operatorname*{N};j,k}^{\mathbb{B}}  &  :=\left.
    -\left(  \gamma_{\operatorname*{N},j}^{\mathbb{B},-}w_{j}\right)  \right\vert
  _{\Gamma_{j,k}}-\left.  \left(  \gamma_{\operatorname*{N},k}^{\mathbb{B}%
        ,-}w_{k}\right)  \right\vert _{\Gamma_{j,k}},& \left[  \mathbf{w}\right]
  _{\operatorname*{N};j,k}^{\mathbb{B}}\left(  s\right)  & :=s^{-1/2}\left[
    \mathbf{w}\right]  _{\operatorname*{N};j,k}^{\mathbb{B}}. \label{partjumpb}%
\end{align}
\end{subequations}%
We set $\left[  \mathbf{u}\right]_{\operatorname*{D};j,k}:=0$ and $\left[
\mathbf{w}\right]_{\operatorname*{N};j,k}^{\mathbb{B}}:=0$ if $\Gamma_{j,k}$
has zero surface measure or $j=k$.

Note that for coefficients $\mathbb{B}$ and functions $\mathbf{w}$ which are
piecewise sufficiently regular, the Neumann jump across $\Gamma_{j,k}$ can be
written as%
\begin{align*}
\left[  \mathbf{w}\right]  _{\operatorname*{N};j,k}^{\mathbb{B}}  &  =-\left.
\left\langle \gamma_{\operatorname*{D};j}^{-}\left(  \mathbb{B}\nabla
w_{j}\right)  ,\mathbf{n}_{j}\right\rangle \right\vert _{\Gamma_{j,k}}-\left.
\left\langle \gamma_{\operatorname*{D};k}^{-}\left(  \mathbb{B}\nabla
w_{k}\right)  ,\mathbf{n}_{k}\right\rangle \right\vert _{\Gamma_{j,k}}\\
&  =\left.  \left\langle \gamma_{\operatorname*{D};j}^{-}\left(
\mathbb{B}\nabla w_{j}\right)  -\gamma_{\operatorname*{D};k}^{-}\left(
\mathbb{B}\nabla w_{k}\right)  ,\mathbf{n}_{k}\right\rangle \right\vert
_{\Gamma_{j,k}}  =\left.  \left\langle \left[  \mathbb{B}\nabla\mathbf{w}\right]
_{\operatorname*{D};j,k},\mathbf{n}_{k}\right\rangle \right\vert
_{\Gamma_{j,k}},
\end{align*}
where we used $\mathbf{n}_{j}=-\mathbf{n}_{k}$ on $\Gamma_{j,k}$. Clearly
$\left[  \mathbf{u}\right]  _{\operatorname*{D};j,k}=-\left[  \mathbf{u}%
\right]  _{\operatorname*{D};k,j}$ depends on the ordering of the indices
$j,k$, while the Neumann jump is independent of it.

\begin{notation}
\label{NotBAext}We have defined co-normal derivatives, Neumann jumps, and
Neumann means for an abstract coefficient $\mathbb{B}\in\mathbb{L}%
_{>0}^{\infty}\left(  \Omega_{j},\mathbb{R}_{\operatorname*{sym}}^{3\times
3}\right)  $ and used a superscript $\mathbb{B}$ in the notation. In our
application, the choices $\mathbb{B}\leftarrow\mathbb{A}$ and $\mathbb{B}%
\leftarrow\mathbb{A}_{j}^{\operatorname*{ext}}$ will appear. To simplify
notation we skip the superscript $\mathbb{B}$ if $\mathbb{B}=\mathbb{A}$ and
write $\gamma_{\operatorname*{N};j}^{\sigma}$ short for $\gamma
_{\operatorname*{N};j}^{\mathbb{A},\sigma}$ and similar for analogous
quantities. If $\mathbb{B}=\mathbb{A}_{j}^{\operatorname*{ext}}$, we\ replace
the superscript by \textquotedblleft$\operatorname*{ext}$\textquotedblright%
\ and write $\gamma_{\operatorname*{N};j}^{\operatorname*{ext},\sigma}$ short
for $\gamma_{\operatorname*{N};j}^{\mathbb{A}_{j}^{\operatorname*{ext}}%
,\sigma}$ and in the same way for analogous quantities. This convention is
applied verbatim also to the notation of Cauchy traces.
\end{notation}

\subsection{Transmission problem}

Now we have collected all ingredients to state the acoustic transmission problem. Let
$\mathbb{A}\in\mathbb{L}_{>0}^{\infty}\left( \mathbb{R}%
  ^{3},\mathbb{R}_{\operatorname*{sym}}^{3\times3}\right) $ and
$p\in L_{>0}^{\infty}\left( \mathbb{R}^{3},\mathbb{R}\right) $ be given and let the
coefficients in (\ref{homPDEGreen1}) be defined by $\mathbb{A}_{j}%
^{-}:=\left.  \mathbb{A}\right\vert _{\Omega_{j}^{-}}$ and
$p_{j}^{-}:=\left.  p\right\vert _{\Omega_{j}^{-}}$\ such that Assumption \ref{Acoeff} is
satisfied. We do not require that the extensions $\mathbb{A}_{j}%
^{\operatorname*{ext}}$, $p_{j}^{\operatorname*{ext}}$ in Assumption \ref{Acoeff} coincide
with $\mathbb{A}$ (see Remark \ref{RemAgiven}).

The given excitation of the acoustic transmission problem consists of given
data on the skeleton as well as on the Dirichlet and Neumann parts
$\Gamma_{\operatorname*{D}}$ and $\Gamma_{\operatorname*{N}}$ of the boundary
(cf. (\ref{GammaDGammaN})). Let $%
\mbox{\boldmath$ \beta$}%
=\left(
\mbox{\boldmath$ \beta$}%
_{j}\right)  _{j=1}^{n_{\Omega}}\in\mathbb{X}\left(  \mathcal{P}_{\Omega
}\right)  $ with $%
\mbox{\boldmath$ \beta$}%
_{j}=\left(  \beta_{\operatorname*{D};j},\beta_{\operatorname*{N};j}\right)
\in\mathbf{X}_{j}$. For $1\leq j,k\leq n_{\Omega}$, define the jumps of $%
\mbox{\boldmath$ \beta$}%
$ across $\Gamma_{j,k}:=\Gamma_{j}\cap\Gamma_{k}$ by%
\[
\left[
\mbox{\boldmath$ \beta$}%
\right]  _{j,k}:=\left(  \left.  \beta_{\operatorname*{D};j}\right\vert
_{\Gamma_{j,k}}-\left.  \beta_{\operatorname*{D};k}\right\vert _{\Gamma_{j,k}%
},-\left.  \beta_{\operatorname*{N};j}\right\vert _{\Gamma_{j,k}}-\left.
\beta_{\operatorname*{N};k}\right\vert _{\Gamma_{j,k}}\right)
\]
if $j\neq k$ and $\Gamma_{j}\cap\Gamma_{k}$ has positive surface measure.
Otherwise, we set $\left[\mbox{\boldmath$ \beta$}\right]  _{j,k}:=0$.

Given data $\boldsymbol{\beta}\in \mathbb{X}\left(  \mathcal{P}_{\Omega}\right)$, 
the acoustic transmission problem with mixed boundary condition seeks
$\mathbf{u}=\left(  u_{j}\right)  _{j=1}^{n_{\Omega}}\in\mathbb{H}%
^{1}\left(  \mathcal{P}_{\Omega},\mathbb{A}\right)  $ such that%
\begin{equation}%
\begin{array}
[c]{ll}%
-\operatorname{div}\left(  \mathbb{A}_{j}\nabla u_{j}\right)  +s^{2}p_{j}%
u_{j}=0 & \text{in }\Omega_{j},\; 1\leq j\leq n_{\Omega},\\
\left(  \left[  \mathbf{u}\right]  _{\operatorname*{D};j,k}\left(  s\right)
,\left[  \mathbf{u}\right]  _{\operatorname*{N};j,k}^{\mathbb{A}}\left(
s\right)  \right)  =\left[
\mbox{\boldmath$ \beta$}%
\right]  _{j,k}, & 1\leq j,k\leq n_{\Omega},\\
\left.  \left(  \gamma_{\operatorname*{D};j}^{-}\left(  s\right)
u_{j}\right)  \right\vert _{\Gamma_{j}\cap\Gamma_{\operatorname*{D}}}=\left.
\beta_{\operatorname*{D};j}\right\vert _{\Gamma_{j}\cap\Gamma
  _{\operatorname*{D}}} & 1\leq j\leq n_{\Omega},\\
  \left.  \left(  \gamma_{\operatorname*{N}
;j}^{-}\left(  s\right)  u_{j}\right)  \right\vert _{\Gamma_{j}\cap
\Gamma_{\operatorname*{N}}}=\left.  \beta_{\operatorname*{N};j}\right\vert
_{\Gamma_{j}\cap\Gamma_{\operatorname*{N}}} & 1\leq j\leq n_{\Omega}.
\end{array}
\label{generalTP}%
\end{equation}

\begin{remark}
\label{RemIncWave}The inhomogeneity $%
\mbox{\boldmath$ \beta$}%
$ in (\ref{generalTP}) is given in some applications via an incident wave
$u_{\operatorname*{inc}}\in H_{\operatorname*{loc}}^{1}\left(  \mathbb{R}%
^{3},\mathbb{A}_{\nu}^{\operatorname*{ext}}\right)  $ for some fixed $\nu
\in\left\{  1,2,\ldots n_{\Omega}\right\}  $ which satisfies
$-\operatorname{div}\left(  \mathbb{A}_{\nu}^{\operatorname*{ext}}\nabla
u_{\operatorname*{inc}}\right)  +s^{2}p_{\nu}^{\operatorname*{ext}%
}u_{\operatorname*{inc}}=0$ in $\mathbb{R}^{3}$. If $\Omega$ is unbounded,
then typically, $\nu$ is chosen such that $\Omega_{\nu}$ is unbounded. In any
case, it is assumed that the Cauchy trace of $u_{\operatorname*{inc}}$ is well
defined, more precisely, (at least) one of the following two conditions is required:

\begin{enumerate}
\item $%
\mbox{\boldmath$ \gamma$}%
_{\operatorname*{C};\nu}^{-}u_{\operatorname*{inc}}\in\mathbf{X}_{\nu},$

\item the function $u_{\operatorname*{inc}}$ belongs to $C^{1}\left(
\mathbb{R}^{3}\right)  $ and satisfies

\begin{enumerate}
\item the traces $\gamma_{\operatorname*{D};\nu}u_{\operatorname*{inc}}$ and
$\gamma_{\operatorname*{N};\nu}u_{\operatorname*{inc}}$ exist in the classical
pointwise sense,

\item the restrictions of the traces $\left.  \gamma_{\operatorname*{D};\nu
}u_{\operatorname*{inc}}\right\vert _{\Gamma_{\operatorname*{D}}}$ and
$\left.  \gamma_{\operatorname*{N};\nu}u_{\operatorname*{inc}}\right\vert
_{\Gamma_{\operatorname*{N}}}$ have compact supports.
\end{enumerate}
\end{enumerate}
\end{remark}

We will derive the well-posedness of this problem in Section
\ref{SecMultiSingleTraceForm} via layer potentials. For this goal, we will
present a general method to transform such acoustic transmission problems with
mixed boundary conditions and variable coefficients to a system of non-local
Calder\'{o}n operators on the skeleton, without relying on the explicit
knowledge of the Green's function. The resulting boundary integral
operators\footnote{We use here the traditional notion of \textit{boundary}
integral operators (instead of skeleton operators) since they are defined on
the \textit{subdomain} boundaries.} are coercive, self-dual and continuous
(Thm. \ref{ThmWellPosedSkeletonEq}) so that the Lax-Milgram theorem implies
well-posedness. In turn, well-posedness of the original formulation
(\ref{generalTP}) follows.

\section{Potentials and Green's formula\label{SecGRF}}

In the subdomains $\Omega_{j}\in\mathcal{P}_{\Omega}$, a function $u_{j}\in
H^{1}\left(  \Omega_{j},\mathbb{A}_{j}\right)  $ which satisfies the
homogeneous partial differential equation (\ref{homPDEGreen1}) can be
expressed in terms of its Cauchy trace via \textit{layer potentials}. In this
section, we introduce in a fairly standard way the Newton potential and the
single layer potential as solutions to coercive, full space PDEs in
variational form. We present a new definition for the double layer potential
as a solution of an ultra-weak variational problem. This allows us to derive
its mapping properties and jump relations from the theory of elliptic PDEs.
Finally, we derive a \textit{Green's representation formula} for our acoustic
transmission problem based on these potentials.

\subsection{Sesquilinear forms and associated operators}

Throughout this section we require that Assumption \ref{Acoeff} holds and
employ the notation%
\begin{align*}
\Omega_{j}^{-}&:=\Omega_{j}, & \Omega_{j}^{+}&:=\mathbb{R}^{3}\backslash
\overline{\Omega_{j}},\\
\mathbb{A}_{j}^{+}&:=\left.  \mathbb{A}_{j}^{\operatorname*{ext}}\right\vert
_{\Omega_{j}^{+}}, & p_{j}^{+}&:=\left.  p_{j}^{\operatorname*{ext}}\right\vert
_{\Omega_{j}^{+}}.
\end{align*}
We also need the piecewise gradient $\nabla_{\operatorname*{pw};j}$ which is given, for a
function $w\in H^{1}\left( \mathbb{R}^{3}\backslash\Gamma _{j}\right) $, by
\begin{equation}
\left.  \left(  \nabla_{\operatorname*{pw};j}w\right)  \right\vert
_{\Omega_{j}^{\sigma}}:=\nabla\left(  \left.  w\right\vert _{\Omega
_{j}^{\sigma}}\right)\;,\quad \sigma\in\left\{  -,+\right\}  \label{defpwgradient}%
\end{equation}
and considered as a function in $\mathbf{L}^{2}\left(  \mathbb{R}^{3}\right)
$.

\begin{definition}
\label{Defellsj}Let Assumption \ref{Acoeff} be satisfied. For $s\in
\mathbb{C}_{>0}$, the sesquilinear form%
\[
\ell_{j}\left(  s\right)  :H^{1}\left(  \mathbb{R}^{3}\right)  \times
H^{1}\left(  \mathbb{R}^{3}\right)  \rightarrow\mathbb{C}%
\]
is given by%
\[
\ell_{j}\left(  s\right)  \left(  u,v\right)  :=\left\langle \mathbb{A}%
_{j}^{\operatorname*{ext}}\nabla u,\overline{\nabla v}\right\rangle
_{\mathbb{R}^{3}}+s^{2}\left\langle p_{j}^{\operatorname*{ext}}u,\overline
{v}\right\rangle _{\mathbb{R}^{3}}\qquad\forall u,v\in H^{1}\left(
\mathbb{R}^{3}\right)  ,
\]
and the associated operator $\mathsf{L}_{j}\left(  s\right)  :H^{1}\left(
\mathbb{R}^{3}\right)  \rightarrow H^{-1}\left(  \mathbb{R}^{3}\right)  $ by%
\begin{equation}
\left\langle \mathsf{L}_{j}\left(  s\right)  u,\overline{v}\right\rangle
_{\mathbb{R}^{3}}:=\ell_{j}\left(  s\right)  \left(  u,v\right)  \qquad\forall
u,v\in H^{1}\left(  \mathbb{R}^{3}\right)  . \label{FF7ab}%
\end{equation}

\end{definition}

Next, we prove continuity and coercivity for the sesquilinear form
$\ell _{j}\left( s\right) \left( \cdot,\cdot\right) $ in the spirit of
\cite{bambduong}. We take pains to elaborate the explicit dependence of the constants on
$s$.

\begin{lemma}
\label{LemLaxMilgram}Let Assumption \ref{Acoeff} be satisfied. The
sesquilinear forms $\ell_{j}$ are continuous and coercive: for $\mu
:=s/\left\vert s\right\vert $ and for all $  $ holds
\begin{align}
  \label{eq:7}
  [c]{ll}%
  \left\vert \ell_{j}\left(  s\right)  \left(  v,w\right)  \right\vert
  & \leq\Lambda_{j}\left\Vert v\right\Vert _{H^{1}\left(  \mathbb{R}^{3}\right)
    ;s}\left\Vert w\right\Vert _{H^{1}\left(  \mathbb{R}^{3}\right)  ;s} \quad \forall v,w\in H^{1}\left(
    \mathbb{R}^{3}\right),\\
  \operatorname{Re}\ell_{j}\left(  s\right)  \left(  v,\mu v\right)  & \geq
  \lambda_{j}\frac{\operatorname{Re}s}{\left\vert s\right\vert }\left\Vert
    v\right\Vert _{H^{1}\left(  \mathbb{R}^{3}\right)  ;s}^{2}\quad \forall v\in H^{1}\left(
    \mathbb{R}^{3}\right),
\end{align}
with%
\begin{equation}
\lambda_{j}:=\min\left\{  \lambda_{j}\left(  p_{j}^{\operatorname*{ext}%
}\right)  ,\lambda_{j}\left(  \mathbb{A}_{j}^{\operatorname*{ext}}\right)
\right\}  \quad\text{and\quad}\Lambda_{j}:=\max\left\{  \Lambda_{j}\left(
p_{j}^{\operatorname*{ext}}\right)  ,\Lambda_{j}\left(  \mathbb{A}%
_{j}^{\operatorname*{ext}}\right)  \right\}  . \label{deflambdaj}%
\end{equation}
\end{lemma}%

\begin{proof}
Fix $\mu=s/\left\vert s\right\vert $. For $v\in H^{1}\left(  \mathbb{R}%
^{3}\right)  $, it holds%
\begin{align}
\operatorname{Re}\ell_{j}\left(  s\right)  \left(  v,\mu v\right)   &
=\operatorname{Re}\left\langle s^{2}p_{j}^{\operatorname*{ext}}v,\overline{\mu
v}\right\rangle _{\mathbb{R}^{3}}+\operatorname{Re}\left\langle \mathbb{A}%
_{j}^{\operatorname*{ext}}\nabla v,\overline{\mu\nabla v}\right\rangle
_{\mathbb{R}^{3}}\label{ljcoercive}\\
&  \geq\lambda\left(  p_{j}^{\operatorname*{ext}}\right)  \operatorname{Re}%
\left(  s^{2}\overline{\mu}\right)  \left\Vert v\right\Vert _{L^{2}\left(
\mathbb{R}^{3}\right)  }^{2}+\lambda\left(  \mathbb{A}_{j}%
^{\operatorname*{ext}}\right)  \left(  \operatorname{Re}\mu\right)  \left\Vert
\nabla v\right\Vert _{\mathbf{L}^{2}\left(  \mathbb{R}^{3}\right)  }%
^{2}\nonumber\\
&  \geq\frac{\operatorname{Re}s}{\left\vert s\right\vert }\lambda
_{j}\left\Vert v\right\Vert _{H^{1}\left(  \mathbb{R}^{3}\right)  ;s}%
^{2}.\nonumber
\end{align}
To establish continuity, we use%
\begin{align*}
\left\vert \ell_{j}\left(  s\right)  \left(  v,w\right)  \right\vert  &
=\left\vert s^{2}\left\langle p_{j}^{\operatorname*{ext}}v,%
\overline{w}%
\right\rangle _{\mathbb{R}^{3}}\right\vert +\left\vert \left\langle
\mathbb{A}_{j}^{\operatorname*{ext}}\nabla v,%
\overline{\nabla w}%
\right\rangle _{\mathbb{R}^{3}}\right\vert \\
&  \leq\Lambda\left(  p_{j}^{\operatorname*{ext}}\right)  \left\vert
s\right\vert ^{2}\left\Vert v\right\Vert _{L^{2}\left(  \mathbb{R}^{3}\right)
}\left\Vert w\right\Vert _{L^{2}\left(  \mathbb{R}^{3}\right)  }%
+\Lambda\left(  \mathbb{A}_{j}^{\operatorname*{ext}}\right)  \left\Vert \nabla
v\right\Vert _{L^{2}\left(  \mathbb{R}^{3}\right)  }\left\Vert \nabla
w\right\Vert _{L^{2}\left(  \mathbb{R}^{3}\right)  }\\
&  \leq\Lambda_{j}\left\Vert v\right\Vert _{H^{1}\left(  \mathbb{R}%
^{3}\right)  ;s}\left\Vert w\right\Vert _{H^{1}\left(  \mathbb{R}^{3}\right)
;s}%
\end{align*}
for all $v,w\in H^{1}\left(  \mathbb{R}^{3}\right)  $.
\end{proof}

Since the right-hand side in the first equation of (\ref{generalTP}) is the
zero function we conclude that a solution $u_{j}$ belongs to $H^{1}\left(
\Omega_{j}^{-},\mathbb{A}_{j}^{-}\right)  $.

\begin{lemma}
[Green's identities]Let Assumption \ref{Acoeff} be satisfied and set
$\mathbb{A}_{j}^{\sigma}:=\left.  \mathbb{A}_{j}^{\operatorname*{ext}%
}\right\vert _{\Omega_{j}^{\sigma}}$, $p_{j}^{\sigma}:=\left.  p_{j}%
^{\operatorname*{ext}}\right\vert _{\Omega_{j}^{\sigma}}$ for $\sigma
\in\left\{  +,-\right\}  $.

\begin{enumerate}
\item For any $\sigma\in\left\{  +,-\right\}  $, assume that $v^{\sigma}\in
H^{1}\left(  \Omega_{j}^{\sigma},\mathbb{A}_{j}^{\sigma}\right)  $ satisfies%
\begin{equation}
\mathsf{L}_{j}^{\sigma}\left(  s\right)  v^{\sigma}=0\quad\text{in }\Omega
_{j}^{\sigma}. \label{lochom}%
\end{equation}
Then, the co-normal derivative of $v^{\sigma}$ satisfies%
\begin{equation}
  \begin{aligned}
    \left\langle \mathbb{A}_{j}^{\sigma}\nabla v^{\sigma},\nabla\overline
      {w}\right\rangle _{\Omega_{j}^{\sigma}}+\quad&\\
    s^{2}\left\langle p_{j}^{\sigma
      }v^{\sigma},\overline{w}\right\rangle _{\Omega_{j}^{\sigma}} & =\left\langle
      \gamma_{\operatorname*{N};j}^{\operatorname*{ext},\sigma}\left(  s\right)
      v^{\sigma},\gamma_{\operatorname*{D};j}^{\sigma}\left(  s\right)  \overline
      {w}\right\rangle _{\Gamma_{j}}\quad\forall w\in H^{1}\left(  \Omega
      _{j}^{\sigma}\right)  .
  \end{aligned}
  \label{Green1}%
\end{equation}

\item For $v\in H^{1}\left(  \mathbb{R}^{3}\right)  $, set $v^{\sigma
}:=\left.  v\right\vert _{\Omega_{j}^{\sigma}}$. Assume that $v^{\sigma}$
belongs to $H^{1}\left(  \Omega_{j}^{\sigma},\mathbb{A}_{j}^{\sigma}\right)  $
and satisfies (\ref{lochom}) for $\sigma\in\left\{  +,-\right\}  $. Then%
\begin{equation}
\ell_{j}\left(  s\right)  \left(  v,w\right)  =\left\langle -\left[  v\right]
_{\operatorname*{N};j}^{\operatorname*{ext}}\left(  s\right)  ,\gamma
_{\operatorname*{D};j}\left(  s\right)  \overline{w}\right\rangle _{\Gamma
_{j}},\qquad\forall w\in H^{1}\left(  \mathbb{R}^{3}\right)  .
\label{ljjumprel}%
\end{equation}

\item For $v\in L^{2}\left(  \mathbb{R}^{3}\right)  $, set $v^{\sigma
}:=\left.  v\right\vert _{\Omega_{j}^{\sigma}}$. Assume $v^{\sigma}\in
H^{1}\left(  \Omega_{j}^{\sigma},\mathbb{A}_{j}^{\sigma}\right)  $, $\left[
v\right]  _{\operatorname*{N};j}^{\operatorname*{ext}}=0$, and that
$v^{\sigma}$ satisfies, (\ref{lochom}). Then
\begin{equation}
\sum_{\sigma\in\left\{  +,-\right\}  }\left\langle \mathbb{A}_{j}^{\sigma
}\nabla v^{\sigma},\nabla\overline{w^{\sigma}}\right\rangle _{\Omega
_{j}^{\sigma}}+s^{2}\left\langle p_{j}^{\sigma}v^{\sigma},\overline{w^{\sigma
}}\right\rangle _{\Omega_{j}^{\sigma}}=\left\langle \gamma_{\operatorname*{N}%
;j}^{\operatorname*{ext}}\left(  s\right)  v,-\left[  \overline{w}\right]
_{\operatorname*{D};j}\left(  s\right)  \right\rangle _{\Gamma_{j}}
\label{ljconjump}%
\end{equation}
for any $w\in L^{2}\left(  \mathbb{R}^{3}\right)  $ with $w^{\sigma}:=\left.
w\right\vert _{\Omega_{j}^{\sigma}}\in H^{1}\left(  \Omega^{\sigma}\right)  $,
$\sigma\in\left\{  +,-\right\}  $.

\item Let $v^{\sigma},w^{\sigma}\in H^{1}\left(  \Omega_{j}^{\sigma
},\mathbb{A}_{j}^{\sigma}\right)  $. Then,%
\begin{align}
\left\langle v^{\sigma},\mathsf{L}_{j}^{\sigma}\left(  s\right)
\overline{w^{\sigma}}\right\rangle _{\Omega_{j}^{\sigma}}- \left\langle
\mathsf{L}_{j}^{\sigma}\left(  s\right)  v^{\sigma},\overline{w^{\sigma}%
}\right\rangle _{\Omega_{j}^{\sigma}} & =\left\langle \gamma_{\operatorname*{N}%
;j}^{\operatorname*{ext},\sigma}\left(  s\right)  v^{\sigma},\gamma
_{\operatorname*{D};j}^{\sigma}\left(  s\right)  \overline{w^{\sigma}%
                                          }\right\rangle _{\Gamma_{j}}\label{2ndGreen}\\
                                      &
                                        \quad-\left\langle \gamma_{\operatorname*{D};j}%
^{\sigma}\left(  s\right)  v^{\sigma},\gamma_{\operatorname*{N};j}%
^{\operatorname*{ext},\sigma}\left(  s\right)  \left(  \overline{w}^{\sigma
}\right)  \right\rangle _{\Gamma_{j}}.\nonumber
\end{align}

\end{enumerate}
\end{lemma}%

\begin{proof}
\textbf{@ 1. }For any $v\in H^{1}\left(  \Omega_{j}^{\sigma},\mathbb{A}%
_{j}^{\sigma}\right)  $, it holds%
\begin{align}
\left\langle \mathbb{A}_{j}^{\sigma}\nabla v,\overline{\nabla w}\right\rangle
_{\Omega_{j}^{\sigma}}+\left\langle s^{2}p_{j}^{\sigma}v,\overline
{w}\right\rangle _{\Omega_{j}^{\sigma}}  &  =\left\langle \mathsf{L}%
_{j}^{\sigma}\left(  s\right)  v,\overline{w}\right\rangle _{\Omega
_{j}^{\sigma}}+\left\langle \gamma_{\operatorname*{N};j}^{\operatorname*{ext}%
,\sigma}\left(  s\right)  v,\gamma_{\operatorname*{D};j}^{\sigma}\left(
s\right)  \overline{w}\right\rangle _{\Gamma_{j}}\label{Greens2ndId}\\
&  \overset{\text{(\ref{lochom})}}{=}\left\langle \gamma_{\operatorname*{N}%
;j}^{\operatorname*{ext},\sigma}\left(  s\right)  v,\gamma_{\operatorname*{D}%
;j}^{\sigma}\left(  s\right)  \overline{w}\right\rangle _{\Gamma_{j}}%
\qquad\forall w\in H^{1}\left(  \Omega_{j}^{\sigma}\right)  .\nonumber
\end{align}

\textbf{@ 2. }Let $v\in H^{1}\left(  \mathbb{R}^{3}\right)  $ and assume $v$
satisfies the conditions in part 2. We conclude from part 1 that%
\begin{align*}
\ell_{j}\left(  s\right)  \left(  v,w\right)   &  =\sum_{\sigma\in\left\{
+,-\right\}  }\left\langle \mathbb{A}_{j}^{\sigma}\nabla v,\overline{\nabla
w}\right\rangle _{\Omega_{j}^{\sigma}}+\left\langle s^{2}p_{j}^{\sigma
}v,\overline{w}\right\rangle _{\Omega_{j}^{\sigma}}\\ & =\left\langle
\gamma_{\operatorname*{N};j}^{\operatorname*{ext},+}\left(  s\right)
v^{+}+\gamma_{\operatorname*{N};j}^{\operatorname*{ext},-}\left(  s\right)
v^{-},\gamma_{\operatorname*{D};j}\left(  s\right)  \overline{w}\right\rangle
                                                        _{\Gamma_{j}}\\
  &=\left\langle -\left[  v\right]  _{\operatorname*{N};j}%
^{\operatorname*{ext}}\left(  s\right)  ,\gamma_{\operatorname*{D};j}\left(
s\right)  \overline{w}\right\rangle _{\Gamma_{j}}%
\end{align*}
holds for all $w\in H^{1}\left(  \mathbb{R}^{3}\right)  $.

\textbf{@ 3. }The relation (\ref{ljconjump}) follows in the same fashion as (\ref{Green1}). 

\textbf{@ 4. }Relation (\ref{2ndGreen}) follows by integrating by parts the
first term in (\ref{Greens2ndId}).
\end{proof}

\subsection{Volume and layer potentials}

In this section we define volume and layer potentials as solutions to certain
variational formulations of elliptic partial differential equations without
relying on the explicit knowledge of the Green's function.

\subsubsection{The Newton potential} 

We will define the acoustic Newton potential as the solution of the
variational formulation of a full space partial differential equation
depending on a single subdomain $\Omega_{j}$, corresponding to extended
coefficients $\mathbb{A}_{j}^{\operatorname*{ext}}$, $p_{j}%
^{\operatorname*{ext}}$, and the frequency parameter $s$.

\begin{definition}
\label{DefSolOp}Let Assumption \ref{Acoeff} be satisfied. The solution
operator (acoustic Newton potential) $\mathsf{N}_{j}\left(  s\right)
:H^{-1}\left(  \mathbb{R}^{3}\right)  \rightarrow H^{1}\left(  \mathbb{R}%
^{3}\right)  $ is defined through
\begin{gather}
  \label{eq:4}
  \ell_{j}\left( s\right) \left( \mathsf{N}_{j}\left( s\right) f,w\right) =\left\langle
    f,\overline{w}\right\rangle _{\mathbb{R}^{3}}\qquad\forall f\in H^{-1}\left(
    \mathbb{R}^{3}\right) ,\quad\forall w\in H^{1}\left( \mathbb{R}^{3}\right) .
\end{gather}

\end{definition}

Lemma \ref{LemLaxMilgram} implies that $\ell_{j}\left(  s\right)  $ is
continuous and coercive. Hence, the Lax-Milgram theorem ensures that%
\begin{equation}
\mathsf{N}_{j}\left(  s\right)  :H^{-1}\left(  \mathbb{R}^{3}\right)
\rightarrow H^{1}\left(  \mathbb{R}^{3}\right)  \label{mappropNjs}%
\end{equation}
is well defined, linear, and bounded. An estimate of the operator norm in
frequency dependent norms (see (\ref{fs_norm}), (\ref{defdualnorm})) is given
by the following lemma. Note that the dual space of $\left(  H^{1}\left(
\mathbb{R}^{3}\right)  ,\left\Vert \cdot\right\Vert _{H^{1}\left(
\mathbb{R}^{3}\right)  ;s}\right)  $ is given by $\left(  H^{-1}\left(
\mathbb{R}^{3}\right)  ,\left\Vert \cdot\right\Vert _{H^{-1}\left(
\mathbb{R}^{3}\right)  ;s}\right)  $ with dual norm defined by%
\begin{equation}
\left\Vert f\right\Vert _{H^{-1}\left(  \mathbb{R}^{3}\right)  ;s}:=\sup_{g\in
H^{1}\left(  \mathbb{R}^{3}\right)  \backslash\left\{  0\right\}  }%
\frac{\left\vert \left\langle f,\overline{g}\right\rangle _{\mathbb{R}^{3}%
}\right\vert }{\left\Vert g\right\Vert _{H^{1}\left(  \mathbb{R}^{3}\right)
;s}}. \label{defdualnorm}%
\end{equation}

\begin{lemma}
Let Assumption \ref{Acoeff} be satisfied. The Newton potential is an
inverse of $\mathsf{L}_{j}\left(  s\right)  $, i.e.,%
\begin{gather*}
v=\mathsf{N}_{j}\left(  s\right)  \circ\mathsf{L}_{j}\left(  s\right)
v\quad\forall v\in H^{1}\left(  \mathbb{R}^{3}\right)  \quad%
\text{and\quad}f=\mathsf{L}_{j}\left(  s\right)  \circ\mathsf{N}_{j}\left(
s\right)  f\quad\forall f\in H^{-1}\left(  \mathbb{R}^{3}\right);  \label{IDNL}%
\end{gather*}%
\modified{it} satisfies the \rhc{$s$-explicit} estimate%
\begin{equation}
\left\Vert \mathsf{N}_{j}\left(  s\right)  f\right\Vert _{H^{1}\left(
\mathbb{R}^{3}\right)  ;s}\leq\frac{\left\vert s\right\vert }{\lambda
_{j}\operatorname{Re}s}\left\Vert f\right\Vert _{H^{-1}\left(  \mathbb{R}%
^{3}\right)  ;s}\quad\forall f\in H^{-1}\left(  \mathbb{R}^{3}\right)  ,
\label{Njest}%
\end{equation}
with $\lambda_{j}$ as in (\ref{deflambdaj}).
\end{lemma}

\begin{proof}
For $v\in H^{1}\left(  \mathbb{R}^{3}\right)  $, we have $\mathsf{L}%
_{j}\left(  s\right)  v\in H^{-1}\left(  \mathbb{R}^{3}\right)  $ and hence
the Newton potential can be applied:%
\[
\ell_{j}\left(  s\right)  \left(  \mathsf{N}_{j}\left(  s\right)
\circ\mathsf{L}_{j}\left(  s\right)  v,w\right)  =\left\langle \mathsf{L}%
_{j}\left(  s\right)  v,\overline{w}\right\rangle _{\mathbb{R}^{3}}=\ell
_{j}\left(  s\right)  \left(  v,w\right)  \qquad\forall w\in H^{1}\left(
\mathbb{R}^{3}\right)  .
\]
Since $\ell_{j}\left(  s\right)  \left(  \cdot,\cdot\right)  $ is coercive
the first identity in (\ref{IDNL}) follows.
The second one is a direct consequence of the definition of $\mathsf{N}%
_{j}\left(  s\right)  $.%

{To prove (\ref{Njest}), we use the coercivity of $\ell_{j}\left(  s\right)
\left(  \cdot,\cdot\right)  $ with respect to the Hilbert space $\left(
H^{1}\left(  \mathbb{R}^{3}\right)  ,\left\Vert \cdot\right\Vert
_{H^{1}\left(  \mathbb{R}^{3}\right)  ;s}\right)$ as stated in Lemma~\ref{LemLaxMilgram}.} From
the Babu\v{s}ka-Lax-Milgram theorem \cite[Thm. 2.1]{BabuskaLaxMilgram} and the
definition (\ref{defdualnorm}) of the dual norm the assertion follows.
\end{proof}

\subsubsection{The single layer potential\label{slp}}

The single layer potential is defined by using the same sesquilinear form as
for the Newton potential for a certain type of right-hand sides.

\begin{definition}
\label{DefSLP}Let Assumption \ref{Acoeff} be satisfied. For $1\leq j\leq
n_{\Omega}$ and $\varphi\in H^{-1/2}\left(  \Gamma_{j}\right)  $ the
\emph{\bf single layer potential} $\mathsf{S}_{j}\left(  s\right)  \varphi\in
H^{1}\left(  \mathbb{R}^{3}\right)  $ is given as the unique solution of:%
\begin{equation}
\boxed{\ell_{j}\left(  s\right)  \left(  \mathsf{S}_{j}\left(  s\right)
\varphi,w\right)  =\left\langle \varphi,\gamma_{\operatorname*{D};j}\left(
s\right)  \overline{w}\right\rangle _{\Gamma_{j}}\quad\forall w\in
H^{1}\left(  \mathbb{R}^{3}\right)}  . \label{DefSLPForm}%
\end{equation}
\end{definition}

This defines a continuous operator
$H^{-1/2}\left( \Gamma_{j}\right)\to H^{1}(\mathbb{R}^{3})$. The single layer can be
represented as the composition of the Newton potential and the dual Dirichlet trace as can
be seen from the next lemma, where also important properties of
$\mathsf{S}_{j}\left( s\right) $ are collected.

\begin{lemma}
\label{LemMapPropSLP}Let Assumption \ref{Acoeff} be satisfied. Then%
\begin{equation}
\mathsf{S}_{j}\left(  s\right)  =\mathsf{N}_{j}\left(  s\right)  \circ\left(
\gamma_{\operatorname*{D};j}\left(  s\right)  \right)  ^{\prime}.
\label{elljsS}%
\end{equation}
For any $\varphi\in H^{-1/2}\left(  \Gamma\right)  $, the single layer
potential $u:=\mathsf{S}_{j}\left(  s\right)  \varphi$ satisfies
\begin{equation*}
  u\in
  H^{1}\left(  \mathbb{R}^{3}\backslash\Gamma_{j},\mathbb{A}_{j}%
    ^{\operatorname*{ext}}\right).
\end{equation*}
For the restrictions $u^{\sigma
}:=\left.  u\right\vert _{\Omega_{j}^{\sigma}}$, $\sigma\in\left\{
+,-\right\}  $, hold
\begin{equation}
\mathsf{L}_{j}^{\sigma}\left(  s\right)  u^{\sigma}=0\quad\text{in\ }%
\Omega_{j}^{\sigma} \label{SlpProp1}%
\end{equation}
and the jump relations%
\begin{equation}%
\boxed{\left[  \left(  \mathsf{S}_{j}\left(  s\right)  \varphi\right)  \right]
_{\operatorname*{D};j}\left(  s\right)  =0\quad, \quad \left[  \left(  \mathsf{S}%
_{j}\left(  s\right)  \varphi\right)  \right]  _{\operatorname*{N}%
;j}^{\operatorname*{ext}}\left(  s\right)  =-\varphi}.
\label{jumprelSLP}%
\end{equation}

\end{lemma}%

\begin{proof}
The representation (\ref{elljsS}) follows by writing (\ref{DefSLPForm}) as%
\[
\ell_{j}\left(  s\right)  \left(  \mathsf{S}_{j}\left(  s\right)
\varphi,w\right)  =\left\langle \left(  \gamma_{\operatorname*{D};j}\left(
s\right)  \right)  ^{\prime}\varphi,\overline{w}\right\rangle _{%
\mathbb{R}^{3}}\quad%
\forall w\in H^{1}\left(  \mathbb{R}^{3}\right)  ,
\]
so that $\mathsf{S}_{j}\left(  s\right)  \varphi=\mathsf{N}_{j}\left(
s\right)  \left(  \gamma_{\operatorname*{D};j}\left(  s\right)  \right)
^{\prime}\varphi$. Indeed, the mapping properties of the dual Dirichlet trace
$\left(  \gamma_{\operatorname*{D};j}\left(  s\right)  \right)  ^{\prime
}:H^{1/2}\left(  \Gamma_{j}\right)  \rightarrow H^{-1}\left(  \mathbb{R}%
^{3}\right)  $ imply that the Newton potential can be applied in (\ref{elljsS}).

For $\varphi\in H^{-1/2}\left(  \Gamma\right)  $, let $u:=\mathsf{S}%
_{j}\left(  s\right)  \varphi$ and $u^{\sigma}:=\left.  u\right\vert
_{\Omega_{j}^{\sigma}}$. By choosing in (\ref{DefSLPForm}) test functions
$v\in H^{1}\left(  \mathbb{R}^{3}\right)  $ with zero trace on $\Gamma_{j}$ we
obtain%
\[
\mathsf{L}_{j}^{\sigma}\left(  s\right)  u^{\sigma}=0\quad\text{in }\Omega
_{j}^{\sigma}\text{, }\sigma\in\left\{  +,-\right\}  .
\]
In particular, this implies $u\in H^{1}\left(  \mathbb{R}^{3}\backslash
\Gamma_{j},\mathbb{A}_{j}^{\operatorname*{ext}}\right)  $. An integration by
parts in (\ref{DefSLPForm}) over $\Omega_{j}^{-}$ and $\Omega_{j}^{+}$ leads
to%
\[
-\left\langle \left[  u\right]  _{\operatorname*{N};j}^{\operatorname*{ext}%
}\left(  s\right)  ,\gamma_{\operatorname*{D};j}\left(  s\right)  \overline
{w}\right\rangle _{\Gamma_{j}}=\left\langle \varphi,\gamma_{\operatorname*{D}%
;j}\left(  s\right)  \overline{w}\right\rangle _{\Gamma_{j}}\quad\forall w\in
H^{1}\left(  \mathbb{R}^{3}\right)  .
\]
Since $\gamma_{\operatorname*{D};j}\left(  s\right)  :H^{1}\left(
\mathbb{R}^{3}\right)  \rightarrow H^{1/2}\left(  \Gamma_{j}\right)  $ is
surjective (see, e.g., \cite[Thm. 3.37]{Mclean00}, \cite[Lem. 2.6]%
{Mikhailov_traces_1}) it follows that $\left[  u\right]  _{\operatorname*{N}%
;j}^{\operatorname*{ext}}\left(  s\right)  =-\varphi$. Finally, the relation
$\left[  u\right]  _{\operatorname*{D};j}\left(  s\right)  =0$ follows from
$u\in H^{1}\left(  \mathbb{R}^{3}\right)  $ (see, e.g. \cite[(6.20)]%
{Mclean00}, \cite[Lem. 2.5]{Mikhailov_traces_1}).
\end{proof}

\subsubsection{The double layer potential\label{dlp}}

Next, we introduce the double layer potential and start by reviewing some
standard definitions as already sketched in the introduction. For problems
with constant coefficients as, e.g., in \cite[Def. 3.1.5]{SauterSchwab2010},
the double layer potential is defined by%
\begin{equation}
\mathsf{D}_{j}\left(  s\right)  :=\mathsf{N}_{j}\left(  s\right)  \circ\left(
\gamma_{\operatorname*{N};j}^{\operatorname*{ext}}\right)  ^{\prime}\left(
s\right)  . \label{defDstandard}%
\end{equation}
The continuity of the co-normal derivative $\gamma_{\operatorname*{N}%
;j}^{\operatorname*{ext}}:H^{1}\left(  \mathbb{R}^{3},\mathbb{A}%
_{j}^{\operatorname*{ext}}\right)  \rightarrow H^{-1/2}\left(  \Gamma
_{j}\right)  $ (see (\ref{continuitynormaltrace})) leads to the continuity of
its dual $\left(  \gamma_{\operatorname*{N};j}^{\operatorname*{ext}}\right)
^{\prime}:H^{1/2}\left(  \Gamma_{j}\right)  \rightarrow\left(  H^{1}\left(
\mathbb{R}^{3},\mathbb{A}_{j}^{\operatorname*{ext}}\right)  \right)  ^{\prime
}$. The problem with (\ref{defDstandard}) is that the image space $\left(
H^{1}\left(  \mathbb{R}^{3},\mathbb{A}_{j}^{\operatorname*{ext}}\right)
\right)  ^{\prime}$ in general is larger than $H^{-1}\left(  \mathbb{R}%
^{3}\right)  $ and hence exceeds the domain of $\mathsf{N}_{j}\left(
s\right)  $ in (\ref{defDstandard}). The extension of the domain of
$\mathsf{N}_{j}\left(  s\right)  $ for problems with varying coefficients is
far from trivial. Another common definition uses explicit knowledge of the
fundamental solution $G\left(  \mathbf{x},\mathbf{y}\right)  $ and first
defines%
\[
\left(  \mathsf{D}_{j}\left(  s\right)  \psi\right)  \left(  \mathbf{x}%
\right)  :=\int_{\Gamma_{j}}\left(  \frac{\partial}{\partial\mathbf{\tilde{n}%
}_{\mathbf{y}}}G\left(  \mathbf{x},\mathbf{y}\right)  \right)  \psi\left(
\mathbf{y}\right)  d\Gamma_{\mathbf{y}}\quad\mathbf{x}\in\mathbb{R}%
^{3}\backslash\Gamma_{j}%
\]
($\partial/\partial\mathbf{\tilde{n}}_{\mathbf{y}}$ with $\mathbf{\tilde{n}%
}_{\mathbf{y}}:=\mathbb{A}_{j}^{\operatorname*{ext}}\mathbf{n}_{j}$ denotes
the co-normal derivative with respect to $\mathbf{y}$) for coefficients
$\mathbb{A}_{j}^{\operatorname*{ext}}$ and boundary densities $\psi:\Gamma
_{j}\rightarrow\mathbb{C}$, which are sufficiently regular, and then
continuously extends this definition to appropriate Sobolev spaces. However,
the derivation of mapping properties of $\mathsf{D}_{j}\left(  s\right)  $ via
this approach relies on properties of the unknown fundamental solution and is
far from trivial for problems with $L^{\infty}$ coefficients.

Instead, we present here a new definition of the double layer potential as a
solution of some ultra-weak variational problem which allows us to derive
properties of these potentials directly from the well-established theory of
linear elliptic partial differential operators of second order.

For the definition of the double layer potential we introduce two
auxiliary variational problems.

\noindent\fbox{\begin{minipage}{0.98\linewidth}
  I. \textit{Ultra-weak variational problem (UWVP)}: Given
  $\psi\in H^{1/2}\left( \Gamma_{j}\right) $, find
  $u\in L^{2}\left( \mathbb{R}^{3}\right) $ such that%
  \begin{gather}
    \left\langle u,\mathsf{L}_{j}\left(  s\right)  \overline{v}\right\rangle
    _{\mathbb{R}^{3}}=\left\langle \psi,\gamma_{\operatorname*{N};j}%
      ^{\operatorname*{ext}}\left(  s\right)  \overline{v}\right\rangle _{\Gamma
      j}\qquad\forall v\in H^{1}\left(  \mathbb{R}^{3},\mathbb{A}_{j}%
      ^{\operatorname*{ext}}\right)  . \label{uwvp}%
  \end{gather}
\end{minipage}}

\noindent\fbox{\begin{minipage}{0.98\linewidth}
    II. \textit{Mixed variational problem (MVP)}. For given $\psi\in H^{1/2}\left(
\Gamma_{j}\right)  $, find $\mathbf{j}\in\mathbf{H}\left(  \mathbb{R}%
^{3},\operatorname*{div}\right)  $ and $u\in L^{2}\left(  \mathbb{R}%
^{3}\right)  $ such that%
\begin{equation}%
\begin{array}
[c]{llll}%
-\left\langle \left(  \mathbb{A}_{j}^{\operatorname*{ext}}\right)
^{-1}\mathbf{j},\overline{\mathbf{m}}\right\rangle _{\mathbb{R}^{3}} &
-\left\langle u,\operatorname*{div}\overline{\mathbf{m}}\right\rangle
_{\mathbb{R}^{3}} & =\left\langle \psi,\gamma_{\mathbf{n};j}\left(  s\right)
\overline{\mathbf{m}}\right\rangle _{\Gamma_{j}} & \forall\mathbf{m}%
\in\mathbf{H}\left(  \mathbb{R}^{3},\operatorname*{div}\right)  ,\\
&  &  & \\
-\left\langle \operatorname{div}\mathbf{j},\overline{q}\right\rangle
_{\mathbb{R}^{3}} & +s^{2}\left\langle p_{j}^{\operatorname*{ext}}%
u,\overline{q}\right\rangle _{\mathbb{R}^{3}} & =0 & \forall q\in L^{2}\left(
\mathbb{R}^{3}\right)  .
\end{array}
\label{mvp}%
\end{equation}
\end{minipage}}
\medskip

In Lemmas \ref{LemMVPwp} and \ref{LemUWVP} we will prove that the variational problems
(\ref{LemUWVP}) and (\ref{mvp}) are well posed.

\begin{lemma}
\label{LemMVPwp}Let Assumption \ref{Acoeff} be satisfied. The ultra-weak
variation problem (\ref{uwvp}) is well posed.
\end{lemma}

\begin{proof}
We will show that there exist constants $0 < C_{1},C_{2},c_{1}<\infty$ such
that the continuity estimates%
\begin{subequations}
\label{BB}
\begin{align}
\forall u\in L^{2}\left(  \mathbb{R}^{3}\right),\; v\in
H^{1}\left(  \mathbb{R}^{3},\mathbb{A}_{j}^{\operatorname*{ext}}\right)
\quad\left\vert \left\langle u,\mathsf{L}_{j}\left(  s\right)  \overline
{v}\right\rangle _{\mathbb{R}^{3}}\right\vert & \leq C_{1}\left\Vert
u\right\Vert _{L^{2}\left(  \mathbb{R}^{3}\right)  }\left\Vert v\right\Vert
_{H^{1}\left(  \mathbb{R}^{3},\mathbb{A}_{j}^{\operatorname*{ext}}\right)
}, \label{BBa}\\
  \begin{aligned}
    \forall &\psi\in H^{1/2}(\Gamma_{j}),\\
    &v\in H^{1}\left(  \mathbb{R}^{3},\mathbb{A}_{j}^{\operatorname*{ext}%
      }\right)
  \end{aligned}
    \quad   \left\vert \left\langle \psi,\gamma_{\operatorname*{N}%
          ;j}^{\operatorname*{ext}}\left(  s\right)  \overline{v}\right\rangle
      _{\Gamma_{j}}\right\vert & \leq C_{2}\left\Vert \psi\right\Vert_{H^{1/2}(\Gamma)}
    \left\Vert v\right\Vert _{H^{1}\left(
        \mathbb{R}^{3},\mathbb{A}_{j}^{\operatorname*{ext}}\right)  }.
  \label{BBd}%
\end{align}
and the following inf-sup conditions hold:%
\begin{align}
\forall u\in L^{2}\left(  \mathbb{R}^{3}\right)  \quad\exists v\in
H^{1}\left(  \mathbb{R}^{3},\mathbb{A}_{j}^{\operatorname*{ext}}\right)
\quad &  \left\vert \left\langle u,\mathsf{L}_{j}\left(  s\right)
\overline{v}\right\rangle _{\mathbb{R}^{3}}\right\vert \geq c_{1}\left\Vert
u\right\Vert _{L^{2}\left(  \mathbb{R}^{3}\right)  }\left\Vert v\right\Vert
_{H^{1}\left(  \mathbb{R}^{3},\mathbb{A}_{j}^{\operatorname*{ext}}\right)
},
\label{BBb}\\
\forall v\in H^{1}\left(  \mathbb{R}^{3},\mathbb{A}_{j}^{\operatorname*{ext}%
}\right)  \quad &  \left(  \sup_{u\in L^{2}\left(  \mathbb{R}^{3}\right)
}\left\vert \left\langle u,\mathsf{L}_{j}\left(  s\right)  \overline
{v}\right\rangle _{\mathbb{R}^{3}}\right\vert =0\right)  \implies\left(
v=0\right)  .\label{BBc}
\end{align}
\end{subequations}%

The Babu\v{s}ka-Lax-Milgram theorem (also sometimes called
Banach-Nečas-Ba\-buš\-ka theorem) (see, e.g., \cite[Thm. 2.1]%
{BabuskaLaxMilgram} and, e.g., \cite[Thm. 25.9]{ErnGuermondII} for the form we
will apply it) then implies well-posedness of (\ref{uwvp}).

\textbf{@(\ref{BBa}).} The continuity of the sesquilinear form in (\ref{uwvp})
follows from%
\begin{align*}
\left\vert \left\langle u,\mathsf{L}_{j}\left(  s\right)  \overline
{v}\right\rangle _{\mathbb{R}^{3}}\right\vert  &  \leq\left\Vert u\right\Vert
_{L^{2}\left(  \mathbb{R}^{3}\right)  }\left\Vert \mathsf{L}_{j}\left(
                                                 s\right)  \overline{v}\right\Vert _{L^{2}\left(  \mathbb{R}^{3}\right)  }
  \\
  & \leq\left\Vert u\right\Vert _{L^{2}\left(  \mathbb{R}^{3}\right)  }\left\Vert
-\operatorname*{div}\left(  \mathbb{A}_{j}^{\operatorname*{ext}}\nabla\bar
{v}\right)  +s^{2}p_{j}^{\operatorname*{ext}}\bar{v}\right\Vert _{L^{2}\left(
\mathbb{R}^{3}\right)  }\\
&  \leq\sqrt{2}\left\Vert u\right\Vert _{L^{2}\left(  \mathbb{R}^{3}\right)
}\left(  \left\Vert \operatorname*{div}\left(  \mathbb{A}_{j}%
^{\operatorname*{ext}}\nabla\bar{v}\right)  \right\Vert _{L^{2}\left(
\mathbb{R}^{3}\right)  }^{2}+\left\vert s\right\vert ^{4}\Lambda_{j}%
^{2}\left\Vert v\right\Vert _{L^{2}\left(  \mathbb{R}^{3}\right)  }%
^{2}\right)  ^{1/2}\\
&  \leq C_{1}\left\Vert u\right\Vert _{L^{2}\left(  \mathbb{R}^{3}\right)
}\left\Vert v\right\Vert _{H^{1}\left(  \mathbb{R}^{3},\mathbb{A}%
_{j}^{\operatorname*{ext}}\right)  }%
\end{align*}
for $C_{1}=\sqrt{2}\max\left\{  1,\left\vert s\right\vert ^{2}\Lambda
_{j}\right\}  $.

\textbf{@(\ref{BBd}). }It is a simple consequence of the mapping properties of
the trace operator that the right-hand side in (\ref{uwvp}) $\left\langle
\psi,\gamma_{\operatorname*{N};j}^{\operatorname*{ext}}\left(  s\right)
\overline{\cdot}\right\rangle _{\Gamma_{j}}$ defines a continuous functional
on $H^{1}\left(  \mathbb{R}^{3},\mathbb{A}_{j}^{\operatorname*{ext}}\right)  $
so that (\ref{BBd}) follows.%

\textbf{@(\ref{BBb}).} We choose the test function in (\ref{uwvp}) as
$v\leftarrow\mathsf{N}_{j}\left(  \overline{s}\right)  u$. It is easy to
deduce from Definition \ref{DefSolOp} that $\overline{\mathsf{N}_{j}\left(
\overline{s}\right)  u}=\mathsf{N}_{j}\left(  s\right)  \overline{u}$ holds so
that%
\[
\left\langle u,\mathsf{L}_{j}\left(  s\right)  \overline{\mathsf{N}_{j}\left(
\overline{s}\right)  u}\right\rangle _{\mathbb{R}^{3}}=\left\langle
u,\mathsf{L}_{j}\left(  s\right)  \mathsf{N}_{j}\left(  s\right)  \overline
{u}\right\rangle _{\mathbb{R}^{3}}=\left\Vert u\right\Vert _{L^{2}\left(
\mathbb{R}^{3}\right)  }^{2}.
\]
Hence, the inf-sup constant for problem (\ref{uwvp}) can be estimated from
below by%
\begin{align*}
\inf_{u\in L^{2}\left(  \mathbb{R}^{3}\right)  \backslash\left\{  0\right\}
}\sup_{v\in H^{1}\left(  \mathbb{R}^{3},\mathbb{A}_{j}^{\operatorname*{ext}%
}\right)  \backslash\left\{  0\right\}  } & \frac{\left\langle u,\mathsf{L}%
_{j}\left(  s\right)  \overline{v}\right\rangle _{\mathbb{R}^{3}}}{\left\Vert
u\right\Vert _{L^{2}\left(  \mathbb{R}^{3}\right)  }\left\Vert v\right\Vert
_{H^{1}\left(  \mathbb{R}^{3},\mathbb{A}_{j}^{\operatorname*{ext}}\right)  }%
  }\\
  & \geq\inf_{u\in L^{2}\left(  \mathbb{R}^{3}\right)  \backslash\left\{
0\right\}  }\frac{\left\Vert u\right\Vert _{L^{2}\left(  \mathbb{R}%
^{3}\right)  }}{\left\Vert \mathsf{N}_{j}\left(  \overline{s}\right)
u\right\Vert _{H^{1}\left(  \mathbb{R}^{3},\mathbb{A}_{j}^{\operatorname*{ext}%
}\right)  }}.
\end{align*}
We estimate the denominator by%
\begin{align*}
\left\Vert \mathsf{N}_{j}\left(  \overline{s}\right)  u\right\Vert
_{H^{1}\left(  \mathbb{R}^{3},\mathbb{A}_{j}^{\operatorname*{ext}}\right)
}^{2} &  =\left\Vert \operatorname*{div}\left(  \mathbb{A}_{j}%
^{\operatorname*{ext}}\nabla\mathsf{N}_{j}\left(  \overline{s}\right)
u\right)  \right\Vert _{L^{2}\left(  \mathbb{R}^{3}\right)  }^{2}+\left\Vert
\mathsf{N}_{j}\left(  \overline{s}\right)  u\right\Vert _{H^{1}\left(
\mathbb{R}^{3}\right)  }^{2}\\
&  =\left\Vert \mathsf{L}_{j}\left(  \overline{s}\right)  \mathsf{N}%
_{j}\left(  \overline{s}\right)  u-\bar{s}^{2}p_{j}^{\operatorname*{ext}%
}\mathsf{N}_{j}\left(  \overline{s}\right)  u\right\Vert _{L^{2}\left(
\mathbb{R}^{3}\right)  }^{2}+\left\Vert \mathsf{N}_{j}\left(  \overline
{s}\right)  u\right\Vert _{H^{1}\left(  \mathbb{R}^{3}\right)  }^{2}\\
&  \leq2\left\Vert \mathsf{L}_{j}\left(  \overline{s}\right)  \mathsf{N}%
_{j}\left(  \overline{s}\right)  u\right\Vert _{L^{2}\left(  \mathbb{R}%
^{3}\right)  }^{2}+2\left\vert s\right\vert ^{4}\Lambda_{j}^{2}\left\Vert
\mathsf{N}_{j}\left(  \overline{s}\right)  u\right\Vert _{L^{2}\left(
  \mathbb{R}^{3}\right)  }^{2}\\
  &\quad+\left\Vert \mathsf{N}_{j}\left(  \overline
{s}\right)  u\right\Vert _{H^{1}\left(  \mathbb{R}^{3}\right)  }^{2}\\
&  \leq2\left\Vert u\right\Vert _{L^{2}\left(  \mathbb{R}^{3}\right)  }%
^{2}+\left\vert s\right\vert ^{2}C_{0}^{2}\left\Vert \mathsf{N}_{j}\left(
\overline{s}\right)  u\right\Vert _{H^{1}\left(  \mathbb{R}^{3}\right)
;s}^{2}%
\end{align*}
for $C_{0}:=\max\left\{  \sqrt{2\Lambda_{j}^{2}+\frac{1}{s_{0}^{4}}}%
,s_{0}^{-1}\right\}  $. From (\ref{Njest}) we get%
\begin{align*}
\left\Vert \mathsf{N}_{j}\left(  \overline{s}\right)  u\right\Vert
_{H^{1}\left(  \mathbb{R}^{3}\right)  ;s} &  \leq\frac{\left\vert s\right\vert
}{\lambda_{j}\operatorname{Re}s}\left\Vert u\right\Vert _{H^{-1}\left(
                                            \mathbb{R}^{3}\right)  ;s}\\
                                          &\leq\frac{\left\vert s\right\vert }{\lambda
_{j}\operatorname{Re}s}\left(  \sup_{g\in H^{1}\left(  \mathbb{R}^{3}\right)
\backslash\left\{  0\right\}  }\frac{\left\Vert g\right\Vert _{L^{2}\left(
\mathbb{R}^{3}\right)  }}{\left\Vert g\right\Vert _{H^{1}\left(
\mathbb{R}^{3}\right)  ;s}}\right)  \left\Vert u\right\Vert _{L^{2}\left(
\mathbb{R}^{3}\right)  }\\
&  \leq\frac{1}{\lambda_{j}\operatorname{Re}s}\left\Vert u\right\Vert
_{L^{2}\left(  \mathbb{R}^{3}\right)  }%
\end{align*}
and, in turn,%
\[
\left\Vert \mathsf{N}_{j}\left(  \overline{s}\right)  u\right\Vert
_{H^{1}\left(  \mathbb{R}^{3},\mathbb{A}_{j}^{\operatorname*{ext}}\right)
}\leq\left(  2+\frac{C_{0}^{2}\left\vert s\right\vert ^{2}}{\lambda_{j}%
^{2}\left(  \operatorname{Re}s\right)  ^{2}}\right)  ^{1/2}\left\Vert
u\right\Vert _{L^{2}\left(  \mathbb{R}^{3}\right)  }.
\]
The combination of these estimates leads to the inf-sup estimate%
\[
\inf_{u\in L^{2}\left(  \mathbb{R}^{3}\right)  \backslash\left\{  0\right\}
}\sup_{v\in H^{1}\left(  \mathbb{R}^{3},\mathbb{A}_{j}^{\operatorname*{ext}%
}\right)  \backslash\left\{  0\right\}  }\frac{\left\langle u,\mathsf{L}%
_{j}\left(  s\right)  \overline{v}\right\rangle _{\mathbb{R}^{3}}}{\left\Vert
u\right\Vert _{L^{2}\left(  \mathbb{R}^{3}\right)  }\left\Vert v\right\Vert
_{H^{1}\left(  \mathbb{R}^{3},\mathbb{A}_{j}^{\operatorname*{ext}}\right)  }%
}\geq c_{1}\frac{\operatorname{Re}s}{\left\vert s\right\vert },
\]
where $c_{1}>0$ only depends on $\lambda_{j},\Lambda_{j}$, $s_{0}$.$\ $

\textbf{@(\ref{BBc}).} We choose $u=\overline{\mathsf{L}_{j}\left(  s\right)
}v$ and obtain%
\begin{equation}
\sup_{u\in L^{2}\left(  \mathbb{R}^{3}\right)  }\left\vert \left\langle
u,\mathsf{L}_{j}\left(  s\right)  \overline{v}\right\rangle _{\mathbb{R}^{3}%
}\right\vert \geq\left\vert \left\langle \overline{\mathsf{L}_{j}\left(
s\right)  }v,\mathsf{L}_{j}\left(  s\right)  \overline{v}\right\rangle
_{\mathbb{R}^{3}}\right\vert =\left\Vert \mathsf{L}_{j}\left(  s\right)
\overline{v}\right\Vert _{L^{2}\left(  \mathbb{R}^{3}\right)  }^{2}.
\label{CondcBB}%
\end{equation}
Since $\mathsf{L}_{j}\left(  s\right)  :H^{1}\left(  \mathbb{R}^{3}\right)
\rightarrow H^{-1}\left(  \mathbb{R}^{3}\right)  $ is an isomorphism (see
(\ref{mappropNjs})), the implication $\mathsf{L}_{j}\left(  s\right)
\overline{v}=0\implies v=0$ holds for all $v\in H^{1}\left(  \mathbb{R}%
^{3}\right)  $. Since $H^{1}\left(  \mathbb{R}^{3},\mathbb{A}_{j}%
^{\operatorname*{ext}}\right)  \subset H^{1}\left(  \mathbb{R}^{3}\right)  $
we conclude from (\ref{CondcBB}) that (\ref{BBc}) holds.
\end{proof}

\begin{lemma}
\label{LemUWVP}Let Assumption \ref{Acoeff} be satisfied. The mixed variational
problem\ (\ref{mvp}) is well posed.
\end{lemma}

\begin{proof}
Again, we employ the Babu\v{s}ka-Lax-Milgram theorem and prove the relevant
properties for the sesquilinear form and anti-linear form associated with
(\ref{mvp}). The sesquilinear form $b:\left(  \mathbf{H}\left(  \mathbb{R}%
^{3},\operatorname*{div}\right)  ,L^{2}\left(  \mathbb{R}^{3}\right)  \right)
\times\left(  \mathbf{H}\left(  \mathbb{R}^{3},\operatorname*{div}\right)
,L^{2}\left(  \mathbb{R}^{3}\right)  \right)  \rightarrow\mathbb{C}$ related
to the mixed variational problem (\ref{mvp}) is given by%
\[
b\left(  \left(  \mathbf{j},u\right)  ,\left(  \mathbf{m},v\right)  \right)
:=-\left\langle \left(  \mathbb{A}_{j}^{\operatorname*{ext}}\right)
^{-1}\mathbf{j},\overline{\mathbf{m}}\right\rangle _{\mathbb{R}^{3}%
}-\left\langle u,\operatorname*{div}\overline{\mathbf{m}}\right\rangle
_{\mathbb{R}^{3}}-\left\langle \operatorname{div}\mathbf{j},\overline
{v}\right\rangle _{\mathbb{R}^{3}}+s^{2}\left\langle p_{j}%
^{\operatorname*{ext}}u,\overline{v}\right\rangle _{\mathbb{R}^{3}}.
\]
The anti-linear form associated to the right-hand side is $f:\left(
\mathbf{H}\left(  \mathbb{R}^{3},\operatorname*{div}\right)  ,L^{2}\left(
\mathbb{R}^{3}\right)  \right)  \rightarrow\mathbb{C}$%
\[
f\left(  \left(  \mathbf{m},v\right)  \right)  :=\left\langle \psi
,\gamma_{\mathbf{n};j}\left(  s\right)  \overline{\mathbf{m}}\right\rangle
_{\Gamma_{j}}.
\]
We will verify the four conditions for the Babu\v{s}ka-Lax-Milgram theorem.
The continuity of $b$ follows by straightforward Cauchy-Schwarz inequalities.
For the analogue of (\ref{BBb}) we choose
\begin{equation}
v\leftarrow\frac{s}{\left\vert s\right\vert ^{3}}u-\frac{s}{\left\vert
s\right\vert ^{3}}\frac{1}{p_{j}^{\operatorname*{ext}}}\operatorname*{div}%
\mathbf{j\quad}\text{and\quad}\mathbf{m}\leftarrow-\frac{\bar{s}}{\left\vert
s\right\vert }\left(  1+\frac{1}{\left\vert s\right\vert ^{2}}\right)
\mathbf{j} \label{substitution}
\end{equation}
and obtain after some straightforward manipulations%
\begin{align*}
b\left(  \left(  \mathbf{j},u\right)  ,\left(  \mathbf{m},v\right)  \right)
&  =\frac{s}{\left\vert s\right\vert }\left(  1+\frac{1}{\left\vert
s\right\vert ^{2}}\right)  \left\langle \left(  \mathbb{A}_{j}%
^{\operatorname*{ext}}\right)  ^{-1}\mathbf{j},\overline{\mathbf{j}%
}\right\rangle _{\mathbb{R}^{3}}+\frac{\bar{s}}{\left\vert s\right\vert ^{3}%
}\left\langle \frac{1}{p_{j}^{\operatorname*{ext}}}\operatorname{div}%
\mathbf{j},\overline{\operatorname*{div}\mathbf{j}}\right\rangle
_{\mathbb{R}^{3}}\\
&  +2\operatorname*{i}\operatorname{Im}\left(  \frac{s}{\left\vert
s\right\vert ^{3}}\left\langle u,\operatorname*{div}\overline{\mathbf{j}%
}\right\rangle _{\mathbb{R}^{3}}\right)  +\frac{s}{\left\vert s\right\vert }\left\langle p_{j}^{\operatorname*{ext}%
}u,\overline{u}\right\rangle _{\mathbb{R}^{3}}.
\end{align*}
Hence,
\begin{align*}
\left\vert b\left(  \left(  \mathbf{j},u\right)  ,\left(  \mathbf{m},v\right)
\right)  \right\vert  &  \geq\operatorname{Re}b\left(  \left(  \mathbf{j}%
,u\right)  ,\left(  \mathbf{m},v\right)  \right) \\
&  \geq\frac{\operatorname{Re}s}{\Lambda_{j}\left\vert s\right\vert }\left(
1+\frac{1}{\left\vert s\right\vert ^{2}}\right)  \left\Vert \mathbf{j}%
\right\Vert _{\mathbf{L}^{2}\left(  \mathbb{R}^{3}\right)  }^{2}%
+\frac{\operatorname{Re}s}{\Lambda_{j}\left\vert s\right\vert ^{3}}\left\Vert
\operatorname{div}\mathbf{j}\right\Vert _{L^{2}\left(  \mathbb{R}^{3}\right)
  }^{2}\\
  & \quad+\frac{\operatorname{Re}s}{\left\vert s\right\vert }\lambda_{j}\left\Vert
u\right\Vert _{L^{2}\left(  \mathbb{R}^{3}\right)  }^{2}.
\end{align*}
From this, the estimate%
\[
\left\vert b\left(  \left(  \mathbf{j},u\right)  ,\left(  \mathbf{m},v\right)
\right)  \right\vert \geq\frac{\operatorname{Re}s}{\left\vert s\right\vert
^{3}}\min\left\{  \frac{1}{\Lambda_{j}},s_{0}^{2}\lambda_{j}\right\}  \left(
\left\Vert \mathbf{j}\right\Vert _{\mathbf{H}\left(  \mathbb{R}^{3}%
,\operatorname*{div}\right)  }^{2}+\left\Vert u\right\Vert _{L^{2}\left(
\mathbb{R}^{3}\right)  }^{2}\right)
\]
follows. The choice (\ref{substitution}) can be bounded by%
\begin{align*}
\left\Vert \mathbf{m}\right\Vert _{\mathbf{H}\left(  \mathbb{R}^{3}%
,\operatorname*{div}\right)  }^{2}+\left\Vert v\right\Vert _{L^{2}\left(
\mathbb{R}^{3}\right)  }^{2}  &  \leq\left(  1+\frac{1}{\left\vert
s\right\vert ^{2}}\right)  ^{2}\left\Vert \mathbf{j}\right\Vert _{\mathbf{H}%
                                \left(  \mathbb{R}^{3},\operatorname*{div}\right)  }^{2}\\
  &\quad +\frac{2}{\left\vert
s\right\vert ^{4}}\left(  \left\Vert u\right\Vert _{L^{2}\left(
\mathbb{R}^{3}\right)  }^{2}+\frac{1}{\lambda_{j}^{2}}\left\Vert
\mathbf{j}\right\Vert _{\mathbf{H}\left(  \mathbb{R}^{3},\operatorname*{div}%
\right)  }^{2}\right) \\
&  \leq C_{0}\left(  \left\Vert \mathbf{j}\right\Vert _{\mathbf{H}\left(
\mathbb{R}^{3},\operatorname*{div}\right)  }^{2}+\left\Vert u\right\Vert
_{L^{2}\left(  \mathbb{R}^{3}\right)  }^{2}\right)
\end{align*}
for a positive constant $C_{0}$ which depends solely on $s_{0}$ and
$\lambda_{j}$. This leads to%
\[
\left\vert b\left(  \left(  \mathbf{j},u\right)  ,\left(  \mathbf{m},v\right)
\right)  \right\vert \geq c_{1}\left(  \left\Vert \mathbf{j}\right\Vert
_{\mathbf{H}\left(  \mathbb{R}^{3},\operatorname*{div}\right)  }%
^{2}+\left\Vert u\right\Vert _{L^{2}\left(  \mathbb{R}^{3}\right)  }%
^{2}\right)  ^{1/2}\left(  \left\Vert \mathbf{m}\right\Vert _{\mathbf{H}%
\left(  \mathbb{R}^{3},\operatorname*{div}\right)  }^{2}+\left\Vert
v\right\Vert _{L^{2}\left(  \mathbb{R}^{3}\right)  }^{2}\right)  ^{1/2}.
\]

Next, we prove the analogue of (\ref{BBc}). Let $\left(  \mathbf{m},v\right)
\in\left(  \mathbf{H}\left(  \mathbb{R}^{3},\operatorname*{div}\right)
,L^{2}\left(  \mathbb{R}^{3}\right)  \right)  $ and assume%
\begin{equation}
\forall\left(  \mathbf{j},u\right)  \in\left(  \mathbf{H}\left(
\mathbb{R}^{3},\operatorname*{div}\right)  ,L^{2}\left(  \mathbb{R}%
^{3}\right)  \right)  \qquad b\left(  \left(  \mathbf{j},u\right)  ,\left(
\mathbf{m},v\right)  \right)  =0. \label{CondCBiliB}%
\end{equation}
The analogous choice to (\ref{substitution}) for the primal variables $\left(
\mathbf{j},u\right)  $ is%
\[
u\leftarrow\frac{\bar{s}}{\left\vert s\right\vert ^{3}}v-\frac{\bar{s}%
}{\left\vert s\right\vert ^{3}}\frac{1}{p_{j}^{\operatorname*{ext}}%
}\operatorname*{div}\mathbf{m\quad}\text{and\quad}\mathbf{j}=-\frac
{s}{\left\vert s\right\vert }\left(  1+\frac{1}{\left\vert s\right\vert ^{2}%
}\right)  \mathbf{m}%
\]
and we obtain in the same way as before%
\begin{align*}
b\left(  \left(  \mathbf{j},u\right)  ,\left(  \mathbf{m},v\right)  \right)
&  =\frac{s}{\left\vert s\right\vert }\left(  1+\frac{1}{\left\vert
s\right\vert ^{2}}\right)  \left\langle \left(  \mathbb{A}_{j}%
^{\operatorname*{ext}}\right)  ^{-1}\mathbf{m},\overline{\mathbf{m}%
}\right\rangle _{\mathbb{R}^{3}}+\frac{\bar{s}}{\left\vert s\right\vert ^{3}%
}\left\langle \frac{1}{p_{j}^{\operatorname*{ext}}}\operatorname*{div}%
\mathbf{m},\operatorname*{div}\overline{\mathbf{m}}\right\rangle
_{\mathbb{R}^{3}}\\
&  +2\operatorname*{i}\operatorname{Im}\left(  \frac{s}{\left\vert
s\right\vert ^{3}}\left\langle \operatorname{div}\mathbf{m},\overline
{v}\right\rangle _{\mathbb{R}^{3}}\right) \\
&  +\frac{s}{\left\vert s\right\vert }\left\langle p_{j}^{\operatorname*{ext}%
}v,\overline{v}\right\rangle _{\mathbb{R}^{3}}.
\end{align*}
For the real part the estimate%
\[
\operatorname{Re}b\left(  \left(  \mathbf{j},u\right)  ,\left(  \mathbf{m}%
,v\right)  \right)  \geq\frac{\operatorname{Re}s}{\left\vert s\right\vert
^{3}}\min\left\{  \frac{1}{\Lambda_{\min}},s_{0}^{2}\lambda_{\min}\right\}
\left(  \left\Vert \mathbf{m}\right\Vert _{\mathbf{H}\left(  \mathbb{R}%
^{3},\operatorname*{div}\right)  }^{2}+\left\Vert v\right\Vert _{L^{2}\left(
\mathbb{R}^{3}\right)  }^{2}\right)
\]
follows. In view of (\ref{CondCBiliB}), $\left(  \mathbf{m},v\right)  =\left(
\mathbf{0},0\right)  $ follows.

The continuity of the anti-linear form $f$ follows by combining a
Cauchy-Schwarz inequality%
\[
\left\vert f\left(  \left(  \mathbf{m},v\right)  \right)  \right\vert
\leq\left\Vert \psi\right\Vert _{H^{1/2}\left(  \Gamma_{j}\right)  }\left\vert
s\right\vert ^{-1/2}\left\Vert \gamma_{\mathbf{n};j}(s)\left(  \mathbf{m}\right)
\right\Vert _{H^{-1/2}\left(  \Gamma_{j}\right)  }
\]
with the estimate (\ref{normaltracebounded}) for the normal trace.
\end{proof}

The next lemma states an equivalence of the solutions of (\ref{uwvp}) and (\ref{mvp}).

\begin{lemma}\label{Propequiuwvpmvp}
  Let Assumption \ref{Acoeff} be satisfied.
  The mixed variational problem\ (\ref{mvp}) and the ultra-weak variational problem (\ref{uwvp}) are equivalent:
  \begin{enumerate}
    \item If $\left(  \mathbf{j},u\right)  \in\left(  \mathbf{H}\left(
        \mathbb{R}^3,\operatorname*{div}\right)  ,L^{2}\left(  \mathbb{R}%
        ^{3}\right)  \right)  $ is the solution of (\ref{mvp}), then $u$ solves
    (\ref{uwvp}).
    \item If $u$ is the solution of (\ref{uwvp}), then the pair \modified{$(\mathbf{j}, u) := \left( \mathbb{A}_{j}^{\operatorname*{ext}}\nabla_{\operatorname*{pw};j}u,u \right)$} solves (\ref{mvp}).
    In particular, it holds
  $\mathbf{j}\in\mathbf{H}\left( \mathbb{R}^{3},\operatorname{div}\right) $.
\item The solution $u$ of the ultra-weak variational problem satisfies the
jump relation%
\begin{equation}
\left[  u\right]  _{\operatorname*{D};j}\left(  s\right)  =\psi.
\label{jumpreldlp}%
\end{equation}
\end{enumerate}
\end{lemma}%

\begin{proof}
\textbf{Part 1. }

Let $\left(  \mathbf{j},u\right)  \in\left(  \mathbf{H}\left(  \mathbb{R}%
^{3},\operatorname*{div}\right)  ,L^{2}\left(  \mathbb{R}^{3}\right)  \right)
$ be the solution of (\ref{mvp}). We test the first equation in (\ref{mvp}) with
$\mathbf{m}:=\mathbb{A}_{j}^{\operatorname*{ext}}\nabla v$ for $v\in
H^{1}\left(  \mathbb{R}^{3},\mathbb{A}_{j}^{\operatorname*{ext}}\right)  $.
Clearly, $\mathbf{m}\in\mathbf{H}\left(  \mathbb{R}^{3},\operatorname*{div}%
\right)  $ is an admissible test function. This leads to%
\[
-\left\langle \mathbf{j},\nabla\overline{v}\right\rangle _{\mathbb{R}^{3}%
}-\left\langle u,\operatorname*{div}\left(  \mathbb{A}_{j}%
^{\operatorname*{ext}}\nabla\overline{v}\right)  \right\rangle _{\mathbb{R}%
^{3}}=\left\langle \psi,\gamma_{\operatorname*{N};j}^{\operatorname*{ext}%
}\left(  s\right)  \overline{v}\right\rangle _{\Gamma_{j}}\quad\forall v\in
H^{1}\left(  \mathbb{R}^{3},\mathbb{A}_{j}^{\operatorname*{ext}}\right)  .
\]
Next, we test the second equation in (\ref{mvp}) with $q\in H^{1}\left(
\mathbb{R}^{3},\mathbb{A}_{j}^{\operatorname*{ext}}\right)  $ and integrate by
parts%
\[
\left\langle \mathbf{j},\nabla\overline{q}\right\rangle _{\mathbb{R}^{3}%
}+s^{2}\left\langle p_{j}^{\operatorname*{ext}}u,\overline{q}\right\rangle
_{\mathbb{R}^{3}}=0\quad\forall q\in H^{1}\left( \mathbb{R}^{3},\mathbb{A}_{j}^{\operatorname*{ext}}\right)  .
\]
We set $q=v$ and sum both equations, which yields
\[
\left\langle u,\mathsf{L}_{j}\left(  s\right)  \overline{v}\right\rangle
_{\mathbb{R}^{3}}=\left\langle \psi,\gamma_{\operatorname*{N};j}%
^{\operatorname*{ext}}\left(  s\right)  \overline{v}\right\rangle _{\Gamma
_{j}}\quad\forall v\in H^{1}\left(  \mathbb{R}^{3},\mathbb{A}_{j}%
^{\operatorname*{ext}}\right)  .
\]
Hence, the solution $u$ of the mixed variational problem (\ref{mvp}) solves
the ultra-weak problem (\ref{uwvp}). Lemma \ref{LemMVPwp} implies uniqueness
of solutions of (\ref{mvp}) so that $u$ is the unique solution of (\ref{uwvp}).

Now, we test the first equation in (\ref{mvp}) with functions $\mathbf{m}%
\in\mathbf{C}_{0}^{\infty}\left(  \mathbb{R}^{3}\right)  $ satisfying
$\operatorname*{supp}\left(  \mathbf{m}\right)  \subset\subset\Omega
_{j}^{\sigma}$ for some $\sigma\in\left\{  +,-\right\}  $. This leads to
$\nabla_{\operatorname*{pw};j}u=\left(  \mathbb{A}_{j}^{\operatorname*{ext}%
}\right)  ^{-1}\mathbf{j}\in\mathbf{L}^{2}\left(  \mathbb{R}^{3}\right)  $
and, in turn, to $u\in H^{1}\left(  \mathbb{R}^{3}\backslash\Gamma_{j}\right)
$.\medskip

\textbf{Part 2.}

Lemma \ref{LemMVPwp} and \ref{LemUWVP} imply the existence and uniqueness of
solutions for the variational problems (\ref{uwvp}) and (\ref{mvp}). For
$\psi\in H^{1/2}\left(  \Gamma_{j}\right)  $, let $u_{\operatorname*{uw}}$
denote the solution of (\ref{uwvp}) and $\left(  \mathbf{j}_{\operatorname*{m}%
},u_{\operatorname*{m}}\right)  $ the solution of (\ref{mvp}). Part 1 implies
that $u_{\operatorname*{m}}\in H^{1}\left(  \mathbb{R}^{3}\backslash\Gamma
_{j}\right)  $ solves the ultra-weak problem so that $u_{\operatorname*{uw}%
}=u_{\operatorname*{m}}$. Vice versa, $u_{\operatorname*{uw}}$ equals the
$u_{\operatorname*{m}}$-component of the solution for the mixed variational
problem. We test the first equation in (\ref{mvp}) with test functions
$\mathbf{m}\in\mathbf{H}\left(  \mathbb{R}^{3},\operatorname*{div}\right)  $
with compact support in $\Omega_{j}^{-}\cup\Omega_{j}^{+}$ and obtain by
integration by parts%
\[
\mathbf{j}_{\operatorname*{m}}=\mathbb{A}_{j}^{\operatorname*{ext}}%
\nabla_{\operatorname*{pw};j}u_{\operatorname*{m}}=\mathbb{A}_{j}%
^{\operatorname*{ext}}\nabla_{\operatorname*{pw};j}u_{\operatorname*{uw}}.
\]
Since $\mathbf{j}_{\operatorname*{m}}\in\mathbf{H}\left(  \mathbb{R}%
^{3},\operatorname{div}\right)  $ it follows that $\left(  \mathbb{A}%
_{j}^{\operatorname*{ext}}\nabla_{\operatorname*{pw};j}u_{\operatorname*{uw}%
},u_{\operatorname*{uw}}\right)  \in\mathbf{H}\left(  \mathbb{R}%
^{3},\operatorname{div}\right)  \times L^{2}\left(  \mathbb{R}^{3}\right)  $
solves the mixed variational problem.\medskip

\textbf{Part 3.}

We consider the first equation of the mixed problem (\ref{mvp}) and employ
$\mathbf{j}=\mathbb{A}_{j}^{\operatorname*{ext}}\nabla_{\operatorname*{pw}%
;j}u$. Integration by parts in each subdomain yields%
\begin{align}
\left\langle \psi,\gamma_{\mathbf{n};j}\left(  s\right)  \overline{\mathbf{m}%
}\right\rangle _{\Gamma_{j}}  &  =-\left\langle \left(  \mathbb{A}%
_{j}^{\operatorname*{ext}}\right)  ^{-1}\mathbf{j},\overline{\mathbf{m}%
}\right\rangle _{\mathbb{R}^{3}}-\left\langle u,\operatorname*{div}%
                                \overline{\mathbf{m}}\right\rangle _{\mathbb{R}^{3}}\\
  &=-\left\langle
\nabla_{\operatorname*{pw};j}u,\overline{\mathbf{m}}\right\rangle
_{\mathbb{R}^{3}}-\left\langle u,\operatorname*{div}\overline{\mathbf{m}%
}\right\rangle _{\mathbb{R}^{3}}\nonumber\\
&  =-\left\langle \nabla_{\operatorname*{pw};j}u,\overline{\mathbf{m}%
}\right\rangle _{\mathbb{R}^{3}}+\left\langle \nabla u,\overline{\mathbf{m}%
}\right\rangle _{\mathbb{R}^{3}\backslash\Gamma_{j}}+\left\langle \left[
u\right]  _{\operatorname*{D};j}\left(  s\right)  ,\gamma_{\mathbf{n}%
;j}\left(  s\right)  \left(  \overline{\mathbf{m}}\right)  \right\rangle
_{\Gamma_{j}}\nonumber\\
&  =\left\langle \left[  u\right]  _{\operatorname*{D};j}\left(  s\right)
,\gamma_{\mathbf{n};j}\left(  s\right)  \left(  \overline{\mathbf{m}}\right)
\right\rangle _{\Gamma_{j}}\quad\forall\mathbf{m}\in\mathbf{H}\left(
\mathbb{R}^{3},\operatorname*{div}\right)  . \label{partinlasteq}%
\end{align}
The range of the normal trace is $H^{-1/2}\left(  \Gamma_{j}\right)
=\gamma_{\mathbf{n};j}\left(  s\right)  \left(  \mathbf{H}\left(
\mathbb{R}^{3},\operatorname*{div}\right)  \right)  $ (cf. \cite[Cor.
2.8]{Girault86}) so that the jump relation (\ref{jumpreldlp}) follows from
(\ref{partinlasteq}).
\end{proof}

The well-posedness of the ultra-weak variational problem allows us to define
the double layer potential as its solution.

\begin{definition}
Let Assumption \ref{Acoeff} be satisfied. For $1\leq j\leq n_{\Omega}$ and
$\psi\in H^{1/2}\left(  \Gamma_{j}\right)  $ the \emph{\bf double layer potential}
$\mathsf{D}_{j}\left(  s\right)  \psi\in L^{2}\left(  \mathbb{R}^{3}\right)  $
is given as the unique solution of the ultra-weak variational problem%
\begin{equation}
  \boxed{\left\langle \mathsf{D}_{j}\left(  s\right)  \psi,\mathsf{L}_{j}\left(
s\right)  \overline{v}\right\rangle _{\mathbb{R}^{3}}=\left\langle \psi
,\gamma_{\operatorname*{N};j}^{\operatorname*{ext}}\left(  s\right)
\overline{v}\right\rangle _{\Gamma_{j}}\qquad\forall v\in H^{1}\left(
\mathbb{R}^{3},\mathbb{A}_{j}^{\operatorname*{ext}}\right)}  . \label{defDLP}%
\end{equation}
\end{definition}

\begin{remark}
  Note that our definition (\ref{defDLP}) has the same form as formula (4.7) in
  \cite{CostabelElemRes}. However, we employ this directly as the definition while, in
  \cite{CostabelElemRes} (where the coefficients are assumed to be infinitely smooth) a
  different definition is used and (\ref{defDLP}) is deduced as an intermediate step
  within the proof of the jump relations.
\end{remark}

In the following lemma, important properties of $\mathsf{D}_{j}\left(
s\right)  $ are collected which are well-known, e.g., for PDEs with piecewise
constant coefficients.

\begin{lemma}
\label{LemMapPropDLP}Let Assumption \ref{Acoeff} be satisfied. For $\psi\in
H^{1/2}\left(  \Gamma\right)  $, the double layer potential $w:=\mathsf{D}%
_{j}\left(  s\right)  \psi$ satisfies $w\in H^{1}\left(  \mathbb{R}%
^{3}\backslash\Gamma_{j},\mathbb{A}_{j}^{\operatorname*{ext}}\right)  $, the
restrictions $w^{\sigma}:=\left.  w\right\vert _{\Omega_{j}^{\sigma}}$ solve
the homogeneous equations:%
\begin{equation}
\mathsf{L}_{j}^{\sigma}\left(  s\right)  w^{\sigma}=0\quad\text{in }\Omega
_{j}^{\sigma}\text{, }\sigma\in\left\{  +,-\right\}  , \label{DlpProp1}%
\end{equation}
and the jump relations hold:%
\begin{equation}%
  \boxed{
\begin{array}
[c]{ll}%
\left[  \left(  \mathsf{D}_{j}\left(  s\right)  \psi\right)  \right]
_{\operatorname*{D};j}\left(  s\right)  =\psi, & \left[  \left(
\mathsf{D}_{j}\left(  s\right)  \psi\right)  \right]  _{\operatorname*{N}%
;j}^{\operatorname*{ext}}\left(  s\right)  =0
\end{array}}.
\label{jumprelDLP}%
\end{equation}

\end{lemma}%

In fact, the double layer potential is a continuous operator $\mathsf{D}_{j}:H^{1/2}\left(
  \Gamma_{j}\right)\to
H^{1}(\mathbb{R}^{3}\setminus\Gamma_{j},\mathbb{A}_{j}^{\operatorname*{ext}})$. 

\begin{proof}
From Lemma \ref{Propequiuwvpmvp} we conclude that the pair $\left(
\mathbf{j},w\right)  $ with $\mathbf{j}:=\mathbb{A}_{j}^{\operatorname*{ext}%
}\nabla_{\operatorname*{pw};j}w$ solves the mixed variational formulation
(\ref{mvp}). We insert this into the second equation of (\ref{mvp}) and test
with functions $q\in L^{2}\left(  \mathbb{R}^{3}\right)  $ which vanish in a
neighborhood of $\Gamma_{j}$. From Lemma \ref{Propequiuwvpmvp}(2) it follows
$w\in H^{1}\left(  \mathbb{R}^{3}\backslash\Gamma_{j},\mathbb{A}%
_{j}^{\operatorname*{ext}}\right)  $ and $w$ satisfies (\ref{DlpProp1}). Again
from Lemma \ref{Propequiuwvpmvp} it follows $\mathbf{j}\in\mathbf{H}\left(
\mathbb{R}^{3},\operatorname*{div}\right)  $ so that $\left[  \left\langle
\mathbf{j},\mathbf{n}_{j}\right\rangle \right]  _{\operatorname*{D};j}=0$. We conclude
$\left[  \left(  \mathsf{D}_{j}\left(  s\right)  \psi\right)
\right]  _{\operatorname*{N};j}^{\operatorname*{ext}}\left(  s\right)  =0$.
Finally, we insert $\mathbf{j}$ into the first equation and substitute
$u\leftarrow w$. Integrating by parts over $\Omega_{j}^{-}$ and $\Omega
_{j}^{+}$ leads to%
\[
\left\langle \left[  w\right]  _{\operatorname*{D};j}\left(  s\right)  ,%
\gamma_{\mathbf{n};j}\left(  s\right)  \overline{\mathbf{m}}%
\right\rangle _{\mathbb{R}^{3}}=\left\langle \psi,%
\gamma_{\mathbf{n};j}\left(  s\right)  \overline{\mathbf{m}}%
\right\rangle _{\Gamma_{j}}\quad\forall\mathbf{m}\in\mathbf{H}\left(
\mathbb{R}^{3},\operatorname*{div}\right)  .
\]
Since the mapping $%
\gamma_{\mathbf{n};j}:\mathbf{H}\left(  \mathbb{R}^{3},\operatorname*{div}\right)  \rightarrow
H^{-1/2}\left(  \Gamma_{j}\right)  $ is surjective (see, e.g., \cite[Cor.
2.8]{Girault86}) it follows $\left[  \left(  \mathsf{D}_{j}\left(  s\right)
\psi\right)  \right]  _{\operatorname*{D};j}\left(  s\right)  =\psi$.
\end{proof}

\subsubsection{Layer potential representation formula\label{rep}}

The key observation for the transformation of our transmission problem to a
non-local skeleton equation is the fact that solutions of the homogeneous PDE
can be expressed by Green's representation formula via their Cauchy data by
means of layer potentials. We start with some preliminaries. For $\varphi\in
H^{-1/2}\left(  \Gamma_{j}\right)  $ and $\psi\in H^{1/2}\left(  \Gamma
_{j}\right)  $ we define the potential
\begin{equation}
w:=\mathsf{D}_{j}\left(  s\right)  \psi-\mathsf{S}_{j}\left(  s\right)
\varphi. \label{wpotential}%
\end{equation}
From Lemmas \ref{LemMapPropSLP} and \ref{LemMapPropDLP} we conclude that $w\in
H^{1}\left(  \mathbb{R}^{3}\backslash\Gamma_{j},\mathbb{A}_{j}%
^{\operatorname*{ext}}\right)  $ and satisfies%
\begin{equation}%
\begin{array}
[c]{cc}%
-\operatorname*{div}\left(  \mathbb{A}_{j}^{\operatorname*{ext}}\nabla
w\right)  +s^{2}p_{j}^{\operatorname*{ext}}w=0 & \text{in }\mathbb{R}%
^{3}\backslash\Gamma_{j},\\
\left[  w\right]  _{\operatorname*{D};j}\left(  s\right)  =\psi\quad
\text{and\quad}\left[  w\right]  _{\operatorname*{N};j}^{\operatorname*{ext}%
}\left(  s\right)  =\varphi. &
\end{array}
\label{tpr3}%
\end{equation}

\begin{proposition}
\label{PropUniqueTransProb}The transmission problem: \textquotedblleft for
given $\varphi\in H^{-1/2}\left(  \Gamma_{j}\right)  $ and $\psi\in
H^{1/2}\left(  \Gamma_{j}\right)  $, find $w\in H^{1}\left(  \mathbb{R}%
^{3}\backslash\Gamma_{j},\mathbb{A}_{j}^{\operatorname*{ext}}\right)  $ such
that (\ref{tpr3}) holds\textquotedblright\ is well posed and the unique
solution is given by $w$ in (\ref{wpotential}).
\end{proposition}

%

\begin{proof}
Existence follows since the potential $w$ in (\ref{wpotential}) defines a
solution. For uniqueness, we assume that there are two solutions $w_{1},w_{2}$
so that the difference $d=w_{1}-w_{2}$ satisfies%
\begin{align*}
-\operatorname*{div}\left(  \mathbb{A}_{j}^{\operatorname*{ext}}\nabla
d\right)  +s^{2}p_{j}^{\operatorname*{ext}}d  &  =0\quad\text{in }%
\mathbb{R}^{3}\backslash\Gamma_{j},\\
\left[  d\right]  _{\operatorname*{D};j}\left(  s\right)   &  =0\quad
\text{and\quad}\left[  d\right]  _{\operatorname*{N};j}^{\operatorname*{ext}%
}\left(  s\right)  =0.
\end{align*}
We multiply the first equation by test functions $v\in H^{1}\left(
\mathbb{R}^{3}\right)  $ and integrate by parts over $\Omega_{j}^{-}$ and
$\Omega_{j}^{+}$. After inserting the transmission conditions we get%
\[
\ell_{j}\left(  s\right)  \left(  d,v\right)  =0\quad\forall v\in H^{1}\left(
\mathbb{R}^{3}\right)  .
\]
Since $\ell_{j}\left(  s\right)  \left(  \cdot,\cdot\right)  $ is coercive
(cf. Lem. \ref{LemLaxMilgram})) we conclude that $d=0$ holds and uniqueness
follows. Hence, the potential $w$ in (\ref{wpotential}) defines the unique
solution. Since the single and double layer operators are continuous,
well-posedness follows.
\end{proof}

\begin{lemma}
[Green's representation formula]\label{LemGreen}Let Assumption \ref{Acoeff} be
satisfied. Let $u^{-}\in H^{1}\left(  \Omega_{j}^{-},\mathbb{A}_{j}%
^{-}\right)  $ and%
\[%
\begin{array}
[c]{cc}%
\mathsf{L}_{j}^{-}\left(  s\right)  u^{-}=0 & \text{in }\Omega_{j}^{-}.
\end{array}
\]
Then, the \emph{Green's representation formulae} hold%
\begin{subequations}
\label{GRFtot}
\end{subequations}%
\begin{align}
u^{-}  &  =\left.  \left(  \mathsf{S}_{j}\left(  s\right)  \gamma
_{\operatorname*{N};j}^{\operatorname*{ext},-}\left(  s\right)  u^{-}%
-\mathsf{D}_{j}\left(  s\right)  \gamma_{\operatorname*{D};j}^{-}\left(
s\right)  u^{-}\right)  \right\vert _{\Omega_{j}^{-}},\tag{%
\ref{GRFtot}%
a}\label{GRFtota}\\
0  &  =\left.  \left(  \mathsf{S}_{j}\left(  s\right)  \gamma
_{\operatorname*{N};j}^{\operatorname*{ext},-}\left(  s\right)  u^{-}%
-\mathsf{D}_{j}\left(  s\right)  \gamma_{\operatorname*{D};j}^{-}\left(
s\right)  u^{-}\right)  \right\vert _{\Omega_{j}^{+}}. \tag{%
\ref{GRFtot}%
b}\label{GRFtotb}%
\end{align}

\end{lemma}%

\begin{proof}
Define $u\in H^{1}\left(  \mathbb{R}^{3}\backslash\Gamma_{j},\mathbb{A}%
_{j}^{\operatorname*{ext}}\right)  $ by $\left.  u\right\vert _{\Omega_{j}%
^{-}}:=u^{-}$ and $\left.  u\right\vert _{\Omega_{j}^{+}}:=0$. Clearly%
\[%
\begin{array}
[c]{cc}%
-\operatorname*{div}\left(  \mathbb{A}_{j}^{\operatorname*{ext}}\nabla
u\right)  +s^{2}p_{j}^{\operatorname*{ext}}u=0 & \text{in }\mathbb{R}%
^{3}\backslash\Gamma_{j}%
\end{array}
\]
and%
\[%
\begin{array}
[c]{ccc}%
\left[  u\right]  _{\operatorname*{D};j}\left(  s\right)  =-\gamma
_{\operatorname*{D};j}^{-}\left(  s\right)  u^{-}, &  & \left[  u\right]
_{\operatorname*{N};j}^{\operatorname*{ext}}\left(  s\right)  =-\gamma
_{\operatorname*{N};j}^{\operatorname*{ext},-}\left(  s\right)  u^{-}.
\end{array}
\]

From Proposition \ref{PropUniqueTransProb} we deduce that the unique solution
of this transmission problem can be written in the form%
\[
u=\mathsf{S}_{j}\left(  s\right)  \gamma_{\operatorname*{N};j}%
^{\operatorname*{ext},-}\left(  s\right)  u^{-}-\mathsf{D}_{j}\left(
s\right)  \gamma_{\operatorname*{D};j}^{-}\left(  s\right)  u^{-}.
\]
From this and the definition of $u$, the representation (\ref{GRFtot})
follows.
\end{proof}

\section{Calder\'{o}n operators\label{SecCaldOp}}

Green's representation formula from Lemma~\ref{LemGreen} expresses homogeneous
solutions of a linear, second order, elliptic PDE by means of their Cauchy
data on the domain boundary. By applying the Cauchy trace to this formula we
obtain the Calder\'{o}n identity. In this way, Dirichlet and Neumann traces
have to be applied to the single layer and double layer potential which give
rise to non-local boundary integral operators on the subdomain boundaries.

\begin{definition}
\label{DefBIO}Let Assumption \ref{Acoeff} be satisfied. For $1\leq j\leq
n_{\Omega}$, the \emph{single layer boundary integral operator (}%
$\mathsf{V}_{j}\left(  s\right)  $\emph{)}, the \emph{double layer boundary
integral operator (}$\mathsf{K}_{j}\left(  s\right)  $\emph{)}, the \emph{dual
double layer boundary integral operator (}$\mathsf{K}_{j}^{\prime}\left(
s\right)  $\emph{)}, the \emph{hypersingular boundary integral operator
(}$\mathsf{W}_{j}\left(  s\right)  $\emph{)} are given by%
\begin{align*}
\mathsf{V}_{j}\left(  s\right)  & :H^{-1/2}\left(  \Gamma_{j}\right)
\rightarrow H^{1/2}\left(  \Gamma_{j}\right)  , & \mathsf{V}_{j}\left(
s\right)  \varphi & :=\{\!\!\{\mathsf{S}_{j}\left(  s\right)  \varphi
\}\!\!\}_{\operatorname*{D};j}\left(  s\right)  ,\\
\mathsf{K}_{j}\left(  s\right)  & :H^{1/2}\left(  \Gamma_{j}\right)  \rightarrow
H^{1/2}\left(  \Gamma_{j}\right)  , & \mathsf{K}_{j}\left(  s\right)
\psi & :=\{\!\!\{\mathsf{D}_{j}\left(  s\right)  \psi\}\!\!\}_{\operatorname*{D}%
;j}\left(  s\right)  ,\\
\mathsf{K}_{j}^{\prime}\left(  s\right)  & :H^{-1/2}\left(  \Gamma_{j}\right)
\rightarrow H^{-1/2}\left(  \Gamma_{j}\right)  , & \mathsf{K}_{j}^{\prime
}\left(  s\right)  \varphi&:=\{\!\!\{\mathsf{S}_{j}\left(  s\right)
\varphi\}\!\!\}_{\operatorname*{N};j}^{\operatorname*{ext}}\left(  s\right)
,\\
\mathsf{W}_{j}\left(  s\right)  &:H^{1/2}\left(  \Gamma_{j}\right)  \rightarrow
H^{-1/2}\left(  \Gamma_{j}\right)  , & \mathsf{W}_{j}\left(  s\right)
\psi & :=-\{\!\!\{\mathsf{D}_{j}\left(  s\right)  \psi\}\!\!\}_{\operatorname*{N}%
;j}^{\operatorname*{ext}}\left(  s\right)  ,
\end{align*}
for all $\varphi\in H^{-1/2}\left(  \Gamma_{j}\right)  $ and $\psi\in
H^{1/2}\left(  \Gamma_{j}\right)  $.
\end{definition}

In order to define the Calder\'{o}n operator we introduce a bilinear form on
the multi trace spaces (cf. Def. \ref{DefMT}) and set, for $%
\mbox{\boldmath$ \phi$}%
_{j}=\left(  \phi_{\operatorname*{D};j},\phi_{\operatorname*{N};j}\right)
\in\mathbf{X}_{j}$ and $%
\mbox{\boldmath$ \psi$}%
_{j}=\left(  \psi_{\operatorname*{D};j},\psi_{\operatorname*{N};j}\right)
\in\mathbf{X}_{j}$,
\begin{subequations}
\label{0}%
\begin{equation}
\left\langle
\mbox{\boldmath$ \phi$}%
_{j},%
\mbox{\boldmath$ \psi$}%
_{j}\right\rangle _{\mathbf{X}_{j}}:=\left\langle \phi_{\operatorname*{D}%
;j},\psi_{\operatorname*{N};j}\right\rangle _{\Gamma_{j}}+\left\langle
\psi_{\operatorname*{D};j},\phi_{\operatorname*{N};j}\right\rangle
_{\Gamma_{j}}, \label{iXj}%
\end{equation}
where, again, $\left\langle \cdot,\cdot\right\rangle _{\Gamma_{j}}$ designates the
pairing between $H^{1/2}\left(  \Gamma_{j}\right)  $ and $H^{-1/2}\left(
\Gamma_{j}\right)  $. For $\mbox{\boldmath$ \phi$}=\left(
\mbox{\boldmath$ \phi$}_{j}\right)  _{j=1}^{n_{\Omega}}\in\mathbb{X}\left(  \mathcal{P}_{\Omega}\right)  $ and
\mbox{\boldmath$ \psi$}$=\left(\mbox{\boldmath$ \psi$}_{j}\right)  _{j=1}^{n_{\Omega}}\in\mathbb{X}\left(  \mathcal{P}_{\Omega
}\right)  $ we define the bilinear form $\left\langle \cdot,\cdot\right\rangle
:\mathbb{X}\left(  \mathcal{P}_{\Omega}\right)  \times\mathbb{X}\left(
\mathcal{P}_{\Omega}\right)  \rightarrow\mathbb{C}$ by%
\end{subequations}
\begin{equation}
\left\langle\mbox{\boldmath$ \phi$},\mbox{\boldmath$ \psi$}
\right\rangle _{\mathbb{X}}:=\sum_{1\leq j\leq n_{\Omega}}\left\langle
\mbox{\boldmath$ \phi$}_{j},
\mbox{\boldmath$ \psi$}_{j}\right\rangle _{\mathbf{X}_{j}}. \label{bilixfull}%
\end{equation}

\noindent\fbox{\begin{minipage}{0.98\linewidth}
\begin{definition}
\label{DefCaldOp}Let Assumption \ref{Acoeff} be satisfied. The
\emph{\bf Calder\'{o}n operator }$\mathsf{C}\left(  s\right)  :\mathbb{X}\left(
\mathcal{P}_{\Omega}\right)  \rightarrow\mathbb{X}\left(  \mathcal{P}_{\Omega
}\right)  $ is given by%
\[
\mathsf{C}\left(  s\right)  :=\operatorname*{diag}\left[  \mathsf{C}%
_{j}\left(  s\right)  :1\leq j\leq n_{\Omega}\right]  \quad\text{with\quad
}\mathsf{C}_{j}\left(  s\right)  :=\left[
\begin{array}
[c]{ll}%
-\mathsf{K}_{j}\left(  s\right)  & \mathsf{V}_{j}\left(  s\right) \\
\mathsf{W}_{j}\left(  s\right)  & \mathsf{K}_{j}^{\prime}\left(  s\right)
\end{array}
\right]  .
\]
The sesquilinear form $c\left(  s\right)  :\mathbb{X}\left(  \mathcal{P}%
_{\Omega}\right)  \times\mathbb{X}\left(  \mathcal{P}_{\Omega}\right)
\rightarrow\mathbb{C}$ associated to the operator $\mathsf{C}\left(  s\right)
$ is%
\begin{gather}
  \label{eq:COP}
  c\left(  s\right)  \left(\mbox{\boldmath$\phi$},\mbox{\boldmath$\psi$}\right)
  :=\left\langle
    \mathsf{C}\left(  s\right)
    \mbox{\boldmath$\phi$},\overline{\mbox{\boldmath$ \psi$}}
  \right\rangle _{\mathbb{X}}.
\end{gather}
\end{definition}%
\end{minipage}%
}

Let
$\operatorname*{Id}:\mathbb{X}\left( \mathcal{P}_{\Omega}\right)
\rightarrow\mathbb{X}\left( \mathcal{P}_{\Omega}\right) $ denote the identity. An
essential property of the Calder\'{o}n operator is that
{$\left( \frac{1}{2}\operatorname*{Id} + \mathsf{C}\left( s\right) \right) $} is a projector
into the space of Cauchy traces of solutions of the homogeneous PDE (\ref{homPDEGreen1})
as can be seen from the next Lemma. Recall the definition of the one-sided Cauchy trace
$\mbox{\boldmath$ \gamma$}_{\operatorname*{C};j}^{\operatorname*{ext},-}\left( s\right) $
from (\ref{DefCT}) and (\ref{gammadjs}).

\begin{lemma}
\label{LemCaldProj}Let Assumption \ref{Acoeff} be satisfied. Let $u^{-}\in
H^{1}\left(  \Omega_{j}^{-},\mathbb{A}_{j}^{-}\right)  $ and%
\[%
\begin{array}
[c]{cc}%
\mathsf{L}_{j}^{-}\left(  s\right)  u^{-}=0 & \text{in }\Omega_{j}^{-}.
\end{array}
\]
Then, for any $j\in\left\{  1,2,\ldots,n_{\Omega}\right\}  $ it holds%
\begin{equation}
\left(  \mathsf{C}_{j}\left(  s\right)  -\frac{1}{2}\operatorname*{Id}%
\nolimits_{j}\right)
\mbox{\boldmath$ \gamma$}%
_{\operatorname*{C};j}^{\operatorname*{ext},-}\left(  s\right)  u^{-}=0,
\label{CaldProj}%
\end{equation}
where $\operatorname*{Id}_{j}:\mathbf{X}_{j}\rightarrow\mathbf{X}_{j}$ is the
identity in $\mathbf{X}_{j}$.
\end{lemma}

%

\begin{proof}
Green's representation formula (\ref{GRFtota}) gives us%
\begin{align*}
\gamma_{\operatorname*{D};j}^{-}\left(  s\right)  u^{-} &  =\gamma
_{\operatorname*{D};j}^{-}\left(  s\right)  \mathsf{S}_{j}\left(  s\right)
\gamma_{\operatorname*{N};j}^{\operatorname*{ext},-}\left(  s\right)
u^{-}-\gamma_{\operatorname*{D};j}^{-}\left(  s\right)  \mathsf{D}_{j}\left(
s\right)  \gamma_{\operatorname*{D};j}^{-}\left(  s\right)  u^{-},\\
0 &  =\gamma_{\operatorname*{D};j}^{+}\left(  s\right)  \mathsf{S}_{j}\left(
s\right)  \gamma_{\operatorname*{N};j}^{\operatorname*{ext},-}\left(
s\right)  u^{-}-\gamma_{\operatorname*{D};j}^{+}\left(  s\right)
\mathsf{D}_{j}\left(  s\right)  \gamma_{\operatorname*{D};j}^{-}\left(
s\right)  u^{-},\\
\gamma_{\operatorname*{N};j}^{\operatorname*{ext},-}\left(  s\right)  u^{-} &
=\gamma_{\operatorname*{N};j}^{\operatorname*{ext},-}\left(  s\right)
\mathsf{S}_{j}\left(  s\right)  \gamma_{\operatorname*{N};j}%
^{\operatorname*{ext},-}\left(  s\right)  u^{-}-\gamma_{\operatorname*{N}%
;j}^{\operatorname*{ext},-}\left(  s\right)  \mathsf{D}_{j}\left(  s\right)
\gamma_{\operatorname*{D};j}^{-}\left(  s\right)  u^{-},\\
0 &  =-\gamma_{\operatorname*{N};j}^{\operatorname*{ext},+}\left(  s\right)
\mathsf{S}_{j}\left(  s\right)  \gamma_{\operatorname*{N};j}%
^{\operatorname*{ext},-}\left(  s\right)  u^{-}+\gamma_{\operatorname*{N}%
;j}^{\operatorname*{ext},+}\left(  s\right)  \mathsf{D}_{j}\left(  s\right)
\gamma_{\operatorname*{D};j}^{-}\left(  s\right)  u^{-}.
\end{align*}
We multiply the first two relations by $1/2$ and add them and do the same with
the last two relations. This yields%
\begin{align*}
\frac{1}{2}\gamma_{\operatorname*{D};j}^{-}\left(  s\right)  u^{-} &
=\mathsf{V}_{j}\left(  s\right)  \gamma_{\operatorname*{N};j}%
^{\operatorname*{ext},-}\left(  s\right)  u^{-}-\mathsf{K}_{j}\left(
s\right)  \gamma_{\operatorname*{D};j}^{-}\left(  s\right)  u^{-},\\
\frac{1}{2}\gamma_{\operatorname*{N};j}^{\operatorname*{ext},-}\left(
s\right)  u^{-} &  =\mathsf{K}_{j}^{\prime}\left(  s\right)  \gamma
_{\operatorname*{N};j}^{\operatorname*{ext},-}\left(  s\right)  u^{-}%
+\mathsf{W}_{j}\left(  s\right)  \gamma_{\operatorname*{D};j}^{-}\left(
s\right)  u^{-}%
\end{align*}
and after a reordering of the terms (\ref{CaldProj}) follows.
\end{proof}

\section{Single-trace formulation of the transmission
problem\label{SecMultiSingleTraceForm}}

In this section, we formulate the transmission problem (\ref{generalTP}) as a
non-local skeleton equation for the Cauchy data of the solution. We start from
a transmission problem with given jump data: We seek
\[
\mathbf{u}^{\operatorname*{mult}}=\left(  \mathbf{u}_{j}^{\operatorname*{mult}%
}\right)  _{j=1}^{n_{\Omega}}=\left(  \left(  u_{\operatorname*{D}%
;j}^{\operatorname*{mult}},u_{\operatorname*{N};j}^{\operatorname*{mult}%
}\right)  \right)  _{j=1}^{n_{\Omega}}\in\mathbb{X}\left(  \mathcal{P}%
_{\Omega}\right)
\]
as the solution of%
\begin{gather}
  \label{eq:5}
\begin{array}
[c]{cl}%
\left(  \mathsf{C}_{j}\left(  s\right)  -\frac{1}{2}\operatorname*{Id}%
  \nolimits_{j}\right)  \mathbf{u}_{j}^{\operatorname*{mult}}=0, &
1\leq j\leq n_{\Omega},\\%
\left[  \mathbf{u}^{\operatorname*{mult}}\right]  _{j,k}\left(  s\right)
=\left[\mbox{\boldmath$ \beta$}\right]  _{j,k}, & 1\leq j,k\leq n_{\Omega},\\
s^{1/2}\left.  u_{\operatorname*{D};j}^{\operatorname*{mult}}\right\vert
_{\Gamma_{j}\cap\Gamma_{\operatorname*{D}}}=\left.  \beta_{\operatorname*{D}%
  ;j}\right\vert _{\Gamma_{j}\cap\Gamma_{\operatorname*{D}}}& 1\leq j\leq
n_{\Omega}\\
s^{-1/2}\left.  u_{\operatorname*{N};j}^{\operatorname*{mult}}\right\vert
_{\Gamma_{j}\cap\Gamma_{\operatorname*{N}}}=\left.  \beta_{\operatorname*{N}%
    ;j}\right\vert _{\Gamma_{j}\cap\Gamma_{\operatorname*{N}}}, & 1\leq j\leq
n_{\Omega}%
\end{array}
\end{gather}
with $\mbox{\boldmath$ \beta$}$ as in (\ref{generalTP}). Note that
$\mathbf{u}^{\operatorname*{mult}}$ is multi-valued on the inner skeleton
$\Sigma\cap\Omega$.
Following \rhc{\cite[Section~3]{ClHiptJH}}, a \emph{single trace formulation} and
single-valued functions is obtained when the transmission conditions are incorporated into
the multi trace space $\mathbb{X}\left( \mathcal{P}_{\Omega}\right) $. We define the free
single trace space
$\mathbb{X}^{\operatorname{single}}\left( \mathcal{P}_{\Omega }\right) $ and the single
trace space with incorporated homogeneous boundary conditions by%
\begin{align}
\mathbb{X}^{\operatorname{single}}\left(  \mathcal{P}_{\Omega}\right)   &
:=\left\{  \left(  \left(  \psi_{\operatorname*{D};j},\psi_{\operatorname*{N}%
;j}\right)  \right)  _{j=1}^{n_{\Omega}}\in\mathbb{X}\left(  \mathcal{P}%
_{\Omega}\right)  \mid\left\{
\begin{array}[c]{ll}%
\left.
\begin{array}
[c]{l}%
\exists v\in H^{1}\left(  \Omega\right) \\
\text{s.t. }\forall1\leq j\leq n_{\Omega}%
\end{array}
\right\}  : & \psi_{\operatorname*{D};j}=\gamma_{\operatorname*{D};j}v\\
\left.
\begin{array}
[c]{l}%
\exists\mathbf{w}\in\mathbf{H}\left(  \Omega,\operatorname*{div}\right) \\
\text{s.t. }\forall1\leq j\leq n_{\Omega}%
\end{array}
\right\}  : & \psi_{\operatorname*{N};j}=\gamma_{\mathbf{n};j}\mathbf{w}%
\end{array}
\right.
\right\}  ,
\label{defXsingle}\\%
\mathbb{X}_{0}^{\operatorname{single}}\left(  \mathcal{P}_{\Omega}\right)   &
:=\Bigl\{
\begin{aligned}[t]
  & \left( \left( \psi_{\operatorname*{D};j},\psi_{\operatorname*{N}%
        ;j}\right) \right) _{j=1}^{n_{\Omega}}\in\mathbb{X}^{\operatorname{single}%
  }\left( \mathcal{P}_{\Omega}\right) \mid\forall1\leq j\leq n_{\Omega }: \\ &
  \left.
    \psi_{\operatorname*{D};j}\right\vert _{\Gamma_{j}\cap
    \Gamma_{\operatorname*{D}}}=0\;\wedge\text{\ }\left.  \psi_{\operatorname*{N}%
      ;j}\right\vert _{\Gamma_{j}\cap\Gamma_{\operatorname*{N}}}=0 \Bigr\}.
\end{aligned}
\nonumber
\end{align}
{We set $\mathbf{u}%
^{\operatorname{single}}:=\left(  \mathbf{u}_{j}^{\operatorname*{mult}}-%
\mbox{\boldmath$ \beta$}%
\left(  s\right)  \right)  _{j=1}^{n_{\Omega}}$ for $%
\mbox{\boldmath$ \beta$}%
\left(  s\right)  :=\left(  \left(  s^{-1/2}\beta_{\operatorname*{D}%
,j},s^{1/2}\beta_{\operatorname*{N},j}\right)  \right)  _{j=1}^{n_{\Omega}}$}
and observe that $\mathbf{u}^{\operatorname{single}}$ satisfies
\begin{gather}
  \label{eq:6}
\begin{array}
[c]{ll}%
\left(  \mathsf{C}_{j}\left(  s\right)  -\frac{1}{2}\operatorname*{Id}%
\nolimits_{j}\right)  \mathbf{u}_{j}^{\operatorname{single}}=-\left(
\mathsf{C}_{j}\left(  s\right)  -\frac{1}{2}\operatorname*{Id}\nolimits_{j}%
\right)\mbox{\boldmath$ \beta$}_{j}  & \text{on }\Gamma_{j},\quad1\leq j\leq n_{\Omega},\\
\left[  \mathbf{u}^{\operatorname{single}}\right]  _{j,k}\left(  s\right)
=\mathbf{0}, & 1\leq j,k\leq n_{\Omega},\\
\left.  u_{\operatorname*{D};j}^{\operatorname{single}}\right\vert
_{\Gamma_{j}\cap\Gamma_{\operatorname*{D}}}=0\quad\text{and}\quad \left.
u_{\operatorname*{N};j}^{\operatorname{single}}\right\vert _{\Gamma_{j}%
\cap\Gamma_{\operatorname*{N}}}=0, & 1\leq j\leq n_{\Omega}.
\end{array}
\end{gather}
This implies that
$\mathbf{u}^{\operatorname{single}}\in\mathbb{X}_{0}^{\operatorname{single}}\left(
  \mathcal{P}_{\Omega}\right) $.

A reversed perspective on this derivation of the skeleton equation in the
single trace space from the original transmission problem (\ref{generalTP}) is
as follows: One solves the non-local skeleton problem in the single trace
space (in variational form):

\noindent\fbox{\begin{minipage}{0.98\linewidth}
    With $c(s)$ from \eqref{eq:COP} find $\mathbf{u}^{\operatorname{single}}\in\mathbb{X}_{0}%
^{\operatorname{single}}\left(  \mathcal{P}_{\Omega}\right)$ such that 
\begin{equation}%
c\left(  s\right)  \left(  \mathbf{u}^{\operatorname{single}},\mbox{\boldmath$ \psi$}
\right)  -\frac{1}{2}\left\langle \mathbf{u}^{\operatorname{single}},\overline{
\mbox{\boldmath$ \psi$}}\right\rangle _{\mathbb{X}} =-\left(  c\left(  s\right)  \left(\mbox{\boldmath$ \beta$}
\left(  s\right)  ,\mbox{\boldmath$ \psi$}\right)  -\frac{1}{2}\left\langle
\mbox{\boldmath$ \beta$}
\left(  s\right)  ,\overline{\mbox{\boldmath$ \psi$}}\right\rangle _{\mathbb{X}}\right)   
\label{singskeleq}%
\end{equation}
for all $\mbox{\boldmath$ \psi$}\in\mathbb{X}_{0}^{\operatorname{single}}\left(
  \mathcal{P}_{\Omega}\right)$.
\end{minipage}
}

We obtain $\mathbf{u}_{j}^{\operatorname*{mult}}:=\mathbf{u}^{\operatorname{single}}+
\mbox{\boldmath$ \beta$}\left(  s\right)  $. Then, we use Green's representation formula
\[
u_{j}:=\left.  \left(  \mathsf{S}_{j}\left(  s\right)  u_{\operatorname*{N}%
;j}^{\operatorname*{mult}}-\mathsf{D}_{j}\left(  s\right)
u_{\operatorname*{D};j}^{\operatorname*{mult}}\right)  \right\vert
_{\Omega_{j}^{-}},\quad1\leq j\leq n_{\Omega}.
\]
Finally, the function $\mathbf{u}=\left(  u_{j}\right)  _{j=1}^{n_{\Omega}}%
\in\mathbf{H}\left(  \Omega,\mathbb{A}\right)  $ solves the original
transmission problem (\ref{generalTP}).

Next, we prove the well-posedness of (\ref{singskeleq}). The essential point is to prove
$s$-explicit continuity estimates for the layer potentials and the boundary integral
operators as well as coercivity results for $\mathsf{V}%
\left( s\right) $, $\mathsf{W}\left( s\right) $, and
$\mathsf{C}\left( s\right) -\frac{1}{2}\operatorname*{Id}$.

We start with an estimate of the Dirichlet and Neumann trace of homogeneous
solutions of the acoustic PDE.

\begin{lemma}
\label{LemScaledTraces}Let Assumption \ref{Acoeff} be satisfied and set
$\mathbb{A}_{j}^{\sigma}:=\left.  \mathbb{A}_{j}^{\operatorname*{ext}%
}\right\vert _{\Omega_{j}^{\sigma}}$, $\sigma\in\left\{  +,-\right\}  $. Then there are
constants $C_{D},C>0$ independent of $s$ such that 
\begin{gather}
\left\Vert \gamma_{\operatorname*{D};j}^{\sigma}\left(  s\right)  v\right\Vert
_{H^{1/2}\left(  \Gamma_{j}\right)  }\leq C_{\operatorname*{D}}\left\vert
s\right\vert ^{1/2}\left\Vert v\right\Vert _{H^{1}\left(  \Omega_{j}^{\sigma
}\right)  }\leq C\left\vert s\right\vert ^{1/2}\left\Vert v\right\Vert
_{H^{1}\left(  \Omega_{j}^{\sigma}\right)  ;s}\quad\forall v\in H^{1}\left(
\Omega_{j}^{\sigma}\right)  . \label{scaledtraceest}%
\end{gather}
Vice versa, there exists $C>0$ independent of $s$ and a linear bounded extension operator
$\mathsf{E}%
_{j}\left( s\right) :H^{1/2}\left( \Gamma_{j}\right) \rightarrow H^{1}\left(
  \mathbb{R}^{3}\right) $ which satisfies for all
$\varphi\in H^{1/2}\left( \Gamma_{j}\right) :$%
\begin{equation}
\gamma_{\operatorname*{D},j}\left(  s\right)  \mathsf{E}_{j}\left(  s\right)
\varphi=\varphi\quad\text{and\quad}\left\Vert \mathsf{E}_{j}\left(  s\right)
\varphi\right\Vert _{H^{1}\left(  \mathbb{R}^{3}\right)  ;s}\leq C\left\Vert
\varphi\right\Vert _{H^{1/2}\left(  \Gamma_{j}\right)  }. \label{tracelifting}%
\end{equation}

Let $v\in H^{1}\left(  \mathbb{R}^{3}\right)  $ such that $v^{\sigma}:=\left.
v\right\vert _{\Omega_{j}^{\sigma}}$ belongs to $H^{1}\left(  \Omega
_{j}^{\sigma},\mathbb{A}_{j}^{\sigma}\right)  $ and
\[%
\begin{array}
[c]{cc}%
-\operatorname*{div}\left(  \mathbb{A}_{j}^{\operatorname*{ext}}\nabla
v\right)  +s^{2}p_{j}^{\operatorname*{ext}}v=0 & \text{in }\mathbb{R}%
^{3}\backslash\Gamma_{j}.
\end{array}
\]
Then,%
\begin{equation}
\left\Vert \gamma_{\operatorname*{N};j}^{\operatorname*{ext},\sigma}\left(
s\right)  v^{\sigma}\right\Vert _{H^{-1/2}\left(  \Gamma_{j}\right)  }\leq
C\Lambda_{j}\left\Vert v^{\sigma}\right\Vert _{H^{1}\left(  \Omega_{j}%
^{\sigma}\right)  ;s}, \label{conormaltraceest}%
\end{equation}
where $\Lambda_{j}$ is as in Lem. \ref{LemLaxMilgram} and $C$ depends only on
the domain $\Omega_{j}^{\sigma}$.
\end{lemma}

%

\begin{proof}
The estimates in (\ref{scaledtraceest}) follow from the scaling of
$\gamma_{\operatorname*{D};j}^{\sigma}\left(  s\right)  $ with respect to $s$
and (\ref{traceest}).

The extension operator $\mathsf{E}_{j}\left(  s\right)  :H^{1/2}\left(
\Gamma_{j}\right)  \rightarrow H^{1}\left(  \mathbb{R}^{3}\right)  $ is
defined for $\varphi\in H^{1/2}\left(  \Gamma_{j}\right)  $ piecewise in
$\Omega_{j}^{\sigma}$, $\sigma\in\left\{  +,-\right\}  $, by%
\begin{align*}
  \gamma_{\operatorname*{D};j}^{\sigma}\left(  s\right)  \left(
\mathsf{E}_{j}\left(  s\right)  \varphi\right)  & =\varphi\quad\text{and}\\
  \left(  \nabla\left(  \mathsf{E}_{j}\left(  s\right)  \varphi\right)
,\nabla w\right)  _{\mathbf{L}^{2}\left(  \Omega_{j}^{\sigma}\right)
}+\left\vert s\right\vert ^{2}\left(  \mathsf{E}_{j}\left(  s\right)
\varphi,w\right)  _{L^{2}\left(  \Omega_{j}^{\sigma}\right)  }& =0\quad\forall
w\in H^{1}\left(  \Omega_{j}^{\sigma}\right)  .
\end{align*}
From \cite[Prop. 2.5.1]{sayas_book} the estimate (\ref{tracelifting}) follows.

For (\ref{conormaltraceest}) we adapt the standard proof (see, e.g.,
\cite[Prop. 2.5.2]{sayas_book}) to our setting. For given $\psi\in
H^{1/2}\left(  \Gamma_{j}\right)  $ let $w:=\mathsf{E}_{j}\left(  s\right)
\psi$. Let $w^{\sigma}:=\left.  w\right\vert _{\Omega_{j}^{\sigma}}$ and
$v^{\sigma}:=\left.  v\right\vert _{\Omega_{j}^{\sigma}}$, $\sigma\in\left\{
+,-\right\}  $. Green's first identity (\ref{Green1}) gives us
\begin{align*}
\left\vert \left\langle \gamma_{\operatorname*{N};j}^{\operatorname*{ext}%
,\sigma}\left(  s\right)  v^{\sigma},\overline{\psi}\right\rangle _{\Gamma
_{j}}\right\vert  &  =\left\vert \left(  \frac{\overline{s}}{s}\right)
^{1/2}\left\langle \gamma_{\operatorname*{N};j}^{\operatorname*{ext},\sigma
}\left(  s\right)  v^{\sigma},\gamma_{\operatorname*{D};j}^{\sigma}\left(
s\right)  \overline{w^{\sigma}}\right\rangle _{\Gamma_{j}}\right\vert \\
&  =\left\vert \left\langle \mathbb{A}_{j}^{\sigma}\nabla v^{\sigma}%
,\nabla\overline{w^{\sigma}}\right\rangle _{\Omega_{j}^{\sigma}}%
+s^{2}\left\langle p_{j}^{\sigma}v^{\sigma},\overline{w^{\sigma}}\right\rangle
_{\Omega_{j}^{\sigma}}\right\vert \\
&  \overset{\text{Lem. \ref{LemLaxMilgram}}}{\leq}\Lambda_{j}\left\Vert
v^{\sigma}\right\Vert _{H^{1}\left(  \Omega_{j}^{\sigma}\right)  ;s}\left\Vert
w^{\sigma}\right\Vert _{H^{1}\left(  \Omega_{j}^{\sigma}\right)  ;s}\\
&  \overset{\text{(\ref{tracelifting})}}{\leq}C\Lambda_{j}\left\Vert
v^{\sigma}\right\Vert _{H^{1}\left(  \Omega_{j}^{\sigma}\right)  ;s}\left\Vert
\psi\right\Vert _{H^{1/2}\left(  \Gamma_{j}\right)  }.
\end{align*}
Finally,%
\[
\left\Vert \gamma_{\operatorname*{N};j}^{\operatorname*{ext},\sigma}\left(
s\right)  v^{\sigma}\right\Vert _{H^{-1/2}\left(  \Gamma_{j}\right)  }%
=\sup_{\psi\in H^{1/2}\left(  \Gamma_{j}\right)  \backslash\left\{  0\right\}
}\frac{\left\vert \left\langle \gamma_{\operatorname*{N};j}%
^{\operatorname*{ext},\sigma}\left(  s\right)  v^{\sigma},\overline{\psi
}\right\rangle _{\Gamma_{j}}\right\vert }{\left\Vert \psi\right\Vert
_{H^{1/2}\left(  \Gamma_{j}\right)  }}\leq C\Lambda_{j}\left\Vert v^{\sigma
}\right\Vert _{H^{1}\left(  \Omega_{j}^{\sigma}\right)  ;s}.
\]
\end{proof}

\begin{lemma}
  \label{Lem_s_explicit_operator_est}Let Assumption \ref{Acoeff} be satisfied.  Then the
  sesquilinear form induced by the single layer boundary integral operator satisfies the
  \rhc{$s$-explicit} coercivity and continuity estimates%
\begin{subequations}
\label{slbiebounds}
\end{subequations}%
\begin{align}
\operatorname{Re}\left\langle \varphi,\overline{\mathsf{V}_{j}\left(
s\right)  \varphi}\right\rangle _{\Gamma_{j}}  &  \geq c\frac
{\operatorname{Re}s}{\left\vert s\right\vert }\frac{\lambda_{j}}{\Lambda
_{j}^{2}}\left\Vert \varphi\right\Vert _{H^{-1/2}\left(  \Gamma_{j}\right)
}^{2}\qquad\forall\varphi\in H^{-1/2}\left(  \Gamma_{j}\right)  ,\tag{%
\ref{slbiebounds}%
a}\label{slbieboundsa}\\
\left\vert \left\langle \mathsf{V}_{j}\left(  s\right)  \varphi,\overline
{\psi}\right\rangle _{\Gamma_{j}}\right\vert  &  \leq C\frac{\left\vert
s\right\vert ^{2}}{\lambda_{j}\operatorname{Re}s}\left\Vert \varphi\right\Vert
_{H^{-1/2}\left(  \Gamma_{j}\right)  }\left\Vert \psi\right\Vert
_{H^{-1/2}\left(  \Gamma_{j}\right)  }\quad\forall\varphi,\psi\in
H^{-1/2}\left(  \Gamma_{j}\right)  . \tag{%
\ref{slbiebounds}%
b}\label{slbieboundsb}%
\end{align}
The dual double layer boundary integral operator is bounded and satisfies the
estimates%
\begin{equation}
\left\Vert \mathsf{K}_{j}^{\prime}\left(  s\right)  \varphi\right\Vert
_{H^{-1/2}\left(  \Gamma_{j}\right)  }\leq C\frac{\Lambda_{j}}{\lambda_{j}%
}\frac{\left\vert s\right\vert ^{3/2}}{\operatorname{Re}s}\left\Vert
\varphi\right\Vert _{H^{-1/2}\left(  \Gamma_{j}\right)  }\quad\forall
\varphi\in H^{-1/2}\left(  \Gamma_{j}\right)  \text{.} \label{kprimebound}%
\end{equation}
The sesquilinear form induced by the hypersingular boundary integral operator
satisfies the coercivity and continuity estimate%
\begin{subequations}
\label{Wcoercall}
\end{subequations}%
\begin{align}
\operatorname{Re}\left\langle \mathsf{W}_{j}\left(  s\right)  \psi
,\overline{\mathsf{\psi}}\right\rangle _{\Gamma_{j}}  &  \geq c\frac
{\operatorname{Re}s}{\left\vert s\right\vert ^{2}}\lambda_{j}\left\Vert
\psi\right\Vert _{H^{1/2}\left(  \Gamma_{j}\right)  }^{2}\qquad\forall\psi\in
H^{1/2}\left(  \Gamma_{j}\right)  ,\tag{%
\ref{Wcoercall}%
a}\label{Wcoerc}\\
\left\vert \left\langle \mathsf{W}_{j}\left(  s\right)  \psi,\overline
{\varphi}\right\rangle _{\Gamma_{j}}\right\vert  &  \leq C\frac{\Lambda
_{j}^{2}}{\lambda_{j}}\frac{\left\vert s\right\vert }{\operatorname{Re}%
s}\left\Vert \psi\right\Vert _{H^{1/2}\left(  \Gamma_{j}\right)  }\left\Vert
\varphi\right\Vert _{H^{1/2}\left(  \Gamma_{j}\right)  }\quad\forall
\varphi,\psi\in H^{1/2}\left(  \Gamma_{j}\right)  . \tag{%
\ref{Wcoercall}%
b}\label{Wcoercb}%
\end{align}
The double layer boundary integral operator is bounded and satisfies the
estimate%
\begin{equation}
\left\Vert \mathsf{K}_{j}\left(  s\right)  \psi\right\Vert _{H^{1/2}\left(
\Gamma_{j}\right)  }\leq C\frac{\Lambda_{j}}{\lambda_{j}}\frac{\left\vert
s\right\vert ^{3/2}}{\operatorname{Re}s}\left\Vert \psi\right\Vert
_{H^{1/2}\left(  \Gamma_{j}\right)  }\quad\forall\psi\in H^{1/2}\left(
\Gamma_{j}\right)  \text{.} \label{kbound}%
\end{equation}
For the single layer potential, the estimate%
\begin{equation}
\left\Vert \mathsf{S}_{j}\left(  s\right)  \varphi\right\Vert _{H^{1}\left(
\mathbb{R}^{3}\right)  ;s}\leq C\frac{\left\vert s\right\vert ^{3/2}}%
{\lambda_{j}\operatorname{Re}s}\left\Vert \varphi\right\Vert _{H^{-1/2}\left(
\Gamma_{j}\right)  }\quad\forall\varphi\in H^{-1/2}\left(  \Gamma_{j}\right)
\label{SLPest}%
\end{equation}
holds. The operator norm of the double layer potential is bounded by 
\begin{equation}
\left\Vert \mathsf{D}_{j}\left(  s\right)  \psi\right\Vert _{H^{1}\left(
\mathbb{R}^{3}\backslash\Gamma_{j}\right)  ;s}\leq C\frac{\Lambda_{j}}%
{\lambda_{j}}\frac{\left\vert s\right\vert }{\operatorname{Re}s}\left\Vert
\psi\right\Vert _{H^{1/2}\left(  \Gamma_{j}\right)  }\quad\forall\psi\in
H^{1/2}\left(  \Gamma_{j}\right)  , \label{DLPest}%
\end{equation}
where for $u\in L^{2}\left(  \mathbb{R}^{3}\right)  $ with $u^{\sigma
}:=\left.  u\right\vert _{\Omega_{j}^{\sigma}}\in H^{1}\left(  \Omega
_{j}^{\sigma}\right)  $, $\sigma=\left\{  +,-\right\}  $ the \emph{broken
}$H^{1}$ norm is given by%
\[
\left\Vert u\right\Vert _{H^{1}\left(  \mathbb{R}^{3}\backslash\Gamma
_{j}\right)  ;s}:=\left(  \sum_{\sigma\in\left\{  +,-\right\}  }\left\Vert
u^{\sigma}\right\Vert _{H^{1}\left(  \Omega_{j}^{\sigma}\right)  ;s}%
^{2}\right)  ^{1/2}.
\]
All constants $c,C>0$ only depend on $\Omega_{j}$ and, in particular, are independent of
$s$. 
\end{lemma}

The proof of this lemma follows standard arguments and hence is postponed to
Appendix~\ref{AppProof}.

\begin{lemma}
  Let Assumption \ref{Acoeff} be satisfied.

  The sesquilinear form $\left\langle
\mathsf{C}_{j}\left(  s\right)  \cdot,\overline{\cdot}\right\rangle
_{\mathbf{X}_{j}}:\mathbf{X}_{j}\times\mathbf{X}_{j}\rightarrow\mathbb{C}$ is
coercive: 
\begin{gather*}
\operatorname{Re}\left\langle \mathsf{C}_{j}\left(  s\right)
\mbox{\boldmath$ \psi$}_{j},\overline{\mbox{\boldmath$ \psi$}_{j}}\right\rangle
_{\mathbf{X}_{j}}\geq
c\frac{\lambda_{j}}{1+\Lambda_{j}^{2}}\frac{\operatorname{Re}s}{\left\vert s\right\vert
  ^{2}}\left\Vert\mbox{\boldmath$ \psi$}\right\Vert _{\mathbf{X}_{j}}^{2}\quad \forall
\mbox{\boldmath$ \psi$}_{j}\in\mathbf{X}_{j},
\end{gather*}
and continuous: 
\begin{gather*}
  \left\vert \left\langle \mathsf{C}_{j}\left( s\right)
      \mbox{\boldmath$ \psi$}_{j},\overline{\mbox{\boldmath$ \phi$}_{j}}\right\rangle
    _{\mathbf{X}_{j}}\right\vert \leq C\frac{1+\Lambda_{j}%
  }{\lambda_{j}}\frac{\left\vert s\right\vert ^{2}}{\operatorname{Re}%
    s}\left\Vert\mbox{\boldmath$ \psi$}_{j}\right\Vert _{\mathbf{X}_{j}}\left\Vert
    \mbox{\boldmath$ \phi$}_{j}\right\Vert _{\mathbf{X}_{j}} \quad \forall
  \mbox{\boldmath$ \psi$}_{j}, \mbox{\boldmath$ \phi$}_{j}\in\mathbf{X}_{j}.
\end{gather*}
\end{lemma}

The proof follows closely the arguments in \cite[Lem. 3.1]{banjai_coupling}
for the case of constant coefficients and we adapt it here to our general setting.%

\begin{proof}
We pick some $%
\mbox{\boldmath$ \psi$}%
_{j}:=\left(  \psi_{\operatorname*{D};j},\psi_{\operatorname*{N};j}\right)
\in\mathbf{X}_{j}$ and define $u\in H^{1}\left(  \mathbb{R}^{3}\backslash
\Gamma_{j}\right)  $ by%
\[
u:=\mathsf{S}_{j}\left(  s\right)  \psi_{\operatorname*{N};j}-\mathsf{D}%
_{j}\left(  s\right)  \psi_{\operatorname*{D};j}.
\]
We set $u^{\sigma}:=\left.  u\right\vert _{\Omega_{j}^{\sigma}}$, $\sigma
\in\left\{  -,+\right\}  $. The jump relations (\ref{jumprelSLP}),
(\ref{jumprelDLP}) imply%
\[
\left[  u\right]  _{\operatorname*{D};j}\left(  s\right)  =-\psi
_{\operatorname*{D};j}\quad\text{and\quad}\left[  u\right]
_{\operatorname*{N},j}^{\operatorname*{ext}}\left(  s\right)  =-\psi
_{\operatorname*{N};j}%
\]
while the relations%
\begin{align*}
\{\!\!\{u\}\!\!\}_{\operatorname*{D};j}\left(  s\right)   &  =\mathsf{V}%
_{j}\left(  s\right)  \psi_{\operatorname*{N};j}-\mathsf{K}_{j}\left(
s\right)  \psi_{\operatorname*{D};j},\\
\{\!\!\{u\}\!\!\}_{\operatorname*{N};j}^{\operatorname*{ext}}\left(  s\right)
&  =\mathsf{K}_{j}^{\prime}\left(  s\right)  \psi_{\operatorname*{N}%
;j}+\mathsf{W}_{j}\left(  s\right)  \psi_{\operatorname*{D};j}%
\end{align*}
follow directly from the definition of the boundary integral operators. A more
compact formulation is%
\[
\mathsf{C}_{j}\left(  s\right)
\mbox{\boldmath$ \psi$}%
_{j}=\left(
\begin{array}
[c]{c}%
\{\!\!\{u\}\!\!\}_{\operatorname*{D};j}\left(  s\right) \\
\{\!\!\{u\}\!\!\}_{\operatorname*{N};j}^{\operatorname*{ext}}\left(  s\right)
\end{array}
\right)  .
\]
Since $\mathsf{S}_{j}\left(  s\right)  \psi_{\operatorname*{N};j}$ and
$\mathsf{D}_{j}\left(  s\right)  \psi_{\operatorname*{D};j}\ $satisfy the
homogeneous PDE in $\Omega_{j}^{-}$ and $\Omega_{j}^{+}$ (cf. (\ref{SlpProp1}%
), (\ref{DlpProp1})) we may apply Green's identity (\ref{Green1}) and the
definition of the jumps and means (\ref{jumpdef}), (\ref{meandef}) to obtain
by a jump-average parallelogram identity\footnote{
For $\alpha,a_{1},a_{2},b_{1},b_{2}\in\mathbb{C}$ a direct calculation shows%
\[
\alpha\frac{a_{1}+b_{1}}{2}\left(  \overline{b_{2}}-\overline{a_{2}}\right)
+\overline{\alpha}\left(  \overline{b_{1}}-\overline{a_{1}}\right)  \left(
\frac{a_{2}+b_{2}}{2}\right)  =-\operatorname{Re}\left(  \alpha a_{1}%
\overline{a_{2}}+\alpha b_{1}\overline{b_{2}}\right)  +\operatorname*{i}%
\operatorname{Im}\left(  \alpha a_{1}\overline{b_{2}}+\overline{\alpha}%
a_{2}\overline{b_{1}}\right)
\]
so that%
\[
-\operatorname{Re}\left(  \alpha\frac{a_{1}+b_{1}}{2}\left(  \overline{b_{2}%
}-\overline{a_{2}}\right)  +\overline{\alpha}\left(  \overline{b_{1}%
}-\overline{a_{1}}\right)  \left(  \frac{a_{2}+b_{2}}{2}\right)  \right)
=\operatorname{Re}\left(  \alpha a_{1}\overline{a_{2}}+\alpha b_{1}%
\overline{b_{2}}\right)  .
\]
}:%
\begin{align*}
&  \operatorname{Re}\left\langle \mathsf{C}_{j}\left(  s\right)
\mbox{\boldmath$ \psi$}%
_{j},\overline{%
\mbox{\boldmath$ \psi$}%
_{j}}\right\rangle _{\mathbf{X}_{j}}\\
&  \quad=-\operatorname{Re}\left(  \left(  \frac{s}{\overline{s}}\right)
^{1/2}\left\langle \{\!\!\{u\}\!\!\}_{\operatorname*{D};j}\left(  s\right)
,\left[  \overline{u}\right]  _{\operatorname*{N},j}^{\operatorname*{ext}%
}\left(  s\right)  \right\rangle _{\Gamma_{j}}+\left(  \frac{\bar{s}}%
{s}\right)  ^{1/2}\left\langle \left[  \overline{u}\right]
_{\operatorname*{D};j}\left(  s\right)  ,\{\!\!\{u\}\!\!\}_{\operatorname*{N}%
;j}^{\operatorname*{ext}}\left(  s\right)  \right\rangle _{\Gamma_{j}}\right)
\\
&  \quad=\operatorname{Re}\left(  \left(  \frac{\overline{s}}{s}\right)
^{1/2}\left(  \left\langle \mathbb{A}_{j}^{+}\nabla u^{+},\nabla
\overline{u^{+}}\right\rangle _{\Omega_{j}^{+}}+s^{2}\left\langle p_{j}^{+}%
u^{+},%
\overline{u^{+}}\right\rangle _{\Omega_{j}^{+}}\right)  \right) \\
&  \quad\quad+\operatorname{Re}\left(  \left(  \frac{s}{\overline{s}}\right)
^{1/2}\left(  \left\langle \mathbb{A}_{j}^{-}\nabla\overline{u^{-}},\nabla
u^{-}\right\rangle _{\Omega_{j}^{-}}+s^{2}\left\langle p_{j}^{-}%
\overline{u^{-}},%
u^{-}\right\rangle _{\Omega_{j}^{-}}\right)  \right)  .
\end{align*}
As in the proof of Lemma \ref{LemLaxMilgram} we obtain%
\begin{equation}
\operatorname{Re}\left\langle \mathsf{C}_{j}\left(  s\right)
\mbox{\boldmath$ \psi$}%
_{j},\overline{%
\mbox{\boldmath$ \psi$}%
_{j}}\right\rangle _{\mathbf{X}_{j}}\geq\frac{\operatorname{Re}s}{\left\vert
s\right\vert }\lambda_{j}\left\Vert u\right\Vert _{H^{1}\left(  \mathbb{R}%
^{3}\backslash\Gamma_{j}\right)  ;s}^{2}. \label{CaldCoer1}%
\end{equation}
To estimate the right-hand side we start with%
\begin{equation}
  \begin{aligned}
    \left\Vert s^{-1/2}\psi_{\operatorname*{D};j}\right\Vert _{H^{1/2}\left(
        \Gamma_{j}\right)  }^{2}+\left\Vert \psi_{\operatorname*{N};j}\right\Vert
    _{H^{-1/2}\left(  \Gamma_{j}\right)  }^{2} & =\left\Vert s^{-1/2}\left[  u\right]
      _{\operatorname*{D};j}\left(  s\right)  \right\Vert _{H^{1/2}\left(
        \Gamma_{j}\right)  }^{2}\\
    &\quad+\left\Vert \left[  u\right]  _{\operatorname*{N}%
        ,j}^{\operatorname*{ext}}\left(  s\right)  \right\Vert _{H^{-1/2}\left(
        \Gamma_{j}\right)  }^{2}.
  \end{aligned}
  \label{CaldCoer2}%
\end{equation}
From (\ref{scaledtraceest}) and a triangle inequality we conclude that%
\begin{align}
\left\Vert s^{-1/2}\psi_{\operatorname*{D};j}\right\Vert _{H^{1/2}\left(
\Gamma_{j}\right)  }^{2}  &  =\left\Vert s^{-1/2}\left[  u\right]
_{\operatorname*{D};j}\left(  s\right)  \right\Vert _{H^{1/2}\left(
\Gamma_{j}\right)  }^{2}\label{CaldCoer3}\\
&  \leq2\sum_{\sigma\in\left\{  +,-\right\}  }\left\vert s\right\vert
^{-1}\left\Vert \gamma_{\operatorname*{D};j}^{\sigma}\left(  s\right)
u^{\sigma}\right\Vert _{H^{1/2}\left(  \Gamma_{j}\right)  }^{2}\leq
C\left\Vert u\right\Vert _{H^{1}\left(  \mathbb{R}^{3}\backslash\Gamma
_{j}\right)  ;s}^{2}.\nonumber
\end{align}
From (\ref{ljjumprel}) and by using the lifting $\mathsf{E}_{j}\left(
s\right)  $ as in Lemma \ref{LemScaledTraces}, we obtain%
\begin{align}
\left\Vert \psi_{\operatorname*{N};j}\right\Vert _{H^{-1/2}\left(  \Gamma
_{j}\right)  }  &  =\left\Vert \left[  u\right]  _{\operatorname*{N}%
,j}^{\operatorname*{ext}}\left(  s\right)  \right\Vert _{H^{-1/2}\left(
\Gamma_{j}\right)  }=\sup_{\phi\in H^{1/2}\left(  \Gamma_{j}\right)
\backslash\left\{  0\right\}  }\frac{\left\vert \left\langle \left[  u\right]
_{\operatorname*{N},j}^{\operatorname*{ext}}\left(  s\right)  ,\overline{\phi
}\right\rangle _{\Gamma_{j}}\right\vert }{\left\Vert \phi\right\Vert
_{H^{1/2}\left(  \Gamma_{j}\right)  }}\label{CaldCoer4}\\
&  =\sup_{\phi\in H^{1/2}\left(  \Gamma_{j}\right)  \backslash\left\{
0\right\}  }\frac{\left\vert \left(  \frac{\overline{s}}{s}\right)
^{1/2}\left\langle \left[  u\right]  _{\operatorname*{N},j}%
^{\operatorname*{ext}}\left(  s\right)  ,\gamma_{\operatorname*{D};j}\left(
s\right)  \overline{\mathsf{E}_{j}\left(  s\right)  \phi}\right\rangle
_{\Gamma_{j}}\right\vert }{\left\Vert \phi\right\Vert _{H^{1/2}\left(
\Gamma_{j}\right)  }}\nonumber\\
&  =\sup_{\phi\in H^{1/2}\left(  \Gamma_{j}\right)  \backslash\left\{
0\right\}  }\frac{\left\vert \ell_{j}\left(  s\right)  \left(  u,\mathsf{E}%
_{j}\left(  s\right)  \phi\right)  \right\vert }{\left\Vert \phi\right\Vert
_{H^{1/2}\left(  \Gamma_{j}\right)  }}\nonumber\\
&  \overset{\text{Lem. \ref{LemLaxMilgram}}}{\leq}\Lambda_{j}\left\Vert
u\right\Vert _{H^{1}\left(  \mathbb{R}^{3}\backslash\Gamma_{j}\right)  ;s}%
\sup_{\phi\in H^{1/2}\left(  \Gamma_{j}\right)  \backslash\left\{  0\right\}
}\frac{\left\Vert \mathsf{E}_{j}\left(  s\right)  \phi\right\Vert
_{H^{1}\left(  \mathbb{R}^{3}\right)  ;s}}{\left\Vert \phi\right\Vert
_{H^{1/2}\left(  \Gamma_{j}\right)  }}\nonumber\\
&  \overset{\text{(\ref{tracelifting})}}{\leq}C\Lambda_{j}\left\Vert
u\right\Vert _{H^{1}\left(  \mathbb{R}^{3}\backslash\Gamma_{j}\right)
;s}.\nonumber
\end{align}
The combination of (\ref{CaldCoer1})-(\ref{CaldCoer4}) leads to the coercivity
estimate%
\[
\operatorname{Re}\left\langle \mathsf{C}_{j}\left(  s\right)
\mbox{\boldmath$ \psi$}%
_{j},\overline{%
\mbox{\boldmath$ \psi$}%
_{j}}\right\rangle _{\mathbf{X}_{j}}\geq\tilde{c}\frac{\lambda_{j}}%
{1+\Lambda_{j}^{2}}\frac{\operatorname{Re}s}{\left\vert s\right\vert ^{2}%
}\left\Vert
\mbox{\boldmath$ \psi$}%
_{j}\right\Vert _{\mathbf{X}_{j}}^{2}.
\]
\medskip\medskip

For the continuity estimate we obtain for any 
\begin{equation*}
\mbox{\boldmath$ \psi$}%
_{j}=\left(  \psi_{\operatorname*{D};j},\psi_{\operatorname*{N};j}\right) 
\mbox{\boldmath$ \phi$}%
=\left(  \varphi_{\operatorname*{D};j},\varphi_{\operatorname*{N};j}\right)
\in\mathbf{X}_{j}
\end{equation*}
from Lemma \ref{Lem_s_explicit_operator_est}%
\begin{align*}
&  \left\vert \left\langle \mathsf{C}_{j}\left(  s\right)
\mbox{\boldmath$ \psi$}%
_{j},\overline{%
\mbox{\boldmath$ \phi$}%
_{j}}\right\rangle _{\mathbf{X}_{j}}\right\vert =\left\vert \left\langle
\left(
\begin{array}
[c]{l}%
-\mathsf{K}_{j}\left(  s\right)  \psi_{\operatorname*{D};j}+\mathsf{V}%
_{j}\left(  s\right)  \psi_{\operatorname*{N};j}\\
\mathsf{W}_{j}\left(  s\right)  \psi_{\operatorname*{D};j}+\mathsf{K}%
_{j}^{\prime}\left(  s\right)  \psi_{\operatorname*{N};j}%
\end{array}
\right)  ,\overline{\left(
\begin{array}
[c]{c}%
\varphi_{\operatorname*{D};j}\\
\varphi_{\operatorname*{N};j}%
\end{array}
\right)  }\right\rangle _{\mathbf{X}_{j}}\right\vert \\
&  \qquad=\left\vert \left\langle -\mathsf{K}_{j}\left(  s\right)
\psi_{\operatorname*{D};j}+\mathsf{V}_{j}\left(  s\right)  \psi
_{\operatorname*{N};j},\overline{\varphi_{\operatorname*{N};j}}\right\rangle
_{\Gamma_{j}}+\left\langle \overline{\varphi_{\operatorname*{D};j}}%
,\mathsf{W}_{j}\left(  s\right)  \psi_{\operatorname*{D};j}+\mathsf{K}%
_{j}^{\prime}\left(  s\right)  \psi_{\operatorname*{N};j}\right\rangle
_{\Gamma_{j}}\right\vert \\
&  \qquad\leq C\frac{1}{\lambda_{j}}\frac{\left\vert s\right\vert
}{\operatorname{Re}s}\left(  \Lambda_{j}\left\vert s\right\vert ^{1/2}%
\left\Vert \psi_{\operatorname*{D};j}\right\Vert _{H^{1/2}\left(  \Gamma
_{j}\right)  }\left\Vert \varphi_{\operatorname*{N};j}\right\Vert
  _{H^{-1/2}\left(  \Gamma_{j}\right)  }\right.\\
&  \qquad\quad+\left\vert s\right\vert \left\Vert
\psi_{\operatorname*{N};j}\right\Vert _{H^{-1/2}\left(  \Gamma_{j}\right)
}\left\Vert \varphi_{\operatorname*{N};j}\right\Vert _{H^{-1/2}\left(
\Gamma_{j}\right)  } \\
&  \qquad\quad\left.  +\Lambda_{j}^{2}\left\Vert \psi_{\operatorname*{D}%
;j}\right\Vert _{H^{1/2}\left(  \Gamma_{j}\right)  }\left\Vert \varphi
_{\operatorname*{D};j}\right\Vert _{H^{1/2}\left(  \Gamma_{j}\right)
}+\Lambda_{j}\left\vert s\right\vert ^{1/2}\left\Vert \psi_{\operatorname*{N}%
;j}\right\Vert _{H^{-1/2}\left(  \Gamma_{j}\right)  }\left\Vert \varphi
_{\operatorname*{D};j}\right\Vert _{H^{1/2}\left(  \Gamma_{j}\right)  }\right)
\\
&  \qquad\leq {C'}\frac{1}{\lambda_{j}}\frac{\left\vert s\right\vert ^{2}%
}{\operatorname{Re}s}\left(  \Lambda_{j}^{2}\left\Vert s^{-1/2}\psi
_{\operatorname*{D};j}\right\Vert _{H^{1/2}\left(  \Gamma_{j}\right)  }%
^{2}+\left\Vert \psi_{\operatorname*{N};j}\right\Vert _{H^{1/2}\left(
\Gamma_{j}\right)  }^{2}\right)  ^{1/2}\times\\
&  \qquad\quad\times\left(  \Lambda_{j}^{2}\left\Vert s^{-1/2}\varphi
_{\operatorname*{D};j}\right\Vert _{H^{1/2}\left(  \Gamma_{j}\right)  }%
^{2}+\left\Vert \varphi_{\operatorname*{N};j}\right\Vert _{H^{-1/2}\left(
\Gamma_{j}\right)  }^{2}\right)  ^{1/2}\\
&  \qquad\leq {C''}\frac{1+\Lambda_{j}}{\lambda_{j}}\frac{\left\vert s\right\vert
^{2}}{\operatorname{Re}s}\left\Vert
\mbox{\boldmath$ \psi$}%
_{j}\right\Vert _{\mathbf{X}_{j}}\left\Vert
\mbox{\boldmath$ \phi$}%
_{j}\right\Vert _{\mathbf{X}_{j}}.
\end{align*}
\end{proof}

A summation of the local coercivity estimates (of the local continuity
estimates, resp.) over all subdomains leads to the following global coercivity
(global continuity, resp.).

\begin{corollary}
  Let Assumption \ref{Acoeff} be satisfied. The sesquilinear form
  \begin{equation*}
    \left\langle
      \mathsf{C}\left(  s\right)  \cdot,\overline{\cdot}\right\rangle _{\mathbb{X}%
    }:\mathbb{X}\left(  \mathcal{P}_{\Omega}\right)  \times\mathbb{X}\left(
      \mathcal{P}_{\Omega}\right)  \rightarrow\mathbb{C}
  \end{equation*}
  is coercive: for any $%
\mbox{\boldmath$ \psi$}%
\in\mathbb{X}\left(  \mathcal{P}_{\Omega}\right)  $ it holds%
\begin{equation}
\operatorname{Re}\left\langle \mathsf{C}\left(  s\right)
\mbox{\boldmath$ \psi$}%
,\overline{%
\mbox{\boldmath$ \psi$}%
}\right\rangle _{\mathbb{X}}\geq c\frac{\lambda}{1+\Lambda^{2}}\frac
{\operatorname{Re}s}{\left\vert s\right\vert ^{2}}\left\Vert
\mbox{\boldmath$ \psi$}%
\right\Vert _{\mathbb{X}}^{2}; \label{CaldCoerc}%
\end{equation}
and continuous: for any $%
\mbox{\boldmath$ \psi$}%
,%
\mbox{\boldmath$ \phi$}%
\in\mathbb{X}\left(  \mathcal{P}_{\Omega}\right)  $ it holds%
\begin{equation}
\left\vert \left\langle \mathsf{C}\left(  s\right)
\mbox{\boldmath$ \psi$}%
,\overline{%
\mbox{\boldmath$ \phi$}%
}\right\rangle _{\mathbb{X}}\right\vert \leq C\frac{1+\Lambda}{\lambda}%
\frac{\left\vert s\right\vert ^{2}}{\operatorname{Re}s}\left\Vert
\mbox{\boldmath$ \psi$}%
\right\Vert _{\mathbb{X}}\left\Vert
\mbox{\boldmath$ \phi$}%
\right\Vert _{\mathbb{X}} \label{Caldcont}%
\end{equation}
with $\lambda:=\min_{1\leq j\leq n_{\Omega}}\lambda_{j}$ and $\Lambda
:=\max_{1\leq j\leq n_{\Omega}}\Lambda_{j}$.
\end{corollary}

We have collected all prerequisites to prove the well-posedness of the non-local
variational problem on the skeleton (\ref{singskeleq}) in single trace spaces.

\begin{theorem}\label{ThmWellPosedSkeletonEq}
  Let Assumption \ref{Acoeff} be satisfied. The
  sesquilinear form
  \begin{equation*}
    \left(  c\left(  s\right)  \left(  \cdot,\cdot\right)
      -\frac{1}{2}\left\langle \cdot,\overline{\cdot}\right\rangle _{\mathbb{X}%
      }\right)  :\mathbb{X}_{0}^{\operatorname{single}}\left(  \mathcal{P}_{\Omega
      }\right)  \times\mathbb{X}_{0}^{\operatorname{single}}\left(  \mathcal{P}%
      _{\Omega}\right)  \rightarrow\mathbb{C}
  \end{equation*}
  is coercive and continuous: for any $\mbox{\boldmath$ \alpha$}
  \in\mathbb{X}_{0}^{\operatorname{single}}\left( \mathcal{P}_{\Omega}\right)$ and
  $\mbox{\boldmath$ \psi$}, \mbox{\boldmath$ \phi$} \in\mathbb{X}\left(
    \mathcal{P}_{\Omega}\right) $ holds true
  \rhc{$\left\langle \boldmath{\alpha},\overline{\boldmath{\alpha}}\right\rangle
    _{\mathbb{X}}=0$}, and
\begin{align*}
\operatorname{Re}\left(  c\left(  s\right)  \left(
\mbox{\boldmath$ \alpha$}%
,%
\mbox{\boldmath$ \alpha$}%
\right) \right)   &  \geq c\frac{\lambda}{1+\Lambda^{2}%
}\frac{\operatorname{Re}s}{\left\vert s\right\vert ^{2}}\left\Vert
\mbox{\boldmath$ \alpha$}%
\right\Vert _{\mathbb{X}}^{2},\\
  \abs{c(s) \left(\boldsymbol{\psi}, \boldsymbol{\phi} \right) -
  \frac{1}{2} \sesquilinear{\boldsymbol{\psi}}{\boldsymbol{\phi}}{\mathbb{X}}}  &  \leq \left(\frac{1}{2} + C \frac{1+\Lambda}{\lambda} \frac{\abs{s}^{2}}{\operatorname{Re}s}\right) \norm{\boldsymbol{\psi}}{\mathbb{X}} \norm{\boldsymbol{\phi}}{\mathbb{X}}.
\end{align*}
For any $%
\mbox{\boldmath$ \beta$}%
\in\mathbb{X}\left(  \mathcal{P}_{\Omega}\right)  $, the variational problem
(\ref{singskeleq}) has a solution $\mathbf{u}^{\operatorname{single}}%
\in\mathbb{X}_{0}^{\operatorname{single}}\left(  \mathcal{P}_{\Omega}\right)
$ which is unique and satisfies%
\begin{equation}
\left\Vert \mathbf{u}^{\operatorname{single}}\right\Vert _{\mathbb{X}}\leq
C\frac{\left\vert s\right\vert ^{9/2}}{\left(  \operatorname{Re}s\right)
^{2}}\left\Vert
\mbox{\boldmath$ \beta$}%
\right\Vert _{\mathbb{X}}, \label{wellposednessbound}%
\end{equation}
where $C$ only depends on $\lambda$, $\Lambda$, $s_{0}$, and on the domain
$\Omega$ via trace estimates.
\end{theorem}

%
\begin{proof}
Let $%
\mbox{\boldmath$ \alpha$}%
=\left(
\mbox{\boldmath$ \alpha$}%
_{j}\right)  _{j=1}^{n_{\Omega}}\in\mathbb{X}_{0}^{\operatorname{single}%
}\left(  \mathcal{P}_{\Omega}\right)  $ with $%
\mbox{\boldmath$ \alpha$}%
_{j}=\left(  \alpha_{\operatorname*{D};j},\alpha_{\operatorname*{N};j}\right)
$ and\\
$%
\mbox{\boldmath$ \phi$}%
=\left(
\mbox{\boldmath$ \phi$}%
_{j}\right)  _{j=1}^{n_{\Omega}},%
\mbox{\boldmath$ \psi$}%
=\left(
\mbox{\boldmath$ \psi$}%
_{j}\right)  _{j=1}^{n_{\Omega}}\in\mathbb{X}\left(  \mathcal{P}_{\Omega
}\right)  $ with $%
\mbox{\boldmath$ \psi$}%
_{j}=\left(  \psi_{\operatorname*{D};j},\psi_{\operatorname*{N};j}\right)  $
and $%
\mbox{\boldmath$ \phi$}%
_{j}=\left(  \varphi_{\operatorname*{D};j},\varphi_{\operatorname*{N}%
;j}\right)  $. Then%
\[
\operatorname{Re}\left(  c\left(  s\right)  \left(
\mbox{\boldmath$ \alpha$}%
,%
\mbox{\boldmath$ \alpha$}%
\right)  -\frac{1}{2}\left\langle
\mbox{\boldmath$ \alpha$}%
,\overline{%
\mbox{\boldmath$ \alpha$}%
}\right\rangle _{\mathbb{X}}\right)  =\operatorname{Re}c\left(  s\right)
\left(
\mbox{\boldmath$ \alpha$}%
,%
\mbox{\boldmath$ \alpha$}%
\right)
\]
owing to the self-polarity of the single trace space, see \cite[Lem.
4.1]{Hiptmair_multiple_trace}, \cite[Remark~55]{ClHiptJH}. Thus, the coercivity estimate follows
from (\ref{CaldCoerc}):%
\[
\operatorname{Re}\left(  c\left(  s\right)  \left(
\mbox{\boldmath$ \alpha$}%
,%
\mbox{\boldmath$ \alpha$}%
\right)  -\frac{1}{2}\left\langle
\mbox{\boldmath$ \alpha$}%
,\overline{%
\mbox{\boldmath$ \alpha$}%
}\right\rangle _{\mathbb{X}}\right)  \geq c\frac{\lambda}{1+\Lambda^{2}}%
\frac{\operatorname{Re}s}{\left\vert s\right\vert ^{2}}\left\Vert
\mbox{\boldmath$ \alpha$}%
\right\Vert _{\mathbb{X}}^{2}.
\]
The continuity estimate follows by combining (\ref{Caldcont}) with%
\begin{align*}
\left\vert \left\langle
\mbox{\boldmath$ \psi$}%
,\overline{%
\mbox{\boldmath$ \phi$}%
}\right\rangle _{\mathbb{X}}\right\vert  &  \leq\sum_{1\leq j\leq n_{\Omega}%
}\left\vert \left\langle \psi_{\operatorname*{D};j},\overline{\varphi
_{\operatorname*{N};j}}\right\rangle _{\Gamma_{j}}+\left\langle \psi
_{\operatorname*{N};j},\overline{\varphi_{\operatorname*{D};j}}\right\rangle
_{\Gamma_{j}}\right\vert \\
&  \leq\sum_{1\leq j\leq n_{\Omega}}\left(  \left\Vert \psi_{\operatorname*{D}%
;j}\right\Vert _{H^{1/2}\left(  \Gamma_{j}\right)  }\left\Vert \varphi
_{\operatorname*{N};j}\right\Vert _{H^{-1/2}\left(  \Gamma_{j}\right)
  }\right.\\
                                         & \quad+ \left.
                                           \left\Vert \psi_{\operatorname*{N};j}\right\Vert _{H^{-1/2}\left(
\Gamma_{j}\right)  }\left\Vert \varphi_{\operatorname*{D};j}\right\Vert
_{H^{1/2}\left(  \Gamma_{j}\right)  }\right) \\
&  \leq\sum_{1\leq j\leq n_{\Omega}}\left\Vert
\mbox{\boldmath$ \psi$}%
_{j}\right\Vert _{\mathbf{X}_{j}}\left\Vert
\mbox{\boldmath$ \phi$}%
_{j}\right\Vert _{\mathbf{X}_{j}}\leq\left\Vert
\mbox{\boldmath$ \psi$}%
\right\Vert _{\mathbb{X}}\left\Vert
\mbox{\boldmath$ \phi$}%
\right\Vert _{\mathbb{X}}.
\end{align*}
In particular, the continuity of
$\left( c\left( s\right) \left( \cdot,\cdot\right) -\frac{1}{2}\left\langle
    \cdot,\overline{\cdot}\right\rangle _{\mathbb{X}%
  }\right) $ implies that for any $\mbox{\boldmath$ \beta$}\in\mathbb{X}\left(
  \mathcal{P}_{\Omega}\right) $ the form $\left( c\left(s\right)
  \left(\mbox{\boldmath$ \beta$} \left( s\right) ,\cdot\right)
  -\frac{1}{2}\left\langle\mbox{\boldmath$ \beta$} \left( s\right)
    ,\overline{\cdot}\right\rangle _{\mathbb{X}}\right)
:\mathbb{X}_{0}^{\operatorname{single}}\left( \mathcal{P}_{\Omega}\right)
\rightarrow\mathbb{C}$ defines an anti-linear operator with upper bound
$\left.  \left( \frac{1}{2}+C\frac{1+\Lambda}{\lambda}\frac{\left\vert s\right\vert
      ^{2}}{\operatorname{Re}s}\right) \left\vert s\right\vert ^{1/2}\left\Vert
    \mbox{\boldmath$ \beta$}\right\Vert _{\mathbb{X}}\right.  $ for its norm. By the
Lax-Milgram theorem we infer well-posedness of (\ref{singskeleq}) and the bound in
(\ref{wellposednessbound}).
\end{proof}

\rhc{%
\begin{remark}
    Our approach also paves the way for pursuing a multi-trace formulation
    as in \cite{ClHiptJH}; all the ingredients are available! We expect
    that the multi-trace formulation becomes well-posed and the resulting equations are well-suited for
    operator preconditioning. In this paper, we have focus on the single-trace formulation
    since it directly inherits the stability of the boundary value problem.
\end{remark}}

\section{Conclusion\label{SecConclusion}}

In this paper, we have considered acoustic transmission problems with mixed boundary
conditions, \rhc{variable coefficients and absorption}. We have developed a
general approach to transform these equations to non-local skeleton equations in such a
way that the resulting variational form is continuous and coercive so that well-posedness
follows by the Lax-Milgram theorem. The transformation is based on Green's representation
formula involving single and double layer potentials which are defined as solutions of
some variational full space problems without relying on the explicit knowledge of the
Green's function. The paper can be regarded as a generalization of \cite{EbFlHiSa} by
allowing for unbounded domains (full space/half space) and variable coefficients in the
subdomains.

In contrast to other methods such as the \textit{indirect} method of boundary integral
equations (see, e.g., \cite[Chap. 3.4.1]{SauterSchwab2010}) the well-posedness of the
non-local skeleton (integral) equation follows directly from the well-posedness of the
auxiliary variational problems in full space.


Another important contribution of this work is the completely $s$-explicit nature of all
estimates, \rhc{$s$ the frequency parameter}. Its significance is due to the possibility
to apply our boundary integral equation method to transform the space-time wave
transmission problem (in analogy to (\ref{generalTP})) to an integro-differential equation
which may serve as a starting point for its discretization by convolution quadrature. The
well-posedness of this integro-differential equation follows from the coercivity and
continuity of the variational skeleton equation (\ref{singskeleq}) via \textit{operational
  calculus}; for details we refer to \cite{EbFlHiSa}, \cite{banjai_coupling},
\cite{sayas_book}, \cite{banjai_sayas}. We also mention that the restriction to mixed
Dirichlet and Neumann boundary conditions was merely done to reduce technicalities:
Dirichlet-to-Neumann boundary conditions and impedance conditions can be incorporated into
the variational skeleton equation following the approach in \cite{EbFlHiSa}.

\appendix

\section{Proof of Lemma \ref{Lem_s_explicit_operator_est}\label{AppProof}}

The proof of Lemma \ref{Lem_s_explicit_operator_est} is an adaptation of the
arguments in \cite[Prop. 16, 19]{LaSa2009} to our setting; see also \cite[Lem.
5.2]{Barton_potential}. In this appendix, we present the proof to show that
the known arguments apply to our general setting.

\begin{proof}[of Lemma \ref{Lem_s_explicit_operator_est}.] Let $\varphi\in
H^{-1/2}\left(  \Gamma\right)  $ and set $u:=\mathsf{S}_{j}\left(  s\right)
\varphi$. The jump relations for the single layer potential (cf.
(\ref{jumprelSLP})) imply $\gamma_{\operatorname*{D};j}\left(  s\right)
u=\mathsf{V}\left(  s\right)  \varphi$ and $\left[  u\right]
_{\operatorname*{N};j}^{\operatorname*{ext}}\left(  s\right)  =-\varphi$.
Then, we have%
\begin{align*}
\operatorname{Re}\left\langle \varphi,\overline{\mathsf{V}_{j}\left(
s\right)  \varphi}\right\rangle _{\Gamma_{j}}  &  =\operatorname{Re}%
\left\langle -\left[  u\right]  _{\operatorname*{N};j}^{\operatorname*{ext}%
}\left(  s\right)  ,\gamma_{\operatorname*{D};j}\left(  \overline{s}\right)
\overline{u}\right\rangle _{\Gamma_{j}}\\
&  \overset{\text{(\ref{gammadjs})}}{=}\operatorname{Re}\left(  \left(
\frac{\overline{s}}{s}\right)  ^{1/2}\left\langle -\left[  u\right]
_{\operatorname*{N};j}^{\operatorname*{ext}}\left(  s\right)  ,\gamma
_{\operatorname*{D};j}\left(  s\right)  \overline{u}\right\rangle _{\Gamma
_{j}}\right)  .
\end{align*}
We employ (\ref{ljjumprel}) with $v=w=u$ and $\lambda\left(  \cdot\right)  $,
$\lambda_{j}$ as in Lem. \ref{LemLaxMilgram} to obtain (cf. (\ref{ljcoercive}%
))%
\begin{align*}
\operatorname{Re} & \left(  \left(  \frac{\overline{s}}{s}\right)  ^{1/2}%
\left\langle -\left[  u\right]  _{\operatorname*{N};j}^{\operatorname*{ext}%
}\left(  s\right)  ,\gamma_{\operatorname*{D};j}\left(  s\right)  \overline
                    {u}\right\rangle _{\Gamma_{j}}\right)
                    =\operatorname{Re}\left(  \left(
\frac{\overline{s}}{s}\right)  ^{1/2}\ell_{j}\left(  s\right)  \left(
u,u\right)  \right) \\
&  =\frac{\operatorname{Re}s}{\left\vert s\right\vert }\left(  \left\langle
\mathbb{A}_{j}^{\operatorname*{ext}}\nabla u,\nabla\overline{u}\right\rangle
_{\mathbb{R}^{3}}+\left\vert s\right\vert ^{2}\left\langle p_{j}%
^{\operatorname*{ext}}u,\overline{u}\right\rangle _{\mathbb{R}^{3}}\right) \\
&  \geq\frac{\operatorname{Re}s}{\left\vert s\right\vert }\left(
\lambda\left(  \mathbb{A}_{j}^{\operatorname*{ext}}\right)  \left\vert \nabla
u\right\vert _{\mathbf{L}^{2}\left(  \mathbb{R}^{3}\right)  }^{2}%
+\lambda\left(  p_{j}^{\operatorname*{ext}}\right)  \left\vert s\right\vert
^{2}\left\Vert u\right\Vert _{L^{2}\left(  \mathbb{R}^{3}\right)  }^{2}\right)
\\
&  \geq\frac{\operatorname{Re}s}{\left\vert s\right\vert }\lambda
_{j}\left\Vert u\right\Vert _{H^{1}\left(  \mathbb{R}^{3}\right)  ;s}^{2}.
\end{align*}
Finally, the coercivity estimate (\ref{slbieboundsa}) for $\mathsf{V}\left(
s\right)  $ follows from (\ref{conormaltraceest})

Next, we prove the continuity of the single layer operator. For $\varphi\in
H^{-1/2}\left(  \Gamma_{j}\right)  $, let $v:=\mathsf{S}_{j}\left(  s\right)
\varphi$. Then (\ref{SLPest}) follows from%
\begin{align*}
\frac{\operatorname{Re}s}{\left\vert s\right\vert }\lambda_{j}\left\Vert
v\right\Vert _{H^{1}\left(  \mathbb{R}^{3}\right)  ;s}^{2}  &  \leq
\operatorname{Re}\left\langle \varphi,\overline{\mathsf{V}_{j}\left(
s\right)  \varphi}\right\rangle _{\Gamma_{j}}\leq\left\Vert \varphi\right\Vert
_{H^{-1/2}\left(  \Gamma_{j}\right)  }\left\Vert \gamma_{\operatorname*{D}%
;j}\left(  s\right)  v\right\Vert _{H^{1/2}\left(  \Gamma_{j}\right)  }\\
&  \overset{\text{(\ref{scaledtraceest})}}{\leq}C\left\vert s\right\vert
^{1/2}\left\Vert \varphi\right\Vert _{H^{-1/2}\left(  \Gamma_{j}\right)
}\left\Vert v\right\Vert _{H^{1}\left(  \mathbb{R}^{3}\right)  ;s}.
\end{align*}
The continuity (\ref{slbieboundsb}) of $\mathsf{V}\left(  s\right)  $ is a
direct consequence of the estimate%
\begin{align*}
\left\vert \left\langle \mathsf{V}_{j}\left(  s\right)  \varphi,\overline
{\psi}\right\rangle _{\Gamma_{j}}\right\vert  &  =\left\vert \left\langle
\gamma_{\operatorname*{D};j}\left(  s\right)  \mathsf{S}_{j}\left(  s\right)
\varphi,\overline{\psi}\right\rangle _{\Gamma_{j}}\right\vert \leq\left\Vert
\gamma_{\operatorname*{D};j}\left(  s\right)  \mathsf{S}_{j}\left(  s\right)
\varphi\right\Vert _{H^{1/2}\left(  \Gamma_{j}\right)  }\left\Vert
\psi\right\Vert _{H^{-1/2}\left(  \Gamma_{j}\right)  }\\
&  \leq C\left\vert s\right\vert ^{1/2}\left\Vert \mathsf{S}_{j}\left(
s\right)  \varphi\right\Vert _{H^{1}\left(  \mathbb{R}^{3}\right)
  ;s}\left\Vert \psi\right\Vert _{H^{-1/2}\left(  \Gamma_{j}\right)  }\\
                                              &
                                                \leq
C\frac{\left\vert s\right\vert ^{2}}{\lambda_{j}\operatorname{Re}s}\left\Vert
\varphi\right\Vert _{H^{-1/2}\left(  \Gamma_{j}\right)  }\left\Vert
\psi\right\Vert _{H^{-1/2}\left(  \Gamma_{j}\right)  }.
\end{align*}
Finally, the dual double layer boundary integral operator\emph{ }%
$\mathsf{K}_{j}^{\prime}\left(  s\right)  $ can be estimated by using the
mapping properties of $\mathsf{S}_{j}$ and $\gamma_{\operatorname*{N}%
,j}^{\operatorname*{ext}}$. Let $v^{\sigma}:=\left.  \left(  \mathsf{S}%
_{j}\left(  s\right)  \varphi\right)  \right\vert _{\Omega_{j}^{\sigma}}$,
$\sigma\in\left\{  +,-\right\}  $. Then, we have for all $\varphi\in
H^{-1/2}\left(  \Gamma_{j}\right)  :$%
\begin{align*}
\left\Vert \mathsf{K}_{j}^{\prime}\left(  s\right)  \varphi\right\Vert
_{H^{-1/2}\left(  \Gamma_{j}\right)  }  &  =\left\Vert \{\!\!\{\mathsf{S}%
_{j}\left(  s\right)  \varphi\}\!\!\}_{\operatorname*{N};j}%
^{\operatorname*{ext}}\left(  s\right)  \right\Vert _{H^{-1/2}\left(
\Gamma_{j}\right)  }\leq\sum_{\sigma\in\left\{  +,-\right\}  }\left\Vert
\gamma_{\operatorname*{N};j}^{\operatorname*{ext},\sigma}\left(  s\right)
v^{\sigma}\right\Vert _{H^{-1/2}\left(  \Gamma_{j}\right)  }\\
&  \overset{\text{(\ref{conormaltraceest})}}{\leq}C\Lambda_{j}\sum_{\sigma
\in\left\{  +,-\right\}  }\left\Vert v^{\sigma}\right\Vert _{H^{1}\left(
\Omega_{j}^{\sigma}\right)  ;s}\overset{\text{(\ref{SLPest})}}{\leq}%
C\frac{\Lambda_{j}}{\lambda_{j}}\frac{\left\vert s\right\vert ^{3/2}%
}{\operatorname{Re}s}\left\Vert \varphi\right\Vert _{H^{-1/2}\left(
\Gamma_{j}\right)  }.
\end{align*}

Next, we investigate the mapping properties of the operators related to the
double layer potential and start with the coercivity estimate of
$\mathsf{W}_{j}\left(  s\right)  $. Let $\psi\in H^{1/2}\left(  \Gamma
_{j}\right)  $ and set $u:=\mathsf{D}_{j}\left(  s\right)  \psi$. The jump
relations for the double layer potentials (cf. (\ref{jumprelDLP})) imply
$\gamma_{\operatorname*{N};j}^{\operatorname*{ext}}\left(  s\right)
u=-\mathsf{W}_{j}\left(  s\right)  \psi$ and $\left[  u\right]
_{\operatorname*{D};j}\left(  s\right)  =\psi$. Then, we have%
\begin{align*}
\operatorname{Re}\left\langle \mathsf{W}_{j}\left(  s\right)  \psi
,\overline{\psi}\right\rangle _{\Gamma_{j}}  &  =\operatorname{Re}\left\langle
-\gamma_{\operatorname*{N};j}^{\operatorname*{ext}}\left(  s\right)  u,\left[
\overline{u}\right]  _{\operatorname*{D};j}\left(  \overline{s}\right)
\right\rangle _{\Gamma_{j}}\\
&  \overset{\text{(\ref{gammadjs})}}{=}\operatorname{Re}\left(  \left(
\frac{\overline{s}}{s}\right)  ^{1/2}\left\langle -\gamma_{\operatorname*{N}%
;j}^{\operatorname*{ext}}\left(  s\right)  u,\left[  \overline{u}\right]
_{\operatorname*{D};j}\left(  s\right)  \right\rangle _{\Gamma_{j}}\right)  .
\end{align*}
We employ (\ref{ljconjump}) with $v=w=u$ and $\lambda\left(  \cdot\right)  $,
$\lambda_{j}$ as in Lem. \ref{LemLaxMilgram} to obtain (cf. (\ref{ljcoercive}%
)) with $\mathbb{A}_{j}^{\sigma}:=\left.  \mathbb{A}_{j}^{\operatorname*{ext}%
}\right\vert _{\Omega_{j}^{\sigma}}$ and $p_{j}^{\sigma}:=\left.
p_{j}^{\operatorname*{ext}}\right\vert _{\Omega_{j}^{\sigma}}$, $\sigma
\in\left\{  +,-\right\}  $:%
\begin{align}
\operatorname{Re}\left\langle \mathsf{W}_{j}\left(  s\right)  \psi
,\overline{\psi}\right\rangle _{\Gamma_{j}}  &  =\operatorname{Re}\left(
\left(  \frac{\overline{s}}{s}\right)  ^{1/2}\sum_{\sigma\in\left\{
+,-\right\}  }\left(  \left\langle \mathbb{A}_{j}^{\sigma}\nabla u^{\sigma
},\nabla\overline{u^{\sigma}}\right\rangle _{\Omega_{j}^{\sigma}}%
+s^{2}\left\langle p_{j}^{\sigma}u^{\sigma},\overline{u^{\sigma}}\right\rangle
_{\Omega_{j}^{\sigma}}\right)  \right) \nonumber\\
&  =\frac{\operatorname{Re}s}{\left\vert s\right\vert }\sum_{\sigma\in\left\{
+,-\right\}  }\left(  \left\langle \mathbb{A}_{j}^{\sigma}\nabla u^{\sigma
},\nabla\overline{u^{\sigma}}\right\rangle _{\Omega_{j}^{\sigma}}+\left\vert
s\right\vert ^{2}\left\langle p_{j}^{\sigma}u^{\sigma},\overline{u^{\sigma}%
}\right\rangle _{\Omega_{j}^{\sigma}}\right) \nonumber\\
&  \geq\frac{\operatorname{Re}s}{\left\vert s\right\vert }\lambda_{j}%
\sum_{\sigma\in\left\{  +,-\right\}  }\left\Vert u^{\sigma}\right\Vert
_{H^{1}\left(  \Omega_{j}^{\sigma}\right)  ;s}^{2}=\frac{\operatorname{Re}%
s}{\left\vert s\right\vert }\lambda_{j}\left\Vert u\right\Vert _{H^{1}\left(
\mathbb{R}^{3}\backslash\Gamma_{j}\right)  ;s}^{2}. \label{jumplowerest}%
\end{align}
Thus, the coercivity relation (\ref{Wcoerc}) follows from the trace estimate
(cf. (\ref{scaledtraceest}))%
\begin{align*}
\left\Vert \psi\right\Vert _{H^{1/2}\left(  \Gamma_{j}\right)  }^{2}  &
=\left\Vert \left[  u\right]  _{\operatorname*{D};j}\left(  s\right)
\right\Vert _{H^{1/2}\left(  \Gamma_{j}\right)  }^{2}\leq\sum_{\sigma
\in\left\{  +,-\right\}  }\left\Vert \gamma_{\operatorname*{D};j}^{\sigma
}\left(  s\right)  u^{\sigma}\right\Vert _{H^{1/2}\left(  \Gamma_{j}\right)
}^{2}\\
&  \leq C\left\vert s\right\vert \sum_{\sigma\in\left\{  +,-\right\}
}\left\Vert u^{\sigma}\right\Vert _{H^{1}\left(  \Omega_{j}^{\sigma}\right)
;s}^{2}=C\left\vert s\right\vert \left\Vert u\right\Vert _{H^{1}\left(
\mathbb{R}^{3}\backslash\Gamma_{j}\right)  ;s}^{2}.
\end{align*}
Next, we prove the continuity of the double layer operator. For $\psi\in
H^{1/2}\left(  \Gamma_{j}\right)  $, let $u:=\mathsf{D}_{j}\left(  s\right)
\psi$. Then (\ref{DLPest}) follows from%
\begin{align*}
\frac{\operatorname{Re}s}{\left\vert s\right\vert }\lambda_{j}\left\Vert
u\right\Vert _{H^{1}\left(  \mathbb{R}^{3}\backslash\Gamma_{j}\right)
;s}^{2}  &  \leq\operatorname{Re}\left\langle \mathsf{W}_{j}\left(  s\right)
\psi,\overline{\psi}\right\rangle _{\Gamma_{j}}=\operatorname{Re}\left\langle
-\gamma_{\operatorname*{N};j}^{\operatorname*{ext}}\left(  s\right)
u,\overline{\psi}\right\rangle _{\Gamma_{j}}\\
&  \leq\left\Vert \gamma_{\operatorname*{N};j}^{\operatorname*{ext}}\left(
s\right)  u\right\Vert _{H^{-1/2}\left(  \Gamma_{j}\right)  }\left\Vert
\psi\right\Vert _{H^{1/2}\left(  \Gamma_{j}\right)  }\\
&  \overset{\text{(\ref{conormaltraceest})}}{\leq}C\Lambda_{j}\left\Vert
u\right\Vert _{H^{1}\left(  \mathbb{R}^{3}\backslash\Gamma_{j}\right)
;s}\left\Vert \psi\right\Vert _{H^{1/2}\left(  \Gamma_{j}\right)  }.
\end{align*}
The continuity estimates for the operators $\mathsf{W}_{j}\left(  s\right)  $
and $\mathsf{K}_{j}\left(  s\right)  $ follow from the combination of this and
the trace estimates (Lem. \ref{LemScaledTraces}).
\end{proof}

\subsection*{List of notations}
{In this article we prefer ``verbose'' notations conveying maximum information about
entities. We admit that this leads to lavishly adorned symbols, but enhanced precision is
worth this price.}

{As a convention, we denote scalar functions and spaces of scalar functions with italic
letters, vectors in $\mathbb{C}^{3}$ (tensors of order $1$) with bold letters, and
matrices in $\mathbb{C}^{3\times3}$ (tensors of order $2$) by
blackboard bold letters.}                                                                                                                         \\
{1c\small%
\begin{longtable}{p{0.26\textwidth}p{0.73\textwidth}}
  $\mathbb{R}_{>0}$\dotfill                                              & positive real numbers                                                 \\
  $\mathbb{C}_{>0}$\dotfill                                              & complex numbers with positive real part                               \\
  $\mathbb{R}_{\operatorname*{sym}}^{3\times3}$ \dotfill                 & symmetric $3\times3$ matrices                                         \\
  \modified{$\left\langle \cdot,\cdot\right\rangle,
  \left\langle \cdot,\cdot\right\rangle _{\omega}$} \dotfill             & bilinear form in $\mathbb{C}^{3}$ {see \S \ref{SubSecFuSp}} and
  duality pairing of a function space on a domain (or manifold) $\omega$ with its dual                                                           \\
$\mathbb{A}$\dotfill                                                     & tensor coefficient for transmission problem, see Rem.~\ref{RemAgiven} \\
  \modified{$\mathbb{A}_{j}^{\sigma},
    p_{j}^{\sigma}$}\dotfill                                             & coefficients on the subdomain $\Omega_{j}^{\sigma}$,
  $\sigma\in\left\{ +,-\right\} $, see \cref{Acoeff}, (\ref{defajpjpm})                                                                \\
  \modified{$\mathbb{A}_{j}^{\operatorname*{ext}},
    p_{j}^{\operatorname*{ext}}$} \dotfill                               & extension of the coefficients $\mathbb{A}_{j}^{\operatorname*{ext}},
  p_{j}^{-}$ to $\mathbb{R}^{3}$, see \cref{Acoeff}                                                                                              \\
  $\lambda\left( \mathbb{A}_{j}^{\operatorname*{ext}}\right) $,
  $\Lambda \left( \mathbb{A}_{j}^{\operatorname*{ext}}\right) $ \dotfill & lower and upper spectral
  bound of the tensor coefficient
  $\mathbb{A}_{j}^{\operatorname*{ext}}$, see (\ref{spectralbounds})                                                                             \\
  $\lambda\left( p_{j}^{\operatorname*{ext}}\right) $,
  $\Lambda\left(p_{j}^{\operatorname*{ext}}\right) $ \dotfill            & lower and upper bound of the
  coefficient $p_{j}^{\operatorname*{ext}}$, see (\ref{spectralbounds})                                                                          \\
  $\lambda_{j}$, $\Lambda_{j}$ \dotfill                                  & 
  $\min\left\{ \lambda\left( \mathbb{A}_{j}^{\operatorname*{ext}}\right) ,\lambda\left(
      p_{j}^{\operatorname*{ext}}\right) \right\} $,
  $\max\left\{ \Lambda\left( \mathbb{A}_{j}^{\operatorname*{ext}}\right) ,\Lambda\left(
      p_{j}^{\operatorname*{ext}}\right)  \right\}  $, see (\ref{deflambdaj})                                                                                                                                                                           \\
  $s$ \dotfill                                                                                                                                    & Laplace domain parameter (``wave number'') in $\mathbb{C}_{>0}$, see (\ref{Acoeff})                 \\
  $s_{0}$ \dotfill                                                                                                                                & lower bound of the modulus of $s$, see (\ref{Acoeff})                                               \\
  $\Omega$ \dotfill                                                                                                                               & bounded or unbounded domain in $\mathbb{R}^{3}$, see \S \ref{SubSecDiffOp}                          \\
  $\Omega_{j}=\Omega_{j}^{-}$, \dotfill                                                                                                           & subdomains of $\Omega$ ($1\leq j\leq n_{\Omega}$), see \S \ref{SubSecDiffOp}                        \\
  $\Omega_{j}^{+}$ \dotfill                                                                                                                       & exterior complement $\mathbb{R}^{3}\backslash\overline {\Omega_{j}^{-}}$, see \S \ref{SubSecDiffOp} \\
  $\omega\subset\subset\Omega$ \dotfill                                                                                                           & $\omega$ is compactly contained in $\Omega$, i.e., $\overline{\omega}\subset\Omega$,                \\
  $\Gamma$ \dotfill                                                                                                                               & boundary of $\Omega$; see \S \ref{SubSecDiffOp}                                                     \\
  $\Gamma_{j}$ \dotfill                                                                                                                           & boundary of $\Omega_{j}$; see \S \ref{SubSecDiffOp}                                                 \\
  $\Gamma_{j,k}$\dotfill                                                                                                                          & common boundary of $\Omega_{j}$ and $\Omega_{k}$; see \S \ref{SubSecDiffOp}                         \\
  $\Gamma_{\operatorname*{D}}$\dotfill                                                                                                            & part of $\Gamma$ where Dirichlet boundary conditions are imposed; see \S \ref{SubSecDiffOp}         \\
  $\Gamma_{\operatorname*{N}}$\dotfill                                                                                                            & part of $\Gamma$ where Neumann boundary conditions are imposed; see \S \ \ref{SubSecDiffOp}         \\
  $\mathcal{P}_{\Omega}$\dotfill                                                                                                                  & set of subdomains of $\Omega$; see \S \ref{SubSecDiffOp}                                            \\
  $\Sigma$\dotfill                                                                                                                                & skeleton of $\mathcal{P}_{\Omega}$, union of $\partial\Omega_{j}$, see \S  \ref{SubSecDiffOp}       \\
  $\mathbf{n}_{j}$ \dotfill                                                                                                                       & outward normal vector pointing from $\Omega_{j}^{-}$ to $\Omega_{j}^{+}$, see Prop. \ref{Proptrace} \\
  $C^{\infty}\left( \omega\right) $,
  $\mathbf{C}^{\infty}\left(\omega\right) $\dotfill                                                                                               & space of  infinitely differentiable functions and vector valued version                             \\
  $C_{0}^{\infty}\left(  \omega\right)  $,
  $\mathbf{C}_{0}^{\infty}\left(\omega\right)  $\dotfill                                                                                          & 
  $C_{0}^{\infty}\left(  \omega\right)  :=\left\{  u\in C^{\infty}\left(  \omega\right)  \mid\operatorname*{supp}u\subset
    \omega
  \right\}  $ with vector valued version $\mathbf{C}_{0}^{\infty}\left(\omega\right)  $                                                                                                                                                                 \\
  $\left(  L^{p}\left(  \omega\right),
    \left\Vert \cdot\right\Vert
    _{L^{p}\left(  \omega\right)  }\right)  $\dotfill                                                                                             & Lebesgue space for $1\leq
p\leq\infty$ with norm $\left\Vert \cdot\right\Vert _{L^{p}\left(
\omega\right)  }$; see \S \ref{SubSecFuSp}                                                                                                                                                                                                              \\
$\left(  \mathbf{L}^{p}(\omega)  ,\norm{\cdot}{\mathbf{L}^{p}(\omega)}\right) $\dotfill                                                                                        & $\mathbf{L}^{p}\left(
\omega\right)  :=L^{p}\left(  \omega\right)  ^{3}$ with norm $\left\Vert
\cdot\right\Vert _{\mathbf{L}^{p}\left(  \omega\right)  }$, see
\S \ref{SubSecFuSp}                                                                                                                                                                                                                                     \\ 
$\left(  \mathbb{L}^{p}\left(  \omega\right)  ,\left\Vert \cdot\right\Vert
_{\mathbb{L}^{p}\left(  \omega\right)  }\right)  $\dotfill                                                                                        & $\mathbb{L}^{p}\left(
\omega\right)  :=L^{p}\left(  \omega\right)  ^{3\times3}$ with norm
$\left\Vert \cdot\right\Vert _{\mathbb{L}^{p}\left(  \omega\right)  }$, see
\ref{SubSecFuSp}                                                                                   \\
\modified{$(\cdot,\cdot)_{L^{2}\left(  \omega\right)  }$,
$\left(\cdot,\cdot\right)  _{\mathbf{L}^{2}\left(  \omega\right)  }$},&\\
\modified{$\left(\cdot,\cdot\right)  _{\mathbb{L}^{2}\left(  \omega\right)  }$} \dotfill                      & $L^{2}\left(\omega\right)  $ scalar product in $L^{2}\left(  \Omega\right)  $,$\mathbf{L}^{2}\left(  \Omega\right)  $, $\mathbb{L}^{2}\left(  \Omega\right)
$                                                                                                                                                                    \\
$L_{>0}^{\infty}\left(  \omega,\mathbb{R}\right)  $\dotfill                                 & subset of $L^{\infty}\left(
  \omega\right)  $ of functions which are uniformly positive, see \S \ref{SubSecFuSp}                                                                                \\
$\mathbb{L}^{p}\left(  \omega,\mathbb{R}_{\operatorname*{sym}}^{3\times3}\right)  $\dotfill & subset
of $\mathbb{L}^{p}\left(  \omega\right)  $ of functions which map into the set
of symmetric $3\times3$ matrices; see \S \ref{SubSecFuSp}                                                                                                            \\
$\mathbb{L}_{>0}^{\infty}\left(  \omega,\mathbb{R}_{\operatorname*{sym}%
}^{3\times3}\right)  $\dotfill                                                              & subset of $\mathbb{L}^{\infty}\left(  \omega
,\mathbb{R}_{\operatorname*{sym}}^{3\times3}\right)  $ of functions which are
uniformly positive definite, see Def. \ref{DefH1omegaB}                                                                                                              \\
$H^{k}\left(  \omega\right)  $\dotfill                                                      & Sobolev space $W^{k,2}\left(  \omega\right) $, see \S \ref{SubSecFuSp} \\
$H_{0}^{k}(\omega),H^{-k}(\omega)$\dotfill                                                  & closure of smooth functions with compact support with respect to the $\norm{\cdot}{H^{k}(\omega)}$ norm (see \S \ref{SubSecFuSp})
and its dual space (see \S \ref{SubSecFuSp})                                                                                                                         \\
$H_{\operatorname*{loc}}^{k}\left(  \omega\right)  $\dotfill                                & Sobolev space of functions which
locally belong to $H^{k}\left(  \omega\right)  $; see \S \ref{SubSecFuSp}                                                                                            \\
$\left\Vert \cdot\right\Vert _{H^{1}\left(  \omega\right)  ;s}$,
$\left\Vert \cdot\right\Vert _{H^{-1}\left(  \mathbb{R}^{3}\right)  ;s}$\dotfill            & 
frequency-weighted Sobolev norm and its dual norm, see (\ref{fs_norm}),
(\ref{defdualnorm})                                                                                                                                                  \\
$\mathbf{H}\left(  \omega,\operatorname*{div}\right)  $ \dotfill                            & subspace of $\mathbf{L}^{2}\left(
  \omega\right)  $ of functions $\mathbf{v}$ satisfying $\operatorname*{div}\mathbf{v}\in
L^{2}\left(  \omega\right)  $, see (\ref{defHomegadiv})                                                                                                              \\
$\left(  H^{1}\left(  \omega,\mathbb{B}\right) ,\left\Vert \cdot\right\Vert_{H^{1}\left(  \omega,\mathbb{B}\right)  }\right)  $\dotfill                                & subspace of $H^{1}\left(
\omega\right)  $ of functions $v$ such that $\operatorname*{div} \left(  \mathbb{B}\nabla
v\right)  \in L^{2}\left(  \omega\right)  $ equipped with the graph norm; see Def. \ref{DefH1omegaB}                                                                 \\
$\mathbb{H}^{1}\left(  \mathcal{P}_{\Omega},\mathbb{A}\right)  $\dotfill                    & 
$\BIGOP{\times}_{j=1}^{n_{\Omega}}H^{1}\left(  \Omega_{j},\mathbb{A}_{j}^{-}\right)  $                                                                               \\
$H^{\alpha}\left(  \partial\omega\right)  $\dotfill                                         & Sobolev space on a closed manifold, see \S \ref{SubSecFuSp}            \\
$H^{\pm1/2}\left(  \Gamma_{j,k}\right)  $, $\tilde{H}^{\pm1/2}\left(\Gamma_{j,k}\right)  $
\dotfill                                                                                    & Sobolev spaces on manifolds with boundary; see
(\ref{deftildespaces})                                                                                                                                               \\
$\left(  \mathbf{X}_{j},\left\langle \cdot,\cdot\right\rangle _{\mathbf{X}%
_{j}},\left\Vert \cdot\right\Vert _{\mathbf{X}_{j}}\right)  $\dotfill                       & Sobolev space
$H^{1/2}\left(  \Gamma_{j}\right)  \times H^{-1/2}\left(  \Gamma_{j}\right)
$, equipped with bilinear form $\left\langle \cdot,\cdot\right\rangle
_{\mathbf{X}_{j}}$ and norm $\left\Vert \cdot\right\Vert _{\mathbf{X}_{j}}$,
see Def. \ref{DefMT}, (\ref{iXj}) \\
$\left(  \mathbb{X}\left(  \mathcal{P}_{\Omega}\right)  ,\left\langle
\cdot,\cdot\right\rangle _{\mathbb{X}},\left\Vert \cdot\right\Vert
_{\mathbb{X}}\right)  $\dotfill                                                             & Sobolev space $\mathbb{X}\left(  \mathcal{P}_{\Omega}\right)  :=\BIGOP{\times}_{j=1}^{n_{\Omega}}\mathbf{X}_{j}$, with bilinear form $\left\langle\cdot,\cdot\right\rangle _{\mathbb{X}}$ and norm $\left\Vert \cdot\right\Vert
_{\mathbb{X}}$, see Def. \ref{DefMT}, (\ref{bilixfull})                                                                             \\
$\mathbb{X}^{\operatorname{single}}\left(  \mathcal{P}_{\Omega}\right)  $\dotfill           & single traces space; see (\ref{defXsingle}) \\
$\mathbb{X}_{0}^{\operatorname{single}}\left(  \mathcal{P}_{\Omega}\right)  $\dotfill       & single
traces space with incorporated zero boundary conditions, see (\ref{defXsingle})                                                     \\
$\gamma_{\operatorname*{D};j}^{\sigma}$, $\gamma_{\operatorname*{D};j}$, $\gamma_{\operatorname*{D};j}^{\sigma}\left(  s\right)  $, $\gamma
_{\operatorname*{D};j}\left(  s\right)  $\dotfill                                           & one-sided and two-sided Dirichlet trace
operators and frequency scaled versions; see Prop. \ref{Proptrace}, (\ref{gammadjs})                                                \\
$\gamma_{\mathbf{n};j}^{\sigma}$, $\gamma_{\mathbf{n};j}$, $\gamma
_{\mathbf{n};j}^{\sigma}\left(  s\right)  $, $\gamma_{\mathbf{n};j}\left(
s\right)  $\dotfill                                                                         & one-sided and two-sided normal trace operators and frequency
scaled versions; see Prop. \ref{Proptrace}, (\ref{gammadjs})                                                                        \\
$\gamma_{\operatorname*{N};j}^{\sigma}$,
$\gamma_{\operatorname*{N};j}^{\operatorname*{ext},\sigma}$,
$\gamma_{\operatorname*{N};j}$, $\gamma_{\operatorname*{N};j}^{\operatorname*{ext}}$, $\gamma
_{\operatorname*{N};j}^{\sigma}\left(  s\right)  $, $\gamma_{\operatorname*{N}%
;j}^{\operatorname*{ext},\sigma}\left(  s\right)  $, $\gamma
_{\operatorname*{N};j}\left(  s\right)  $, $\gamma_{\operatorname*{N}%
;j}^{\operatorname*{ext}}\left(  s\right)  $\dotfill                                  & one-sided and two-sided co-normal
derivatives and frequency scaled versions, see Prop. \ref{Proptrace}%
, (\ref{gammadjs}), Notation \ref{NotBAext} \\
$\boldsymbol{\gamma}_{\operatorname*{C};j}^{\sigma}$,
$\boldsymbol{\gamma}_{\operatorname*{C};j}^{\operatorname*{ext},\sigma}$,
$\boldsymbol{\gamma}_{\operatorname*{C};j}^{\sigma}\left(  s\right)  $,
$\boldsymbol{\gamma}_{\operatorname*{C};j}^{\operatorname*{ext},\sigma}\left(  s
\right)  $\dotfill                                                                    & 
one-sided and two-sided Cauchy trace operators and frequency scaled versions, see
(\ref{defboldgammaform}), (\ref{gammadjs}), Notation \ref{NotBAext}                                                                                                                                            \\
$\mathsf{E}_{j}\left(  s\right)  $\dotfill                                            & trace lifting operator; see Lem.~\ref{LemScaledTraces}                                                                 \\
$\jump{u}{D}{;j}, \jump{u}{D}{;j}\left(  s\right)  $\dotfill                          & Dirichlet jump across $\Gamma_{j}$
and frequency scaled version; see Def. \ref{DefJumpmean}                                                                                                                                                       \\
$\left[  \mathbf{u}\right]  _{\operatorname*{D};j,k}$\dotfill                         & Dirichlet jump across partial boundary $\Gamma_{j,k}$; see (\ref{partjump})                                            \\
\modified{$\jump{u}{N}{;j},
\jump{u}{N}{;j}^{\operatorname*{ext}}$},
\modified{$\jump{u}{N}{;j}(s),
  \jump{u}{N}{;j}^{\operatorname*{ext}}(s)$\dotfill}                                 & jump of co-normal derivative across $\Gamma_{j}$ and frequency scaled version; see \cref{DefJumpmean}, \cref{NotBAext} \\
$\left[  \mathbf{w}\right]  _{\operatorname*{N};j,k}$, $\left[  \mathbf{w}\right]
_{\operatorname*{N};j,k}^{\operatorname*{ext}}$ \dotfill                              & 
jump of co-normal derivative across partial boundary $\Gamma_{j,k}$, see (\ref{partjump})                                                                                                                      \\
\modified{$\mean{u}{D}{;j}, \mean{u}{D}{;j}(s)$},
\modified{$\mean{u}{N}{;j}, \mean{u}{N}{;j}(s)$}\dotfill                              & mean value of Dirichlet traces and co-normal derivatives across boundary
$\Gamma_{j}$, and their frequency scaled version; see \cref{DefJumpmean}, \cref{NotBAext}                                                                                                                      \\
\modified{$\ell_{j}\left(  s\right)  \left(  \cdot,\cdot\right),
\mathsf{L}_{j}\left(  s\right)$}\dotfill                                              & sesquilinear form
associated to the full space transmission problem with coefficients $\mathbb{A}_{j}^{\operatorname*{ext}}$, $p_{j}^{\operatorname*{ext}}$ and relative operator; see \cref{Defellsj}                           \\
$\mathsf{L}_{j}^{-}\left(  s\right)  $, $\mathsf{L}_{j}^{+}\left(  s\right)$ \dotfill & 
differential operator on subdomains $\Omega_{j}^{-}$, $\Omega_{j}^{+}$, see (\ref{defDiffOpsigma})                                                                                                             \\
$\nabla_{\operatorname*{pw};j}$\dotfill                                               & piecewise gradient; see (\ref{defpwgradient})                                                                          \\

\modified{$\mathsf{V}_{j}\left(  s\right),
\mathsf{K}_{j}\left(  s\right)$, }         &                                                                                     \\
\modified{$\mathsf{K}_{j}^{\prime}\left(  s\right),
  \mathsf{W}_{j}(s)$}\dotfill         & boundary integral operators, see \cref{DefBIO}                                      \\
$\mathsf{C}_{j}\left(  s\right)  $\dotfill & Calder\'{o}n operator for the subdomain $\Omega_{j}$, see Def. \ref{DefCaldOp}      \\
\modified{$\mathsf{C}\left(  s\right),
c\left(  s\right)  $}\dotfill              & global Calder\'{o}n operator and associated sesquilinear form; see \cref{DefCaldOp} \\
$\operatorname*{Id}$\dotfill               & identity operator                                                              
\end{longtable}}

\newcommand{\noopsort}[1]{} \newcommand{\printfirst}[2]{#1}
  \newcommand{\singleletter}[1]{#1} \newcommand{\switchargs}[2]{#2#1}
  \def\cprime{$'$} \def\cprime{$'$} \def\cprime{$'$}

\end{document}